\documentclass[11pt]{amsart}

\usepackage{epigamath-voisin}

\usepackage[english]{babel}

\numberwithin{equation}{section}

\usepackage[shortlabels]{enumitem}

\usepackage[english]{babel}
\usepackage{amssymb,amsmath,amsthm,mathrsfs,graphicx,color,bm}
\usepackage[cmtip,all,matrix,arrow,tips,curve]{xy}

\usepackage{tikz-cd}
\usepackage{manfnt}
\usepackage{centernot}
\usepackage{multirow}
\usepackage{macro-longdash}
\usepackage{mathtools}
\usepackage{amsbsy}

\theoremstyle{plain}

\newtheorem{thm}{Theorem}[section]
\newtheorem*{thm*}{Theorem}

\newtheorem{clm}[thm]{Claim}

\newtheorem{crl}[thm]{Corollary}
\newtheorem*{hyp*}{Hypothesis}
\newtheorem{lmm}[thm]{Lemma}
\newtheorem{prp}[thm]{Proposition}
\newtheorem{prp-dfn}[thm]{Proposition-Definition}

\theoremstyle{definition}

\newtheorem{dfn}[thm]{Definition}

\newtheorem{hyp-dfn}[thm]{Hypothesis-Definition}

\theoremstyle{remark}

\newtheorem*{qst*}{Main Question}
\newtheorem{rmk}[thm]{Remark}

\newcommand{\bV}{{\bf V}}

\newcommand{\CC}{{\mathbb C}}

\newcommand{\cC}{{\mathscr C}}

\newcommand{\cE}{{\mathscr E}}
\newcommand{\cF}{{\mathscr F}}
\newcommand{\cG}{{\mathscr G}}
\newcommand{\cH}{{\mathscr H}}
\newcommand{\cI}{{\mathscr I}}

\newcommand{\cK}{{\mathscr K}}
\newcommand{\cL}{{\mathscr L}}
\newcommand{\cM}{{\mathscr M}}
\newcommand{\cN}{{\mathscr N}}

\newcommand{\cO}{{\mathscr O}}
\newcommand{\cP}{{\mathscr P}}  
\newcommand{\cQ}{{\mathscr Q}}
\newcommand{\cR}{{\mathscr R}}
\newcommand{\cS}{{\mathscr S}}
\newcommand{\cT}{{\mathscr T}}
\newcommand{\cU}{{\mathscr U}}
\newcommand{\cV}{{\mathscr V}} 
\newcommand{\cW}{{\mathscr W}}
\newcommand{\cX}{{\mathscr X}} 

\newcommand{\cZ}{{\mathscr Z}}

\newcommand{\dra}{\dashrightarrow}

\newcommand{\Ext}{{\rm Ext}}

\newcommand{\gh}{\mathfrak{h}}
\newcommand{\gm}{\mathfrak{m}}

\newcommand{\gS}{\mathfrak{S}}

\newcommand{\hra}{\hookrightarrow}

\newcommand{\Kum}{{\mathsf K}{\mathsf u}{\mathsf m}}
\newcommand{\la}{\langle}

\newcommand{\lra}{\longrightarrow}

\newcommand{\ov}{\overline}
\newcommand{\PP}{{\mathbb P}}
\newcommand{\QQ}{{\mathbb Q}}
\newcommand{\ra}{\rangle}
\newcommand{\RR}{{\mathbb R}}

\newcommand{\sF}{{\mathsf F}}

\newcommand{\wh}{\widehat}
\newcommand{\wt}{\widetilde}
\newcommand{\ZZ}{{\mathbb Z}}
\newcommand{\Gr}{\mathrm{Gr}}
\DeclareMathOperator{\Aut}{Aut}

\DeclareMathOperator{\ch}{ch}
\DeclareMathOperator{\CH}{CH}

\DeclareMathOperator{\cl}{cl}

\DeclareMathOperator{\Coh}{Coh}

\DeclareMathOperator{\Def}{Def}

\DeclareMathOperator{\Disc}{Disc}
\DeclareMathOperator{\disc}{disc}
\DeclareMathOperator{\divisore}{div}

\DeclareMathOperator{\End}{End}

\DeclareMathOperator{\Hom}{Hom}
\DeclareMathOperator{\Id}{Id}
\DeclareMathOperator{\im}{Im}

\DeclareMathOperator{\NS}{NS}

\DeclareMathOperator{\Pic}{Pic}

\DeclareMathOperator{\pr}{pr}

\DeclareMathOperator{\supp}{supp}
\DeclareMathOperator{\Sym}{Sym}
\DeclareMathOperator{\td}{Td}

\DeclareMathOperator{\Tr}{Tr}

\DeclareMathOperator{\U}{U}

\newcommand{\ldra}{\longdashrightarrow}

\newcommand{\lhra}{\lhook\joinrel\!\longrightarrow}

\def\lowsim{\vbox to 0pt{\vss\hbox{$\scriptstyle\sim$}\vskip-2pt}}

\setlist[enumerate,1]{label={\rm(\arabic*)}, ref={\rm\arabic*}} 

\newlist{a-enumerate}{enumerate}{2}
\setlist[a-enumerate,1]{label={\rm(\alph*)}, ref={\rm\alph*}}

\newcommand{\supth}[1]{\ensuremath{#1^{\mathrm{th}}}}

\YearArticle{2024} \EpigaArticleNr{19} \ReceivedOn{January 26, 2023}
\InFinalFormOn{August 13, 2024}
\AcceptedOn{September 11, 2024}

\title{Rigid stable rank 4   vector bundles on\\ HK fourfolds of Kummer type}
\titlemark{Rigid stable rank 4   vector bundles on HK fourfolds of Kummer type}

\author{Kieran G.~O'Grady}
\address{Dipartimento di Matematica, 
 Sapienza Universit\`a di Roma,
 P.le A.~Moro 5,
 00185 Roma,  Italy}
 \email{ogrady@mat.uniroma1.it}

\authormark{K.\,G.~O'Grady}

\AbstractInEnglish{We partially extend to hyperk\"ahler fourfolds of Kummer type the results on hyperk\"ahler (HK) varieties of type $K3^{[2]}$  by the author in a 2022 paper (and later extended to HK varieties of type $K3^{[n]}$).
  Let $(M,h)$ be a general polarized HK fourfold of Kummer type such that $q_M(h)\equiv -6\pmod{16}$ and $\divisore(h)=2$, or $q_M(h)\equiv -6\pmod{144}$ and $\divisore(h)=6$. We show that there exists a unique (up to isomorphism) slope stable vector bundle $\cF$ on $M$ such that $r(\cF)=4$, $ c_1(\cF)= h$, $\Delta(\cF)=c_2(M)$. Moreover, $\cF$ is rigid.  One of our motivations is the desire to describe explicitly a locally complete family of polarized HK fourfolds of Kummer type.}

\MSCclass{14J42 (primary); 14J60, 14F08 (secondary)}

\KeyWords{Hyperk\"ahler manifolds, stable vector bundles}

\acknowledgement{Partially supported by PRIN 2017YRA3LK}

\dedication{Dedicato a Claire con ammirazione e amicizia}

\begin{document}


\maketitle

\begin{prelims}

\DisplayAbstractInEnglish

\bigskip

\DisplayKeyWords

\medskip

\DisplayMSCclass

\end{prelims}


\newpage

\setcounter{tocdepth}{1}

\tableofcontents


\section{Introduction}\label{sec:intro}
\subsection{Background and motivation}
Slope stable rigid vector bundles on polarized $K3$ surfaces exist in abundance (the single necessary condition on the Chern character is the one coming from the equality $\chi(S,F^{\vee}\otimes F)=2$, valid for a simple rigid vector bundle $F$ on a $K3$ surface), and they are determined up to isomorphism by their Chern character. Recently there has been substantial progress in understanding slope stable rigid vector bundles on polarized hyperk\"ahler (HK) varieties of type $K3^{[n]}$; see~\cite{ogfascimod,og-rigidi-su-k3n}. In the present paper, we deal with the analogous problem for HK fourfolds of Kummer type.  We prove the existence and uniqueness of slope stable vector bundles $\cF$ with $r(\cF)=4$, $ c_1(\cF)= h$, $\Delta(\cF)=c_2(M)$ on a general polarized HK fourfold $(M,h)$ of Kummer type such that $q_M(h)\equiv -6\pmod{16}$ and $\divisore(h)=2$, or $q_M(h)\equiv -6\pmod{144}$ and $\divisore(h)=6$, and moreover that such bundles are rigid. Here $q_M$ is the Beauville--Bogomolov--Fujiki (BBF for short) quadratic form of $M$, and $\divisore(h)$ is the divisibility of $h$, \textit{i.e.}, the positive generator of the ideal in $\ZZ$ given by $q_M(h,H^2(X;\ZZ))$

The results are limited to rank $4$, but one should be able to obtain analogous results for vector bundles on HK fourfolds of Kummer type for all ranks of the form $r_0^2/g$, where $g:=\gcd\{r_0,3\}$, and also for vector bundles on HK varieties of Kummer type of arbitrary dimension. We believe that the rank $4$ case is already quite interesting; in particular, it might lead to an explicit description of a locally complete family of polarized HK fourfolds of Kummer type.

The link between rigid vector bundles and locally complete families is provided by an analogy with \lq\lq Mukai models\rq\rq of $K3$ surfaces. We recall that a general polarized $K3$ surface of degree at least $10$ is not a complete intersection in a projective space, but for certain degrees it is a (generalized) complete intersection in a Grassmannian (see for example~\cite{muk-k3s-less10}), and the latter fact provides an explicit description of locally complete families of $K3$ surfaces. The conceptual reason underlying such a description is the existence of a rigid and unique (up to isomorphism) slope stable vector bundle on a general polarized $K3$ surface of the relevant degree, realized explicitly as the pull-back of the tautological vector bundle of the Grassmannian.

In higher dimensions, explicit locally complete families of polarized HK varieties are known, for example the varieties of lines on a smooth cubic hypersurface in $\PP^5$, see~\cite{beaudon}, double EPW sextics, see~\cite{epwduke}, Debarre--Voisin varieties, see~\cite{devo}, the LLSvS varieties attached to smooth cubic hypersurface in $\PP^5$, see~\cite{llsvs}, the infinite family of moduli spaces of stable objects in the Kuznetsov category attached to smooth cubic hypersurface in $\PP^5$, see~\cite{blmnps}. All of the known locally complete polarized families of HK varieties parametrize varieties of type $K3^{[n]}$.

The explicit families given by the varieties of lines on a smooth cubic hypersurface in $\PP^5$ and the Debarre--Voisin varieties are very similar to the Mukai families.  In fact, the restrictions of the tautological quotient vector bundles are stable (for a general variety of that kind) rank $4$ vector bundles which extend to all deformations of the polarized $(X,L)$. Actually the projectivization of those vector bundles extends to all small deformations of the unpolarized $X$ (this last property distinguishes the tautological quotient bundle from the tautological subbundle); \textit{i.e.}, they are projectively hyperholomorphic.

We expect that a similar construction holds for polarized HK fourfolds $(M,h)$ of Kummer type such that $q_M(h)=10$ and $\divisore(h)=2$.  In fact, one should get an explicit description of the general such $M$ by analyzing the rank $4$ slope stable vector bundle whose existence is guaranteed by the main result of the present paper. Our paper~\cite{og-eqni-kumm} contains other material that should be relevant if one wants to describe such a family.
\subsection{Main result}\label{subsec:primorisul}
Before stating our main theorem, we recall a few results on $4$-dimensional polarized HK varieties $(M,h)$ of Kummer type (the ample class $h\in \NS(M)$ is primitive). Let $e:=q_M(h)$ ($e$ is the \emph{square of $h$}), and let $i:=\divisore(h)$.  Then one of the following holds:
\begin{equation}\label{divuno}
i=1,\quad  e\equiv 0\pmod{2};
\end{equation}
\begin{equation}\label{divdue}
i=2,\quad  e\equiv -6\pmod{8};
\end{equation}
\begin{equation}\label{divtre}
i=3,\quad  e\equiv -6\pmod{18};
\end{equation}
\begin{equation}\label{divsei}
i=6,\quad  e\equiv -6\pmod{72}.
\end{equation}
Conversely, if $e$ is a positive integer, $i\in\{1,2,3,6\}$ and one of~\eqref{divuno}, \eqref{divdue} \eqref{divtre} or \eqref{divsei} holds, then there exists a $4$-dimensional polarized HK variety $(M,h)$ of Kummer type such that $q_M(h)=e$ and $\divisore(h)=i$.  Let $\Kum_{e}^i$ be the (quasi-projective) moduli space of $4$-dimensional polarized HK varieties $(M,h)$ of Kummer type (note that $\Kum_{e}^i$ is $4$-dimensional). The number of irreducible components of $\Kum_{e}^i$ has been computed by Onorati in~\cite{onorati-conn-comp-gen-kumm}. In particular, $\Kum_{e}^2$ (we assume that~\eqref{divdue} holds) and $\Kum_{e}^6$ (here we assume that~\eqref{divsei} holds) are irreducible.
\begin{thm}\label{thm:unicita}
Let $e$ be a positive integer such that $e\equiv -6 \pmod{16}$ or $e\equiv -6 \pmod{144}$.  Let $[(M,h)]$ be a general point of\, $\Kum_{e}^2$ if the former holds, and of\, $\Kum_{e}^6$ if the latter holds.  Then there exists one and only one slope stable vector bundle $\cF$ on $M$ $($up to isomorphism$)$ such that
 \begin{equation}\label{ele}
r(\cF)=4,\quad c_1(\cF)= h,\quad \Delta(\cF):=8 c_2(\cF)-3 c_1(\cF)^2=c_2(M).
\end{equation}
$($Chern classes are in cohomology.$)$  Lastly, $H^1(M,End^0(\cF))=0$.
\end{thm}
\begin{rmk}
Let $(X,h)$ be a general polarized HK variety of type $K3^{[n]}$. In~\cite{ogfascimod,og-rigidi-su-k3n} we have proved that, provided certain numerical hypotheses are satisfied, there exists a unique stable vector bundle $\cF$ on $X$ of rank $r_0^n$ with 
\begin{equation}\label{formuniv}
\Delta(\cF)=\frac{r_0^{2n-2}(r_0^2-1)}{12} c_2(X).
\end{equation}
(Here $\Delta(\cF)$ is the discriminant of $\cF$; see~\eqref{eccodisc}.) The equation in~\eqref{formuniv} for $n=2$ and $r_0=2$ is the exact analogue of the last equation in~\eqref{ele}.
\end{rmk}
\begin{rmk}\label{rmk:chesilstreet}
We recall that if $M$ is a HK manifold, then $\Aut^0(M)$ is the group of automorphisms of $M$ acting trivially on $H^2(M)$.  If $M$ is a $2n$-dimensional HK manifold of Kummer type ($n\ge 2$), then $\Aut^0(M)$ is not trivial; in fact, it is a semidirect product of $\ZZ/(n+1)^4$ and $\ZZ/(2)$; see Section~\ref{subsec:yesmen}. Let $[(M,h)]$ be as in Theorem~\ref{thm:unicita}.  If $\cF$ is the slope stable vector bundle $\cF$ on $M$ such that the equalities in~\eqref{ele} hold and $\varphi\in\Aut^0(M)$, then $\varphi^{*}(\cF)$ is an $h$ slope stable vector bundle on $M$ which has the same rank, $c_1$ and discriminant as $\cF$ because $\varphi^{*}(h)=h$ and $\varphi^{*}(c_2(M))=c_2(M)$. By Theorem~\ref{thm:unicita}, it follows that $\varphi^{*}(\cF)\cong\cF$.
\end{rmk}
\begin{rmk}
Let $\U^i_e\subset \Kum_{e}^i$ for $i\in\{2,6\}$ be an open nonempty subset with the property that there exists one and only one stable vector bundle $\cF$ on $[(M,h)]\in \U^i_e$ such that the equations in~\eqref{ele} hold, and let $\cX\to \U^{i}_e$ be the tautological family of HK (polarized) varieties (here $\Kum_{e}^i$ is to be interpreted as the moduli stack). By Theorem~\ref{thm:unicita} and~\cite[Theorem A.5]{mukvb}, there exists a quasi-tautological vector bundle $\mathsf F$ on $\cX$, \textit{i.e.}, a vector bundle whose restriction to a fiber $M$ of $\cX\to \U^{i}_e$ is isomorphic to $\cF^{\oplus d}$ for some $d>0$, where $\cF$ is the vector bundle of Theorem~\ref{thm:unicita}. If $[(M,h)]\in \U^{i}_e$, then the generalized Franchetta conjecture, see~\cite{fulatvial:genfranchetta1}, predicts that the restriction to $\CH^2(M)_{\QQ}$ of $\ch_2({\mathsf F})\in\CH(\cX)_{\QQ}$ is equal to $-d\frac{r_0^{2n-2}(r_0^2-1)}{12}c_2(X)$. In other words, it predicts that the third equality in~\eqref{ele} holds at the level of (rational) Chow groups.  In general, it is not easy to give a rationally defined algebraic cycle class on a nonempty open subset of the moduli stack of polarized HK varieties. Theorem~\ref{thm:unicita} produces such a cycle, and hence a good test for the generalized Franchetta conjecture.

\end{rmk}
\subsection{Outline of the proof}
Our first task is to construct modular vector bundles on generalized Kummer fourfolds which are slope stable. Recall that a torsion-free sheaf $\cF$ on a HK manifold $M$ of dimension $2n$ is modular if there exists a $d(\cF)\in\QQ$ such that for all $\alpha\in H^2(M)$
\begin{equation}\label{fernand}
\int_M \Delta(\cF)  \alpha^{2n-2}=d(\cF) (2n-3)!! q_M(\alpha)^{n-1},
\end{equation}
where the double factorial $(2n-3)!!\coloneq (2n-3)(2n-5)\cdot\ldots\cdot 3\cdot 1$ is the product of natural numbers up to $2n-3$ and of the same parity as $2n-3$ and
\begin{equation}\label{eccodisc}
 \Delta(\cF) :=2r c_2(\cF)-(r-1) c_1(\cF)^2
\end{equation}
is the \emph{discriminant} of $\cF$.  Variation of stability for modular sheaves behaves as if the base were a surface, see~\cite{ogfascimod}, and slope (semi)stability of a modular sheaf on a Lagrangian fibration is related to slope (semi)stability of its restriction to a general Lagrangian fiber, provided the polarization is close to the pull-back of an ample class on the base.  In~\cite{ogfascimod,og-rigidi-su-k3n} we constructed stable modular vector bundles on $S^{[n]}$, where $S$ is a $K3$ surface, by associating to a vector bundle $\cE$ on $S$ a vector bundle $\cE[n]^{\pm}$ on $S^{[n]}$. The following two key facts hold:
\begin{enumerate}
\item
If $\cE$ is spherical, then $\cE[n]^{\pm}$ is modular by the Bridgeland--King--Reid derived version of the McKay correspondence.
\item
If $S$ is elliptic and hence $S^{[n]}$ has a Lagrangian fibration, then the restriction of $\cE[n]^{\pm}$ to a general Lagrangian fiber is slope stable.
\end{enumerate}
In the present paper, we proceed as follows.  Let $f\colon B\to A$ be an isogeny of abelian surfaces, and let $\rho\colon K_n(B)\dra K_n(A)$ be the map between generalized Kummers which associates to $[Z]\in K_n(B)$ the point $[f(Z)]\in K_n(A)$ (as soon as $\deg f\ge 2$ and $n\ge 2$, the map $\rho$ is not regular).  Let $\wt{\rho}\colon X\to K_n(A)$ be a resolution of indeterminacies of $\rho$. If $\cL$ is a line bundle on $X$, let $\cE(\cL):=\wt{\rho}_{*}(\cL)$. Then $\cE(\cL)$ is a torsion-free sheaf on $K_n(A)$ of rank $(\deg f)^n$. Section~\ref{sec:conteduca} is devoted to the study of $\cE(\cL)$ for $n=2$ and $\deg f=2$. We determine exactly under which hypotheses the rank $4$ sheaf $\cE(\cL)$ is modular. In particular, we show that if $\cL$ is the pull-back of a line bundle on $K_2(B)$, then $\cE(\cL)$ is locally free and $\Delta(\cE(\cL))=c_2(K_2(A))$; in particular, it is modular. These are the vector bundles which are studied in the remainder of this paper.

The referee realized that the vector bundle $\cE(\cL)$ corresponds to a semi-homogeneous $\cS_{3}$-equivariant vector bundle on the kernel $N_A(3)$ of the summation map $A^3\to A$ via the Bridgeland--King--Reid equivalence between $D_{\cS_{3}}^b(N_A(3))$ and $D^b(K_2(A)$.
This gives as bonus the vanishing, under suitable hypotheses, of all cohomology of the traceless endomorphism bundle $\cE nd^0\cE(\cL)$.  
 In fact, the referee has given a more general construction of vector bundles on $K_n(A)$, which can be proved to be modular without explicit computations (of course one does not get a formula for the discriminant). These vector bundle deserve to be studied in detail.

In order to prove Theorem~\ref{thm:unicita},  we consider $A$ equipped with   an elliptic fibration $A\to E$. Thus we have a Lagrangian fibration   $\pi_A\colon K_2(A)\to|\cO_E(3(0_E))|$.

The main technical results that we need are about   the restrictions of $\cE(\cL)$ to the Lagrangian fibers of $\pi_A$.
 The first main  result is that the restriction  to a Lagrangian fiber is simple, except possibly   for a finite set of Lagrangian fibers. 
  The second main  result deals with slope stability 
of   the restriction of $\cE(\cL)$ to a Lagrangian fiber (assume that the restriction of $\det\cE(\cL)$ is ample, and consider stability with respect to the restriction). First the restriction  is slope stable  if the Lagrangian fiber is smooth; in fact, $\cE(\cL)$ has this property by construction (thanks to results about semi-homogeneous vector bundles on abelian varieties). This is already enough to prove the existence half of Theorem~\ref{thm:unicita}, but the uniqueness half needs more work.
We do not know whether the restriction of $\cE(\cL)$  to a general singular Lagrangian fiber is slope stable, but we prove that the restriction to a general singular Lagrangian fiber  is not slope destabilized by a subsheaf with integer rank (the rank of sheaves on singular Lagrangian fibers is not necessarily an integer because such fibers are nonreduced and not irreducible). This implies that the extension of $\cE(\cL)$ to a general deformation of $(K_2(A),\det\cE(\cL),\pi_A)$  restricts to a slope stable vector bundle on Lagrangian fibers (of the fibration extending $\pi_A$),  except possibly   for a finite set of Lagrangian fibers. With these results under the belt, one can prove the unicity 
half of Theorem~\ref{thm:unicita}.

\subsection{Organization of the paper}
In Section~\ref{sec:foreplay}, we collect  preliminary results on modular sheaves and HK manifolds of Kummer type. 

Section~\ref{sec:conteduca} contains mainly computations which allow us to determine for which choices of line bundle 
$\cL$ the rank $4$ sheaf $\cE(\cL)$ is modular (and if  this is the  case, to compute its discriminant) and to show that it is locally free  in the cases that are examined in the remainder of the paper.

In Section~\ref{sec:caratteulero}, we compute the Euler characteristic of $End(\cE(\cL))^0$.

Section~\ref{sec:datodabkr} contains the construction by the referee of many modular vector bundles on $K_n(A)$ with vanishing cohomology of the bundle of traceless endomorphismsm and the identification of $\cE(\cL)$ with one of his vector bundles. 

In Section~\ref{sec:analisifine}, we assume that $A$ is equipped with   an elliptic fibration $A\to E$, and hence  we have a Lagrangian fibration   
$\pi_A\colon K_2(A)\to|\cO_E(3(0_E))|$. We prove that the restriction of $\cE(\cL)$  to a smooth Lagrangian fiber is   slope stable and 
 that   the restriction to  Lagrangian fibers is simple, except possibly for a finite set of them.
 
 In Section~\ref{sec:nazariosauro}, we prove the results explained above about subsheaves with integer rank of 
 the restrictions of  $\cE(\cL)$ to general singular Lagrangian fibers.

The last section, \textit{i.e.}, Section~\ref{sec:dimprinc}, wraps everything up to give the proof of Theorem~\ref{thm:unicita}. 
 
 The two appendices contain technical results on semi-homogeneous vector bundles on abelian varieties.
\subsection{Conventions}
 \begin{enumerate}
\item[$\bullet$] 
Algebraic variety is synonymous with complex quasi-projective scheme, and sheaf is synonymous with coherent sheaf on an algebraic variety. Occasionally, we do not differentiate between divisor classes, line bundles and invertible sheaves on a smooth variety.
\item[$\bullet$] 
Chern classes of a  sheaf  on a smooth complex quasi-projective variety $X$ are elements of the Betti cohomology ring $H(X;\ZZ)$ unless we state the contrary. We let $H(X)=H(X;\CC)$ be the complex cohomology of $X$.
\item[$\bullet$] 
Let $\cF$ be a sheaf of positive rank on a smooth projective variety $X$. If $L$ is an ample line bundle, then $\mu_L(\cF)$ is the slope of $\cF$ with respect to $L$, and similarly $\mu_h(\cF)$ is the slope of $\cF$ with respect to the numerical equivalence class of an ample divisor on $X$. Hopefully there will be no confusion with Donaldson's map $\mu$ discussed in Section~\ref{subsec:genkumm}.
\item[$\bullet$] 
 A \emph{polarized variety} is a couple $(X,h)$, where $X$ is a projective scheme and $h$ is the numerical equivalence class of an ample divisor on $X$.  If $X$ is smooth of positive rank,  
 slope (semi)stability of $\cF$ refers to the slope function defined by $h$.   
\item[$\bullet$] 
\lq\lq Abelian variety\rq\rq\  often means a variety isomorphic to a \textit{bona fide} abelian variety $A$,   in other words 
 a torsor over a \textit{bona fide} abelian variety. 
\end{enumerate}

\subsection{Ackowledgements}
I heartily thank the referee for sharing with me their insights, and for letting me outline their results in Section~\ref{sec:datodabkr}.

\section{Preliminaries}\label{sec:foreplay}
\subsection{Modular sheaves}\label{subsec:negroni}
We start by presenting an equivalent characterization of modular sheaves on a HK manifold $M$. Let $V\subset H(M)$ be the image of the map $\Sym H^2(M)\to H(M)$  defined by cup product, and let $U:=V^{\bot}\subset H(M)$ be the  orthogonal (with respect to   the intersection form) of $U$. We let $V^{d}:=V\cap H^d(M)$ and  $U^{d}:=U\cap H^d(M)$. 
By Verbitsky~\cite{verb-cohom}, we have a direct sum decomposition $H(M)=V\oplus U$; in particular, 
\begin{equation*}
H^4(M)=V^4\oplus U^4.
\end{equation*}
We claim that the projection of $c_2(M)$ onto $V^4$ is nonzero. For this, it suffices to show that $\int_M c_2(M)\omega^{2n-2}\not=0$, where 
$\dim M=2n$. Suppose the contrary: Since the tangent bundle $\Theta_M$ is stable (by Yau's theorem) and $c_1(M)=0$, it follows  that
$\Theta_M$ is flat; that gives a contradiction. 
Let us denote the projection of $c_2(M)$ onto $V^4$  (following Markman) by $\ov{c}_2(M)$. Then $\cF$ is modular if and only if the projection of 
$\Delta(\cF)$ onto  $\Sym^2 H^2(M)$ is a multiple of $\ov{c}_2(M)$. 
\subsection{Generalized Kummers}\label{subsec:genkumm}
Let $A$ be an abelian surface. The generalized Kummer $K_n(A)$ is the fiber over $0$ of the map $A^{[n+1]}\to A$ given by the composition
\begin{equation*}
A^{[n+1]}\overset{\mathfrak h}{\lra} A^{(n+1)}\overset{\sigma}{\lra} A, 
\end{equation*}
where  ${\mathfrak h}[Z]:=\sum_{a\in A}\ell(\cO_{Z,a})(a)$ and $\sigma((a_1)+\dots+(a_{n+1})):=a_1+\dots+a_{n+1}$ is the summation map in the group $A$. Here and in the rest of the paper, we denote by $(a)$ the generator of the group of $0$-cycles on $A$ that corresponds to the point $a\in A$. Hence  if $k_1,\ldots,k_{n+1}$ are integers and $a_1,\ldots,a_{n+1}\in A$, then $k_1(a_1)+\dots+k_{n+1}(a_{n+1})$ is a $0$-cycle while $k_1 a_1+\dots+k_{n+1}a_{n+1}$ is an element of $A$. 

The cohomology group $H^2(K_n(A);\ZZ)$ is described as follows. There is a homomorphism $\mu_{n,A}\colon H^2(A)\to H^2(K_n(A))$ given by the composition
\begin{equation*}
H^2(A)\xrightarrow{s_{n+1}} H^2\left(A^{(n+1)}\right)\xrightarrow{({\gh}_{|K_n(A)})^{*}} H^2(K_n(A)), 
\end{equation*}
where $s_{n+1}$ is the natural symmetrization map. The map $\mu_{n,A}$ is injective but not surjective because ${\gh}_{|K_n(A)}$ contracts the prime divisor (here we assume that $n\ge 2$)
\begin{equation*}
\Delta_n(A):=\{[Z]\in K_n(A)\mid \text{$Z$ is not reduced}\}.
\end{equation*}
The cohomology class of $\Delta_n(A)$ is (uniquely) divisible by $2$ in integral cohomology. We let 
$\delta_n(A)\in H^2(A^{(n+1)};\ZZ)$ be the class such that
\begin{equation}
2\delta_n(A)=\cl(\Delta_n(A)).
\end{equation}
(Beware of the potential misunderstanding: $\delta_n(A)$ is \emph{not} the class of $\Delta_n(A)$.) One has
\begin{equation}
H^2(K_n(A);\ZZ)=\mu_{n,A}(H^2(A;\ZZ))\oplus \ZZ\delta_n(A),
\end{equation}
where orthogonality is with respect to the BBF quadratic form. Moreover, the BBF quadratic form is given by
\begin{equation}\label{kummerbbf}
q(\mu_{n,A}(\alpha)+x\delta_n(A))=(\alpha,\alpha)_A-2(n+1)x^2.
\end{equation}
(Here $(\alpha,\alpha)_A$ is the self-intersection of $\alpha\in H^2(A)$.) 
Let $\beta_1,\ldots, \beta_{2n}\in H^2(K_n(A))$. Then 
\begin{equation}\label{intsukum}
  \int\limits_{K_n(A)}\beta_1\cdot \ldots\cdot  \beta_{2n}=(n+1)\cdot\wt{\sum}\, q\left(\beta_{i_1},\beta_{i_2}\right)\cdot\ldots\cdot q\left(\beta_{i_{2n-1}},\beta_{i_{2n}}\right),
\end{equation}
 where $\wt{\sum}$ means that in the summation we avoid repeating addends which are formally equal (\textit{i.e.}, are equal modulo reordering of the factors   $q_X(\cdot,\cdot)$ and switching the entries in $q_X(\cdot,\cdot)$). The last four results are \lq\lq folklore\rq\rq. We am not aware of a printed proof. Proofs can be found in~\cite{noteogrady}. 
 We also recall that if $\cL$ is a line  bundle on a generalized Kummer $X$ of dimension $2n$, then
 \begin{equation}\label{rrgenkumm}
\chi(X;\cL)=(n+1){{\frac{1}{2}q_X(c_1(\cL))+n}\choose{n}}.
\end{equation}
A proof of the above formula can be found in \cite[Section~5.2]{britze-nieper}

From now on we deal only with $4$-dimensional generalized Kummers $K_2(A)$. We replace 
$\mu_{2,A},\Delta_2(A),\delta_2(A)$ by 
$\mu_A,\Delta(A),\delta(A)$, respectively.

Lastly let $M$ be a HK fourfold of Kummer type, \textit{i.e.}, a deformation of $K_2(A)$. If $ \zeta\in H^2(M)$, then (see for example \cite[Equation~(6)]{hass-tschink-lag-planes})
\begin{equation}\label{intcidue}
\int\limits_{M} c_2(M)\cdot  \zeta^2=54 q_M(\zeta).
\end{equation}
Moreover (see \cite[Proof of Proposition~5.1]{hass-tschink-lag-planes}),
we have
\begin{equation}\label{ciduequadro}
\int\limits_{M} c_2(M)^2=756.
\end{equation}
\section{Modular basic  sheaves on 4-dimensional generalized Kummers}\label{sec:conteduca}
\subsection{Statement of the main result}\label{subsec:quandomod}
Let $f\colon B\to A$ be a homomorphism of (\textit{bona fide}) abelian surfaces  of  degree $2$. Then $f$ defines a rational map
\begin{equation}\label{dabiada}
\begin{matrix}
K_{2}(B) & \overset{\rho}{\ldra} & K_{2}(A) \\
[Z] & \longmapsto & [f(Z)]. 
\end{matrix}
\end{equation}
The map $\rho$ is regular away from
\begin{equation}\label{vuenne}
V(f):=\{[Z]\in K_2(B) \mid \ell(f(Z))<\ell(Z)=2\}.
\end{equation}
 The result below is proved in Section~\ref{subsec:mapparo}. 
\begin{prp}\label{prp:risoro}
Keep notation and hypotheses as above. Then $V(f)$ $($see~\eqref{vuenne}$)$ is irreducible  of codimension $2$ and smooth.  Moreover, by blowing  up $V(f)$, one resolves the  indeterminacies of $\rho$.
\end{prp}
Let $\nu\colon X\to K_2(B)$ be the blow-up of $V(f)$, and let  
$\wt{\rho}\colon X\to K_{2}(A) $ be the  regular map lifting the rational map $\rho$. 
Thus we have  
the commutative diagram 
\begin{equation}\label{allevavisoni}
\xymatrix{   & X \ar[dl]_{\nu}  \ar[dr]^{\wt{\rho}} &   \\ 
  K_2(B)  \ar@{-->}[rr]^{\rho} & & K_2(A)\rlap{.} }
\end{equation}
For   a line bundle $\cL$ on $X$, we let 
$\cE(\cL):=\wt{\rho}_{*}(\cL)$, and if there is no ambiguity regarding $\cL$,  we denote it by $\cE$.  Since  the map $\wt{\rho}$ has degree $4$ and $\cL$ is torsion-free,   $\cE(\cL)$ is  a    rank $4$ torsion-free sheaf  on $K_2(A)$. 
We prove that for suitable choices of $\cL$, the sheaf $\cE(\cL)$ is modular.
In order to state our result, we introduce some notation.
Let $D\subset X$ be the exceptional divisor of the blow-up map $\nu$ (notice that $D$ is irreducible by Proposition~\ref{prp:risoro}).   There exist a  
$\omega_B\in \NS(B)$ and integers $x,y$ such that
\begin{equation}\label{ixipsilon}
c_1(\cL)=\nu^{*}(\mu_B(\omega_B)+x\delta(B))+y \cl(D).
\end{equation}
Below is the main result of the present section.
\begin{thm}\label{thm:piazzetta}
 The sheaf $\cE(\cL)$ is  modular   if and only if 
$y=x$ or $y=x+1$. If this is the case, then
\begin{equation}\label{discrigido}
\Delta(\cE(\cL))=c_2(K_2(A)).
\end{equation}
Moreover,  if $y=x$, then $\cE(\cL)$ is locally free.
\end{thm}
Theorem~\ref{thm:piazzetta} is proved in Sections~\ref{subsec:alfinlaprova}, \ref{subsec:razzismo} and~\ref{subsec:strutloc}.

\begin{rmk}\label{rmk:tensorizzo}
Let  $\xi$ be the square root of the line bundle $\cO_{K_2(A)}(\Delta)$. Thus $c_1(\xi)=\delta(A)$.  For  $t\in\ZZ$, we have
\begin{eqnarray}
\wt{\rho}_{*}(\cL\otimes\wt{\rho}^{*}\xi^{\otimes t}) & = & \cE(\cL) \otimes\xi^{\otimes t}, \label{invuno} \\
c_1(\cL\otimes\wt{\rho}^{*}\xi^{\otimes t}) & = & \nu^{*}(\mu_B(\omega_B)+(x+t)\delta(B))+(y+t) \cl(D). \label{invdue} 
\end{eqnarray}
This explains why the hypotheses on $(x,y)$ that ensure that $\cE(\cL)$ is modular are invariant under translation by multiples of $(1,1)$.
\end{rmk}
\subsection{From $S^{[2]}$ to $S^{[3]}$ according to Ellingrud and Str\o mme}
Let $S$ be a smooth surface.  We let $S^{[2,3]}\subset S^{[2]}\times S^{[3]}$ be the nested Hilbert scheme; see Jan Cheah's Ph.D. thesis (Chicago, 1994). 
As a set, we have
\begin{equation*}
S^{[2,3]}:=\left\{[W],[Z])\in S^{[2]}\times S^{[3]} \mid W\subset Z\right\}.
\end{equation*}
Let $\alpha\colon S^{[2,3]}\to S^{[2]}\times S$ be the product of the projection $S^{[2,3]}\to S^{[2]}$ and the map $([W],[Z])\mapsto 
\supp(\cI_W/\cI_Z)$.  Let $\beta\colon S^{[2,3]}\to S^{[3]}$ be the  projection. We have a commutative diagram
\begin{equation*}
\xymatrix{   & S^{[2,3]} \ar[dl]_{\alpha}  \ar[dr]^{\beta} &   \\ 
 S^{[2]}\times S   \ar@{-->}[rr]^{\xi} & &  S^{[3]},}
\end{equation*}
where $\xi([W],p):= [W\sqcup\{p\}]$ if $p\notin\supp W$. Let $\cZ_n(S)\subset S^{[n]}\times S$ be the universal subscheme. 
\begin{prp}[\textit{cf.} Ellingrud-Str\o mme \protect{\cite[Propositions~2.1 and 2.2]{ellstrointnum}}]\label{ellstro}
The map $\alpha$ is the blow-up of $\cZ_2(S)$. The Stein factorization of  $\beta$ is 
\begin{equation*}
S^{[2,3]}\overset{\gamma}{\lra}  \cZ_3(S)\lra S^{[3]},
\end{equation*}
where $\cZ_3(S)\lra S^{[3]}$ is the projection.
\end{prp}
The map $\gamma$ is an isomorphism away from the subset  
$\{(V(\gm_p^2),p)\mid p\in S\}$. Moreover, 
\begin{equation*}
\gamma^{-1}\left(\left(V\left(\gm_p^2,p\right)\right]\right)=\beta^{-1}\left(\left[V\left(\gm_p^2\right)\right]\right)=\left\{\left([W],\left[V\left(\gm_p^2\right)\right]\right) \mid W\subset V\left(\gm_p^2\right)\right\}\cong\PP(\Theta_p(S)).
\end{equation*}
\subsection{The maps $\rho$ and  $\wt{\rho}$}\label{subsec:mapparo}
\begin{proof}[Proof of Proposition~\ref{prp:risoro}]
Let ${\bm \rho}\colon B^{[3]}\dra A^{[3]}$ be defined as $\rho$, \textit{i.e.}, ${\bm \rho}([Z])=[f(Z)]$, and  
let ${\bm V}^{[3]}(f)\subset B^{[3]}$ be the set of $[Z]$ such that $\ell(f(Z))<\ell(Z)$. It suffices to prove that 
${\bm V}^{[3]}(f)$ is irreducible of codimension $2$ and smooth and that blowing it up, we resolve the indeterminacies of ${\bm \rho}$.
 Let $\varepsilon\in B$ be the nonzero element of $\ker f$:
\begin{equation}
\ker f=\{0,\varepsilon\}.
\end{equation}
A simple but useful observation is contained in the following equality:
\begin{equation}\label{verticale}
 {\bm V}^{[3]}(f)=\left\{[Z]\in B^{[3]} \mid Z=W\sqcup \{b\}, \; [W]\in B^{[2]}, \; (b+\varepsilon)\in\supp W\right\}. 
\end{equation}
(Here $\sqcup$ denotes disjoint union.) Let
$\sigma\colon B^{[2]}\times B   \to B^{[2]}\times B$ be defined by $\sigma([W],b):=([W],b+\varepsilon)$. Consider the maps
\begin{equation*}
B^{[2]}\times B   \overset{\sigma}{\lra}  B^{[2]}\times B   \overset{\xi}{\ldra}   B^{[3]} \overset{\bm \rho}{\ldra}  A^{[3]}.
\end{equation*}
Let $[Z]\in {\bm V}^{[3]}(f)$. Then we have $Z=W\sqcup \{b\}$ as in~\eqref{verticale}. Since $b\notin\supp W$, the map $\xi$ is regular at 
$([W],b)$ and defines an isomorphism between a small neighborhood (in the analytic topology)  $\wh{U}$ of  $([W],b)$ and a small neighborhood (in the analytic topology) $U$ of $[Z]$.
Moreover, $\sigma$ maps isomorphically $\wh{U}\cap\cZ_2(B)$ to
  $U\cap\xi^{-1}({\bm V}^{[3]}(f))$. Since $\cZ_2(B)$ is isomorphic to the blow-up of $B\times B$ along the diagonal, it is  irreducible of codimension $2$ and smooth. It follows that the same holds for  ${\bm V}^{[3]}(f)$.  

Moreover, $\sigma$ identifies (locally around $([W],b)$ and $[Z]$) the blow-ups of $B^{[2]}\times B$ with centers 
$\cZ_2(B)$  and ${\bm V}^{[3]}(f)$. Since ${\bm \rho}\circ\xi\circ\sigma={\bm \rho}\circ\xi$, the blow-up of ${\bm V}^{[3]}(f)$ resolves the indeterminacies of ${\bm \rho}$  by Ellingsrud--Str\o mme's Proposition~\ref{ellstro}. 
\end{proof}
Given $b\in B[3]$, let
\begin{equation}\label{erreips}
R_b:=\left\{[\{b\}\sqcup W] \mid W\in B^{[2]},\quad \supp W=\{b+\varepsilon\}\right\}\subset K_2(B).
\end{equation}
Notice that $R_b$ is isomorphic to $\PP^1$. We have an inclusion
\begin{equation}
\begin{matrix}
A[3] & \overset{\iota}{\lhra} & K_2(A) \\
a & \longmapsto & V(\gm^2_a). 
\end{matrix}
\end{equation}
In other words, $\iota(a)$ is the subscheme supported at $a$ with structure sheaf $\cO_A/\gm^2_a$.
\begin{prp}\label{prp:finfuori}
Let $X\overset{\wt{\rho}_2}{\lra} \ov{X} \overset{\wt{\rho}_1}{\lra}  K_2(A)$ be
 the Stein factorization of\, $\wt{\rho}$. 
\begin{a-enumerate}
\item\label{p:f-a}
The map $\wt{\rho}_1$ is  finite of degree $4$.
\item\label{p:f-b}
There exists an embedding $\sF\colon B[3]\hra \ov{X}$ with the following properties.
Let $b\in B[3]$, and let $a:=f(b)\in A[3]$. There is a curve $\wt{R}_b\subset X$ mapped isomorphically to $R_b$ by $\nu$ and contracted by 
$\wt{\rho}_2$ to a point $\sF(b)\in\wt{\rho}_1^{-1}(\iota(a))$. Moreover, $\wt{R}_b=\wt{\rho}_2^{-1}(\sF(b))$.
\item\label{p:f-c}
Away from the union of the $\wt{R}_b$ $($for $b\in B[3])$, the map $\wt{\rho}_2$  is an isomorphism onto its image.
\item\label{p:f-d}
$\ov{X}$ has rational singularities.
\end{a-enumerate}
\end{prp}
\begin{proof}
Items~\eqref{p:f-a}--\eqref{p:f-c} follow easily from
the proof of Proposition~\ref{prp:risoro} together with Ellingsrud and Str\o mme's Proposition~\ref{ellstro}.
Let us prove item~\eqref{p:f-d}. First $\ov{X}$  is smooth away from $\sF(B[3])$. On the other hand, we claim that 
in a neighborhood of $\sF(b)$ for $b\in B[3]$, $\ov{X}$ is isomorphic to the fiber over $0$ of the composition
$$\cZ_3(A)\lra A^{[3]}\lra A^{(3)}\lra A,$$
where the last map is the summation map. In fact, our assertion follows from the proof of  Proposition~\ref{prp:risoro}  and from Proposition~\ref{ellstro}. Since $\cZ_3(A)\lra A$ is a locally (in the classical or \'etale topology) trivial fibration, $\ov{X}$ has a rational singularity at  $\sF(b)$ if
 $\cZ_3(A)$ has rational singularities. The last assertion 
holds by the main result of~\cite{univrat}.
\end{proof}
\subsection{Chern classes, I}
We follow the notation introduced in Theorem~\ref{thm:piazzetta}. We will set $\cE=\cE(\cL)$.
\begin{lmm}\label{joetimes}
For  $\cE=\cE(\cL)$, the following hold:
\begin{eqnarray*}
\ch_1(\cE) & = & 
\wt{\rho}_{*}\left[c_1(\cL)+\frac{c_1(X)}{2} \right], \\
\ch_2(\cE) & = & \wt{\rho}_{*}\left[\frac{c_1(\cL)^2}{2}+\frac{c_1(\cL)\cdot  c_1(X)}{2}
+\frac{c_1(X)^2+c_2(X)}{12}\right]-\frac{c_2(K_2(A))}{3}.
\end{eqnarray*}
\end{lmm}
\begin{proof}
By Proposition~\ref{prp:finfuori}, the higher direct images sheaves $R^i\wt{\rho}_{*}\cL$ (for $i>0$) vanish over $(K_2(A)\setminus A[3])$. Hence the Chern character of $\cE$ over  $(K_2(A)\setminus A[3])$ is computed by the Grothendieck--Riemann--Roch (GRR) theorem.  Since 
$A[3]$ has codimension $4$ in $K_2(A)$, it follows that the same formula gives $\ch_p(\cE)$ on $K_2(A)$ for $p\le 3$. Writing out the GRR theorem, we get the formulae of the corollary.
\end{proof}
\begin{prp}\label{ciunoquad}
For  $\cE=\cE(\cL)$ and $\alpha\in H^2(K_2(A))$, we have
\begin{equation}
\int\limits_{K_2(A)}\ch_1(\cE)^2\cdot \alpha^2=6\left(2\int\limits_A(f_{*}\omega_B)^2-3(2x+2y-1)^2\right) q(\alpha)
+6q(2\mu_A(f_{*}\omega_B)+(2x+2y-1)\delta(A),\alpha)^2.
\end{equation}
\end{prp}
\begin{proof}
Since $X$ is the blow-up of $K_2(B)$ with center $V(f)$, we have 
\begin{equation}\label{ciunoix}
 c_1(X)=-\cl(D). 
\end{equation}
Thus Lemma~\ref{joetimes} gives
\begin{equation}\label{passint}
\ch_1(\cE)  =  
\wt{\rho}_{*} c_1(\cL)-\frac{1}{2} \wt{\rho}_{*}\cl(D).
\end{equation}
We claim that
\begin{equation}\label{solotre}
\wt{\rho}_{*}\nu^{*}\mu_{B}(\omega_B)=2\mu_A(f_{*}\omega_B),\quad \wt{\rho}_{*}\nu^{*}\Delta(B)=2\Delta(A),\quad 
\wt{\rho}_{*}D=\Delta(A).
\end{equation}
One gets the first equality by representing the Poincar\'e dual of $\omega_B$ by a $C^{\infty}$ immersed submanifold $\Sigma$, and the Poincar\'e dual of $\mu_{B}(\omega_B)$ by the immersed submanifold 
$I_{\Sigma}\subset  K_2(B)$ parametrizing the points $[Z]\in K_2(B)$ such that $Z$ meets $\Sigma$.  It follows that  
 the Poincar\'e dual of $\wt{\rho}_{*}\nu^{*}\mu_{B}(\omega_B)$ is represented  by the immersed submanifold $I_{f(\Sigma)}\subset  K_2(A)$ parametrizing the points  $[Z]\in K_2(A)$ such that $Z$ meets 
  $f(\Sigma)$ (the point being that $I_{\Sigma}$ meets properly $V(f)$  for a generic choice of $\Sigma$), counted with multiplicity $2$ because the map $I_{\Sigma}\to I_{f(\Sigma)}$ has degree $2$. 
 
In order to prove the second equality  in~\eqref{solotre}, notice that $\Delta(B)$ intersects properly $V(f)$, and hence $\wt{\rho}_{*}\nu^{*}\Delta(B)$ is represented by the closure of $\wt{\rho}(\Delta(B)\setminus V(f))$, which is $\Delta(A)$, with multiplicity the degree of the map $(\Delta(B)\setminus V(f))\to \Delta(A)$, which is $2$.

The proof of Proposition~\ref{prp:risoro} gives the third equality in~\eqref{solotre}. 

Plugging into~\eqref{passint} the equations in~\eqref{solotre}, we get that
\begin{equation}\label{eccociuno}
\ch_1(\cE)  =2\mu_A(f_{*}\omega_B)+(2x+2y-1)\delta(A).
\end{equation}
The proposition follows from~\eqref{eccociuno} and the  formula in~\eqref{intsukum}.

\end{proof}
\subsection{Analysis of $V(f)$}\label{subsec:vueffe}
In order to compute $\int_{K_2(A)}\ch_2(\cE)\cdot \alpha^2$ for $\alpha\in H^2(K_2(A))$, we must examine more closely $V(f)$. 
Let  $\gh\colon K_2(B)\to B^{(3)}$ be the Hilbert--Chow map. The image of $\gh$ is equal to 
\begin{equation*}
{\mathsf S}:=\{(b_1)+(b_2)+(b_3) \mid b_1+b_2+b_3=0\}.
\end{equation*}
Let ${\mathsf S}_{\epsilon}\subset {\mathsf S}$ be the subset of $(b_1)+(b_2)+(b_3)$ such that $b_i-b_j=\epsilon$ for some $i,j\in\{1,2,3\}$. Then $V(f)=\gh^{-1}({\mathsf S}_{\epsilon})$, and the restriction of $\gh$ to $V(f)$ is a birational map 
$\gh_{V(f)}\colon V(f)\to {\mathsf S}_{\epsilon}$. Let 
\begin{equation}\label{remagi}
\begin{matrix}
B & \overset{g}{\lra} & {\mathsf S}_{\epsilon} \\
b & \longmapsto & (b)+(b+\epsilon)+(-2b+\epsilon). 
\end{matrix}
\end{equation}
Then $g$ identifies ${\mathsf S}_{\epsilon}$ with $B/\ker f=A$.
\begin{prp}\label{dilato}
Identifying ${\mathsf S}_{\epsilon}$ with $A$ via the map $g$, the map $\gh_{V(f)}\colon V(f)\to {\mathsf S}_{\epsilon}$ is identified with the blow-up of $A[3]$. Let $R_1,\ldots,R_{81}$ be the exceptional divisors of\, $\gh_{V(f)}$ $($notice that they are equal to $R_{b_1},\ldots,R_{b_{81}}$, where $b_j\in B[3]$, with notation as in~\eqref{erreips}$)$. Then
\begin{equation}\label{restdel}
\Delta(B)_{|V(f)}=2\sum\limits_{i=1}^{81}R_i.
\end{equation}
\end{prp}
\begin{proof}
By Proposition~\ref{prp:risoro}, we know that $V(f)$ is smooth. Moreover,  $\gh_{V(f)}$ is an isomorphism over $(A\setminus A[3])$,  and it has fiber $\PP^1$   over each point of $A[3]$. It follows that $\gh_{V(f)}$ is  the blow-up of $A[3]$. 
The remaining part of the proposition is straightforward.
\end{proof}
\begin{crl}\label{diciotto}
Let $V(f)\subset K_2(B)$ be as in~\eqref{vuenne}, and let $\zeta\in H^2(B)$. Then
\begin{equation*}
\int\limits_{V(f)}\left(\mu_B(\zeta)+t\delta(B)\right)^2=18\left(\int\limits_B \zeta^2\right)-81 t^2.
\end{equation*}
\end{crl}
\begin{proof}
 We have $\mu_B(\zeta)=\gh^{*}(\zeta^{(3)})$, where  $\gh\colon K_2(B)\to B^{(3)}$ is the Hilbert--Chow map and 
$\zeta^{(3)}\in H^2(B^{(3)})$ is the symmetrization of the class $\zeta$.  Since 
$g^{*}(\zeta^{(3)})=6\zeta$ and  $g$ is surjective of degree $2$, we get that 
\begin{equation*}
\int\limits_{V(f)}\mu_B(\zeta)^2=\int\limits_W \zeta^{(3)}\cdot  \zeta^{(3)}=\frac{1}{2}\int\limits_B(6\zeta)^2=18\int\limits_B \zeta^2.
\end{equation*}

Next we notice that 
\begin{equation}
\int\limits_{V(f)}\mu_B(\zeta)\cdot \delta(B)=\int\limits_{W}\zeta^{(3)}\cdot  \gh_{V(f),*}(\delta(B))=0.
\end{equation}
(The last equality holds because $ \gh_{V(f),*}(\Delta(B))=0$.)

Lastly, by~\eqref{restdel}, we have 
\begin{equation}\pushQED{\qed}
\int\limits_{V(f)}\delta(B)^2=-81.\qedhere \popQED
\end{equation}
\renewcommand{\qed}{}    
\end{proof}
\begin{lmm}\label{lmm:geomro}
The restriction of\, $\wt{\rho}$ defines a map $\nu^{*}\Delta(B)\to \Delta(A)$ of degree $2$ and a birational map $D\to\Delta(A)$. We have
\begin{equation}\label{tirind}
\wt{\rho}^{*}\delta(A)=\nu^{*}\delta(B)+\cl(D).
\end{equation}
\end{lmm}
\begin{proof}
The first sentence is immediate. We also see that $\rho$ is unramified at the generic point of $\Delta(B)$, and hence
\begin{equation}\label{tirana}
\wt{\rho}^{*}\Delta(A)=\nu^{*}\Delta(B)+m\cl(D).
\end{equation}
Since $K_X\equiv D$, we get $m=2$ by the adjunction formula for the map $\wt{\rho}$. Dividing by $2$ the equality in~\eqref{tirana}, we get~\eqref{tirind}.
\end{proof}
\begin{prp}\label{treclassi}
We have 
\begin{eqnarray}
c_1(\cN_{V(f)/K_2(B)}) & = & \sum_{i=1}^{81}\cl(R_i)=\delta(B)_{|V(f)}, \label{ciunonorm} \\
\int\limits_{V(f)}c_2(\cN_{V(f)/K_2(B)}) & = & 3^4, \label{ciduenorm}  \\
\int\limits_{V(f)}c_2(K_2(B)) & = & 3^5. \label{ciduesuvu}
\end{eqnarray}
\end{prp}
\begin{proof}
By Proposition~\ref{dilato}, we have $c_1(V(f))=-\sum_{i=1}^{81}\cl(R_i)$, and hence~\eqref{ciunonorm}   holds because $c_1(K_2(B))=0$ (remember Proposition~\ref{dilato}). 

In order to prove~\eqref{ciduenorm},  we notice that
\begin{multline}
4\cdot 81 =\int\limits_{K_2(A)}\delta(A)^4= \frac{1}{4}\int\limits_X(\nu^{*}\delta(B)+\cl(D))^4 = \\
 =\frac{1}{4}\int\limits_{K_2(B)}\delta(B)^4+
\int\limits_X \nu^{*}\delta(B)^3\cdot  \cl(D)+\frac{3}{2} \int\limits_X \nu^{*}\delta(B)^2\cdot  \cl(D)^2 + \\
+ \int\limits_X \nu^{*}\delta(B)\cdot  \cl(D)^3 + \frac{1}{4}\int\limits_X \cl(D)^4
=81+\frac{3}{2}\cdot 81+ 81+ \frac{1}{4}\int\limits_X \cl(D)^4.
\end{multline}
(To compute the third and fourth integrals into the second line, use Corollary~\ref{diciotto} and~\eqref{ciunonorm}.) 
It follows that 
\begin{equation}
2\cdot 81= \int\limits_X \cl(D)^4=\int\limits_{V(f)}c_2(\cN_{V(f)/K_2(B)})-c_1(\cN_{V(f)/K_2(B)})^2.
\end{equation}
By~\eqref{ciunonorm}, we get that~\eqref{ciduenorm} holds.

Lastly, \eqref{ciduesuvu} follows from~\eqref{ciunonorm}, \eqref{ciduenorm} and the normal exact sequence for the restriction of $\Theta_X$ to $V(f)$.
\end{proof}
\subsection{Chern classes, II}
By definition, $D$ is  $\PP(\cN_{V(f)/K_2(B)})$, where $\cN_{V(f)/K_2(B)}$ is the normal bundle of $V(f)$ in $K_2(B)$. Let $\nu_D\colon D\to V(f)$ be the structure map (\textit{i.e.}, the restriction of $\nu$), and let
\begin{equation}\label{lambxi}
0\lra \lambda\lra \nu_D^{*}\cN_{V(f)/K_2(B)} \lra \xi\lra 0
\end{equation}
be the tautological exact sequence, where $\lambda=\cO_D(-1)$ is the normal  bundle of $D$ in $X$. Let $i\colon D\hra X$ be the inclusion map.
\begin{lmm}\label{carnevaris}
We have
\begin{equation}\label{miglioramica}
 c_2(X)=\nu^{*}c_2(K_2(B))+i_{*}[c_1(\xi)].
\end{equation}
\end{lmm}
\begin{proof}
The exact sequence
\begin{equation}\label{swedishlibrary}
0\lra \Theta_X\overset{d\nu}{\lra} \nu^{*}\Theta_{K_2(B)} \lra i_{*}\xi\lra 0
\end{equation}
gives
\begin{equation}\label{tanix}
 c(X)=\nu^{*}c(K_2(B))\cdot  c( i_{*}\xi)^{-1}. 
\end{equation}
The Chern classes of $i_{*}\xi$ are expressed via the GRR theorem. One gets that
\begin{equation}\label{pushxi}
c_1(i_{*}\xi)=\cl(D),\quad c_2(i_{*}\xi)=i_{*}[c_1(\lambda)-c_1(\xi)].
\end{equation}

  Plugging this into~\eqref{tanix}, we get the lemma.
\end{proof}
\begin{lmm}\label{sviluppo}
For  $\cE=\cE(\cL)$, we have
\begin{multline}
\ch_2(\cE)  =  \frac{1}{2}\wt{\rho}_{*}\left\{\nu^{*}(\mu_B(\omega_B)+x\delta(B))^2+(2y-1)\nu^{*}(\mu_B(\omega_B)+x\delta(B))\cdot \cl(D)\right\} + \\
+\frac{1}{12}\wt{\rho}_{*}\left\{i_{*}\left[c_1(\xi)+(6y^2-6y+1) c_1(\lambda)   \right]    \right\} 
+\frac{1}{12}\wt{\rho}_{*}\left(\nu^{*}c_2(K_2(B))   \right)  -\frac{1}{3} c_2(K_2(A)).
\end{multline}
\end{lmm}
\begin{proof}
Plug~\eqref{ciunoix} and~\eqref{miglioramica}  in the second equation of Lemma~\ref{joetimes}.
\end{proof}
We are now ready to compute $\int_{K_2(A)}\Delta(\cE)\cdot  \alpha^2$ for $\alpha\in H^2(K_2(A))$.
\begin{prp}\label{sumu}
For  $\cE=\cE(\cL)$ and  $\gamma\in H^2(A)$, we have
\begin{equation}\label{ixmenoy}
\int\limits_{K_2(A)}\Delta(\cE)\cdot  \mu_A(\gamma)^2=18\left(4(x-y)^2+4(x-y)+3\right) q(\mu_A(\gamma)).
\end{equation}
\end{prp}
\begin{proof}
First we notice that
\begin{equation}\label{treporc}
\int\limits_{B}f^{*}\gamma^2=2\int\limits_{A}\gamma^2,\quad \int\limits_A (f_{*}\omega_B)^2=2\int\limits_B \omega_B^2,\quad 
\wt{\rho}^{*}\mu_A(\gamma)=\nu^{*}\mu_B(f^{*}\gamma).
\end{equation}
The first equality holds because $\deg f=2$. The second equality holds by the following series of equalities: 
\begin{equation}
\int\limits_A(f_{*}\omega_B)^2=\int\limits_A f_{*}(\omega_B\cdot  f^{*}(f_{*}\omega_B))=\int\limits_B \omega_B\cdot (\omega_B+T_{\epsilon}^{*}\omega_B)
=2\int\limits_B \omega_B^2.
\end{equation}
(We let $T_{\epsilon}\colon B\to B$ be the translation by $\epsilon$.) The third equality in~\eqref{treporc} is proved by representing the Poincar\'e dual of $\mu_A(\gamma)$ and $\mu_B(f^{*}\gamma)$ as in the proof of the first equality in~\eqref{solotre}.

 By Proposition~\ref{ciunoquad}, we have
\begin{equation}\label{cummins}
\int\limits_{K_2(A)}\ch_1(\cE)^2\cdot \alpha^2=\left(24\int\limits_B \omega_B^2-18(2x+2y-1)^2\right)
\cdot\left( \int\limits_A \gamma^2\right)
+24\left( \int\limits_A f_{*}\omega_B\cdot  \gamma \right)^2.
\end{equation}
On the other hand, Lemma~\ref{sviluppo} and the third equality in~\eqref{treporc} give that
\begin{multline}\label{ashes}
\int\limits_{K_2(A)} \ch_2(\cE)\cdot  \mu_A(\gamma)^2  =  \frac{1}{2}\int\limits_{K_2(B)} (\mu_B(\omega_B)+x\delta(B))^2\cdot  
\mu_B(f^{*}\gamma)^2 
-\frac{1}{2}\left(y^2-y\right) \int\limits_{V(f)}\mu_B(f^{*}\gamma)^2 + \\
+\frac{1}{12}\int\limits_{K_2(B)}c_2(K_2(B))\cdot  \mu_B(f^{*}\gamma)^2 
 -\frac{1}{3}\int\limits_{K_2(A)}c_2(K_2(A))\cdot  \mu_A(\gamma)^2.
\end{multline}
Recalling the results of Section~\ref{subsec:genkumm}
 and Corollary~\ref{diciotto}, we get that
\begin{equation}\label{mcgrath}
\int\limits_{K_2(A)} \ch_2(\cE)\cdot  \mu_A(\gamma)^2  
=3 \left(\left(\int\limits_B \omega_B^2\right) -6 x^2-6 y^2+6y-3\right)\cdot\left( \int\limits_A \gamma^2\right)+
3\left(\int\limits_B \omega_B\cdot  f^{*}\gamma \right)^2. 
\end{equation}
Since $\cE$ has rank $4$, we have $\Delta(\cE)=\ch_1(\cE)^2-8\ch_2(\cE)$. The proposition follows from~\eqref{cummins} and~\eqref{mcgrath} (recall that  $\mu\colon H^2(A)\to H^2(K_2(A))$ is an isometry).
\end{proof}
\begin{prp}\label{sumudel}
For  $\cE=\cE(\cL)$ and    $\gamma\in H^2(A)$, we have 
\begin{equation}\label{eqmudel}
\int\limits_{K_2(A)}\Delta(\cE)\cdot  \mu_A(\gamma)\cdot  \delta(A)=0.
\end{equation}
\end{prp}
\begin{proof}
By Proposition~\ref{ciunoquad}, we have
\begin{equation*}
\int\limits_{K_2(A)}\ch_1^2\cE)\cdot  \mu_A(\gamma)\cdot  \delta(A)=-72(2x+2y-1)\int\limits_A 
\gamma\cdot  f_{*}\omega_B.
\end{equation*}
Lemmas~\ref{sviluppo} and  \ref{lmm:geomro}, Corollary~\ref{diciotto} and~\eqref{ciunonorm} give that 
\begin{multline*}
\int\limits_{K_2(A)}\ch_2\cE)\cdot  \mu_A(\gamma)\cdot  \delta(A)  =  \frac{1}{2}\int\limits_{K_2(B)}(\mu_B(\omega_B)+x\delta(B))^2\cdot  \mu_B(f^{*}\gamma)\cdot \delta(B)\;+ \\
+\frac{1}{2}(2y-1)\nu^{*}(\mu_B(\omega_B)+x\delta(B))\cdot \mu_B(f^{*}\gamma)\cdot \cl(D)^2 = \\
=-9(2x+2y-1)\int\limits_{B}\omega_B\cdot  f^{*}\gamma.
\end{multline*}
Since  $\Delta(\cE)=\ch_1(\cE)^2-8\ch_2(\cE)$, the proposition follows from the above equalities.
\end{proof}
\begin{prp}\label{su2delta}
Keeping notation as above, we have
\begin{equation}\label{natavota}
\int\limits_{K_2(A)}\Delta(\cE(\cL))\cdot \delta(A)^2=54((x-y)^2+(x-y)+1)q(\delta(A)).
\end{equation}
\end{prp}
\begin{proof}
Let   $\cE=\cE(\cL)$. By Proposition~\ref{ciunoquad}, we have
\begin{equation}\label{anna}
\int\limits_{K_2(A)}\ch_1(\cE)^2\cdot  \delta(A)^2=324(2x+2y-1)^2-72 \int\limits_A (f_{*}\omega_B)^2.
\end{equation}
By Lemmas~\ref{sviluppo} and~\ref{lmm:geomro}, we have
\begin{multline}
\int\limits_{K_2(A)}\ch_2(\cE)  \cdot  \delta(A)^2= -\frac{1}{3} \int\limits_{K_2(A)} c_2(K_2(A))\cdot  \delta(A)^2+ \\
+ \frac{1}{2}\int\limits_{K_2(B)}(\mu_B(\omega_B)+x\delta(B))^2 \cdot  \delta(B)^2+ 
\frac{1}{12}\int\limits_D\left[c_1(\xi)+(6y^2-6y+1) c_1(\lambda)   \right]\cdot  \delta(B)^2+ \\
+\frac{1}{12} \int\limits_{K_2(B)} c_2(K_2(B)) \cdot  \delta(B)^2+(2y-1)\int\limits_X \nu^{*}((\mu_B(\omega_B)+x\delta(B))\cdot  \delta(B))\cdot \cl(D)^2+ \\
+
\frac{1}{6}\int\limits_D\left[c_1(\xi)\cup c_1(\lambda)+(6y^2-6y+1) c_1(\lambda)^2  \right] \cdot  \delta(B)\;+ \\
+\frac{1}{2}\int\limits_X\nu^{*}(\mu_B(\omega_B)+x\delta(B))^2\cdot  \cl(D)^2+\frac{1}{2}(2y-1)\int\limits_X\nu^{*}(\mu_B(\omega_B)+x\delta(B))\cdot \cl(D)^3+ \\
+\frac{1}{12}\int\limits_D\left[c_1(\xi)\cdot  c_1(\lambda)^2+(6y^2-6y+1) c_1(\lambda)^3  \right]    
+\frac{1}{12}\int\limits_X \nu^{*}c_2(K_2(B)) \cdot  \cl(D)^2.
\end{multline}
We have proved all the results needed to compute each integral above. In particular, notice that~\eqref{lambxi} and Proposition~\ref{treclassi} give the following relations in the cohomology ring of $D$:  
\begin{equation*}
c_1(\xi)=\sum\limits_{i=1}^{81}\cl(R_i)-c_1(\lambda),\quad 
c_1(\lambda)^2=\nu_D^{*} \left(\sum\limits_{i=1}^{81}\cl(R_i)\right)\cdot  c_1(\lambda)-81 \nu_D^{*} (\eta_{V(f)}),
\end{equation*}
where $\nu_D\colon D\to V(f)$ is the natural map (the restriction of $\nu$ to $D$) and $\eta_{V(f)}$ is the fundamental class of $V(f)$. It follows that
\begin{equation}\label{marco}
\int\limits_{K_2(A)}\ch_2(\cE)  \cdot  \delta(A)^2= \\
=\frac{81}{2}\left(5 x^2+6 xy+5 y^2-3 x-5 y+2\right)-18\int_B\omega_B^2.
\end{equation}
The proposition follows from~\eqref{anna} and~\eqref{marco} because $\Delta(\cE)=-8\ch_2(\cE)+\ch_1(\cE)^2$.
\end{proof}
Notice that 
the right-hand sides of~\eqref{ixmenoy}, \eqref{eqmudel} and~\eqref{natavota} are polynomials in $(x-y)$; this is explained by Remark~\ref{rmk:tensorizzo}.
\subsection{Modularity}\label{subsec:alfinlaprova}
By Propositions~\ref{sumu}, \ref{sumudel} and~\ref{su2delta}, the sheaf $\cE(\cL)$ is modular if and only if
\begin{equation}
18(4(x-y)^2+4(x-y)+3)=18(3(x-y)^2+3(x-y)+3).
\end{equation}
The solutions are given by $x-y=0$ and $x-y=-1$. 
For modular $\cE(\cL)$, one gets that 
\begin{equation}\label{primocinquantaquattro}
d(\cE(\cL))=54
\end{equation}
by plugging the values $x-y=0$ and $x-y=-1$ into the polynomial on the right-hand side of~\eqref{ixmenoy} or~\eqref{natavota} (we recall  that 
$d(\cE(\cL))$ is defined by the equation in~\eqref{fernand}).
\subsection{Action of  $\Aut^0(K_2(A))$ on $\cE(\cL)$}\label{subsec:yesmen}
We prove that $\Aut^0(K_2(A))$ acts trivially on $\cE(\cL)$. 
First we describe generators of $\Aut^0(K_2(A))$. 
For  $\tau\in A[3]$, let $g_{\tau}\colon K_2(A)\to K_2(A)$ be the automorphism mapping a subscheme $Z$ to $\tau(Z)$.
Then $g_{\tau}\in \Aut^0(K_2(A))$, and the set of such automorphism is a normal subgroup isomorphic to $\ZZ/(n+1)^4$. 
Let  $\iota\colon K_2(A)\to K_2(A)$ be the automorphism mapping  a subscheme $Z$ to $-{\textbf 1}(Z)$, where 
$-{\textbf 1}\colon A\to A$ is multiplication by $-1$. Then $\iota$ belongs to $\Aut^0(K_2(A))$, and moreover $\Aut^0(K_2(A))$ is generated by the collection of the $g_{\tau}$ and $\iota$; see~\cite[Corollary 5]{boiss-high-enriques}.
\begin{prp}\label{prp:azioneaut}
Let $\cL$ be a line bundle on $X$ as in Section~\ref{subsec:quandomod} $($no further assumptions on $\cL)$. If 
$\varphi\in\Aut^0(K_2(A))$, then
\begin{equation}\label{fibinv}
\varphi^{*}(\cE(\cL))\cong\cE(\cL).
\end{equation}
\end{prp}
\begin{proof}
Let $\cE:=\cE(\cL)$. We must prove that $g_{\tau}^{*}(\cE)\cong\cE$ for all $\tau\in A[3]$ and that $\iota^{*}(\cE)\cong\cE$. 
We start by  noting that the restriction of $f$ to $B[3]$ defines an isomorphism $B[3]\to A[3]$. 
Let $\lambda\in B[3]$ be the element such that $f(\lambda)=\tau$. The corresponding  automorphism $g_{\lambda}\colon K_2(B)\to K_2(B)$ 
maps $V(f)$ to itself and hence  lifts to an automorphism $\wt{g}_{\lambda}\colon X\to X$. Moreover,  $\wt{g}^{*}_{\lambda}(\cL)\cong \cL$ because 
$g_{\lambda}$ acts trivially on $H^2(K_2(B))$.
Since $\wt{\rho}\circ \wt{g}_{\lambda}=g_{\tau}\circ \wt{\rho}$,  we get that 
$g_{\tau}^{*}(\cE)\cong\cE$. The proof that $\iota^{*}(\cE)\cong\cE$ is similar.
\end{proof}
\subsection{The discriminant}\label{subsec:razzismo}
 We prove that the equality in~\eqref{discrigido} holds.  We assume that $\cE(\cL)$ is modular; \textit{i.e.}, either $y=x$ or $y=x+1$.
 As mentioned in Section~\ref{subsec:negroni}, this means that
\begin{equation}\label{delapp}
\Delta(\cE(\cL))\in \QQ \ov{c}_2(K_2(A))\oplus \Sym^2 H^2(M)^{\bot}_{\QQ},
\end{equation}
where $ \ov{c}_2(K_2(A))$ is the orthogonal projection of $c_2(K_2(A))$ in the image of the map 
$\Sym^2 H^2(K_2(A))\to H^4(K_2(A))$.  
The right-hand side of~\eqref{delapp} is studied
 in~\cite{hass-tschink-lag-planes}. In Theorem 4.4 of \textit{op.\ cit.}, one finds the definition of  submanifolds 
 $Z_{\tau}\subset K_2(A)$ for  $\tau\in A[3]$ with the following properties: The cohomology classes 
 $\{\cl(Z_{\tau})\}_{\tau\in A[3]}$ give a basis of 
 the right-hand side of~\eqref{delapp}, and the action of $A[3]$ on the $Z_{\tau}$ is by translation of the index. It follows that  the subspace of invariants for this action has dimension $1$, and since $c_2(K_2(A))$ is a nonzero class, it
   generates the space of 
 invariants  (in fact, by \cite[Proposition~5.1]{hass-tschink-lag-planes}, we have 
 $c_2(K_2(A))=\frac{1}{3}\sum_{\tau}\cl(Z_{\tau})$). 
Since~\eqref{fibinv} holds for all $\varphi=g_{\tau}$, where $\tau\in A[3]$,  $\Delta(\cE(\cL))$ is invariant  and hence  it is a multiple of $c_2(K_2(A))$. By~\eqref{primocinquantaquattro} and~\eqref{intcidue}, it follows that 
$\Delta(\cE(\cL))=c_2(K_2(A))$. 

For later use, we record here the following immediate consequence of the equality in~\eqref{discrigido}:
\begin{equation}\label{eccochardue}
\ch_2(\cE(\cL))=\frac{1}{8}\left(c_1(\cE(\cL))^2-c_2(K_2(A))\right).
\end{equation}
\subsection{Local structure of modular $\cE$}\label{subsec:strutloc}
\begin{prp}\label{prp:loclibero}
Let $\cL$ be a line bundle on $X$ such that
\begin{equation*}
c_1(\cL)=\nu^{*}(\mu_B(\omega_B)+x\delta(B))+x \cl(D); 
\end{equation*}
\textit{i.e.}, $x=y$ in~\eqref{ixipsilon}. Then $R^i \wt{\rho}_{*}\cL=0$ for  $i>0$,
and $\cE(\cL)$ is locally free.
\end{prp}
\begin{proof}
Let  $X\overset{\wt{\rho}_2}{\lra} \ov{X} \overset{\wt{\rho}_1}{\lra}  K_2(A)$ be the Stein factorization of $\wt{\rho}$; see Proposition~\ref{prp:finfuori}. It suffices to prove that $\wt{\rho}_{2,*}\cL$ is an invertible sheaf and that
\begin{equation}\label{altosvan}
R^i \wt{\rho}_{2,*}\cL=0,\quad i>0. 
\end{equation}
By~\eqref{invuno} and~\eqref{invdue}, we may assume that $x=y=0$. Let $\sF\colon B[3]\hra \ov{X}$ be the  embedding of item~\eqref{p:f-b} of Proposition~\ref{prp:finfuori}. The map $\wt{\rho}_{2}$ is an isomorphism over $(\ov{X}\setminus \sF(B[3]))$. Hence  away from $\sF(B[3])$  the sheaf $\wt{\rho}_{2,*}\cL$ is invertible  and~\eqref{altosvan} holds.  

Let  $b\in B[3]$. The fiber $\wt{R}_b=\wt{\rho}_{2}^{-1}(\sF(b))$ is mapped isomorphically to $R_b\subset K_2(B)$, where $R_b$ is defined in~\eqref{erreips}. Since $x=y=0$, the line bundle $\cL$ is the pull-back of a line bundle on $B^{(3)}$ via the Hilbert--Chow map $K_2(B)\to B^{(3)}$. The curve $R_b$ (a $\PP^1$) is contracted by the Hilbert--Chow map. It follows that $\cL$ is trivial on an open neighborhood $\cU$ of 
$\wt{R}_b$. Shrinking $\cU$, we may assume that $\wt{\rho}_2(\cU)$ is open (because   $\wt{R}_b=\wt{\rho}_{2}^{-1}(\sF(a))$, where $a=f(b)$, and $\wt{\rho}_2$ is projective), and hence $\wt{\rho}_{2,*}\cL$ is  invertible in a neighborhood of
$\sF(b)$.

By  Proposition~\ref{prp:finfuori}, we know that $\ov{X}$ has rational singularities. 
Hence~\eqref{altosvan} holds because in a neighborhood of $\wt{R}_b$ we have $\cL\cong\cO_X$.
\end{proof}
\section{Euler characteristic of the endomorphism bundle}\label{sec:caratteulero}
\subsection{The result}
The notation introduced in Section~\ref{sec:conteduca} is in force throughout the present section. The main result is the following. 
\begin{prp}\label{prp:charendotre}
 Assume that $c_1(\cL)=\nu^{*}(\mu_B(\omega_B))$, where $\omega_B\in \NS(B)$. Then
\begin{equation}\label{charendotre}
\chi(K_2(A),End(\cE(\cL))=3.
\end{equation}
\end{prp}
\subsection{Preliminary computations}
Recall that $i\colon D\hra X$ is the inclusion 
of the exceptional divisor of the blow-up map $\nu\colon X\to K_2(B)$. Let $\lambda$  and  $\xi$ be the line bundles on $D$ appearing in~\eqref{lambxi}.
\begin{lmm}\label{lmm:toddtre}
Keep notation as above. Then
\begin{equation}
\td_3(X)=-\frac{1}{24}\cl(D)\cdot\nu^{*}(c_2(K_2(B)))-\frac{1}{24}i_{*}(c_1(\lambda)\cdot c_1(\xi)).
\end{equation}
\end{lmm}
\begin{proof}
We compute   modulo $H^8(X;\QQ)$. By the equalities in~\eqref{pushxi}, we have
\begin{equation*}
\td(i_{*}(\xi))\equiv 1+\frac{1}{2}\cl(D)+\frac{1}{12}\left(\cl(D)^2+i_{*}(c_1(\lambda)-c_1(\xi))\right)+
\frac{1}{24}\cl(D)\cdot i_{*}(c_1(\lambda)-c_1(\xi)).
\end{equation*}
By the exact sequence in~\eqref{swedishlibrary} and multiplicativity of the Todd class, we get that
\begin{multline*}\pushQED{\qed}
\td(X)=\nu^{*}\td(K_2(B))\cdot \td(i_{*}(\xi))^{-1}\equiv\\
\equiv 1-\frac{1}{2}\cl(D)+\frac{1}{12}\left(\nu^{*}c_2(K_2(B)+\cl(D)^2+i_{*}(c_1(\lambda)+c_1(\xi))\right)-
\\
-\frac{1}{24} \cl(D)\cdot\nu^{*}c_2(K_2(B))-\frac{1}{24} i_{*}(c_1(\lambda)\cdot c_1(\xi).\qedhere \popQED
\end{multline*}
\renewcommand{\qed}{}    
\end{proof}
\begin{lmm}\label{lmm:sollciuno}
Let $\wt{\rho}\colon X\to K_2(A)$ be the lifting of   $\rho\colon K_2(B)\dra K_2(A)$.  Then
\begin{equation}
\wt{\rho}^{*}(\ch_1(\cE(\cL))=4\nu^{*}(\mu_B(\omega_B))-\nu^{*}\delta(B)-\cl(D).
\end{equation}
\end{lmm}
\begin{proof}
By the equalities in~\eqref{eccociuno} and in~\eqref{tirind}, we have
\begin{equation*}
\wt{\rho}^{*}(\ch_1(\cE(\cL))=\wt{\rho}^{*}(2\mu_A(f_{*}\omega_B)-\delta(A))=2\wt{\rho}^{*}(\mu_A(f_{*}\omega_B))-\nu^{*}\delta(B)-\cl(D).
\end{equation*}
Arguing as in the proof of the first equality in~\eqref{solotre}, we get that $\wt{\rho}^{*}(\mu_A(f_{*}\omega_B))=2\nu^{*}(\mu_B(\omega_B))$, and the lemma follows.
\end{proof}
\subsection{Chern numbers of $\cE(\cL)$}
\begin{prp}\label{prp:ci4ci1ci3}
Keep hypotheses as in Proposition~\ref{prp:charendotre}, and set   $\omega_B\cdot \omega_B=2a$. 
Then the following equalities hold:
\begin{equation}\label{ciunoquarta}
\int_{K_2(A)} \ch_1(\cE(\cL))^4=2304 a^2-1728 a+324,
\end{equation}
\begin{equation}\label{ciunosecondacidue}
\int_{K_2(A)} \ch_1(\cE(\cL))^2\cdot \ch_2(\cE(\cL))=576 a^2-540 a+81,
\end{equation}
\begin{equation}\label{gianni}
\int_{K_2(A)} \ch_1(\cE(\cL))\cdot \ch_3(\cE(\cL))=24a^2-45a+\frac{27}{2},
\end{equation}
\begin{equation}\label{cidueseconda}
\int_{K_2(A)} \ch_2(\cE(\cL))^2=36a^2-54a+27
\end{equation}
and
\begin{equation}\label{terenzi}
\int_{K_2(A)} \ch_4(\cE(\cL))=\frac{3}{2} a^2-\frac{9}{2} a+\frac{9}{4}.
\end{equation}
\end{prp}
\begin{proof}
Let $\cE=\cE(\cL)$. Let
\begin{equation}\label{eccomegaA}
\omega_A:=f_{*}(\omega_B).
\end{equation}
Notice that $f^{*}(\omega_A)=2\omega_B$ because the covering involution of $f\colon B\to A$ acts trivially on cohomology. It follows that
\begin{equation}\label{raddoppia}
\int_A\omega_A^2=4a.
\end{equation}
By the equality in~\eqref{eccociuno}, we have $\ch_1(\cE)=2\mu_A(\omega_A)-\delta(A)$, and hence 
\begin{equation*}
\int_{K_2(A)} \ch_1(\cE)^4=\int_{K_2(A)}\left(2\mu_A(\omega_A)-\delta(A)\right)^4=2304 a^2-1728 a+324.
\end{equation*}
This proves the equality in~\eqref{ciunoquarta}.

By the equality in~\eqref{eccochardue}, we have
\begin{multline*}
\int_{K_2(A)} \ch_1(\cE)^2\cdot \ch_2(\cE)=\frac{1}{8}\int_{K_2(A)} \left(2\mu_A(\omega_A)-\delta(A)\right)^4- \\
-\frac{1}{8}\int_{K_2(A)}\left(2\mu_A(\omega_A)-\delta(A)\right)^2\cdot c_2(K_2(A))=576 a^2-540 a+81.
\end{multline*}
This proves proves the equality in~\eqref{ciunosecondacidue}. 

Let us prove the equality in~\eqref{gianni}. By Proposition~\ref{prp:loclibero} and the GRR theorem,  we have
\begin{equation}
\ch(\cE)\cdot \td(K_2(A))=\wt{\rho}_{*}\left(e^{c_1(\cL)}\cdot \td(X)\right).
\end{equation}
Multiplying both sides by $\ch_1(\cE)$ and integrating,  we get that
\begin{multline}\label{longtodd}
\int_{K_2(A)}\ch_1(\cE)\cdot \ch_3(\cE)=-\frac{1}{12}\int_{K_2(A)}c_2(K_2(A))\cdot\ch_1(\cE)^2+\\
+\int_X\left(\td_3(X)+\td_2(X)\cdot c_1(\cL)+
\frac{1}{2}\td_1(X)\cdot c_1(\cL)^2+\frac{1}{6} c_1(\cL)^3\right)\cdot \wt{\rho}^{*}\ch_1(\cE).
\end{multline}
Lemmas~\ref{lmm:toddtre} and~\ref{lmm:sollciuno}, together with the results of Section~\ref{sec:conteduca}, allow us to evaluate all the integrals appearing on the right-hand side of~\eqref{longtodd}. In fact, by the equality in~\eqref{intcidue}, we have
\begin{equation}\label{ninfasapienza}
-\frac{1}{12}\int_{K_2(A)}c_2(K_2(A))\cdot\ch_1(\cE)^2=-72a+27.
\end{equation}
By Lemmas~\ref{lmm:toddtre} and~\ref{lmm:sollciuno}, we have
\begin{multline}\label{eq411}
\int_X\td_3(X)\cdot \wt{\rho}^{*}\ch_1(\cE)=-\frac{1}{6}\int_X \cl(D)\cdot\nu^{*}(c_2(K_2(B))\cdot\mu_B(\omega_B))\;+\\
+\frac{1}{24}\int_X \cl(D)\cdot\nu^{*}(c_2(K_2(B))\cdot\delta(B))+\frac{1}{24}\int_X \cl(D)^2\cdot\nu^{*}(c_2(K_2(B))\;-\\
-\frac{1}{6}\int_X i_{*}(c_1(\lambda)\cdot c_1(\xi))\cdot \nu^{*}(\mu_B(\omega_B))+
\frac{1}{24}\int_X  i_{*}(c_1(\lambda)\cdot c_1(\xi))\cdot \nu^{*}(\delta(B))\;+\\
+\frac{1}{24}\int_X  i_{*}(c_1(\lambda)\cdot c_1(\xi))\cdot \cl(D).
\end{multline}
The first two integrals in the right-hand side of the above equality vanish for dimension reasons. The equality in~\eqref{ciduesuvu} gives that
\begin{equation}
\frac{1}{24}\int_X \cl(D)^2\cdot\nu^{*}(c_2(K_2(B))=-\frac{3^5}{24}.
\end{equation}
By the exact sequence in~\eqref{lambxi}, we have $c_1(\lambda)\cdot c_1(\xi)=\nu_D^{*}(c_2(\cN_{V(f)/K_2(B)}))$. It follows that 
the fourth and fifth integrals in the right-hand side of \eqref{eq411} vanish, and by the equality in~\eqref{ciduenorm}, we also get  that 
\begin{equation}
\frac{1}{24}\int_X  i_{*}(c_1(\lambda)\cdot c_1(\xi))\cdot \cl(D)=-\frac{3^4}{24}.
\end{equation}
Summing up, one gets that
\begin{equation}\label{integrotoddtre}
\int_X\td_3(X)\cdot \wt{\rho}^{*}\ch_1(\cE)=-\frac{27}{2}.
\end{equation}
Next we claim that
\begin{equation}\label{integrotodddue}
\int_X\td_2(X)\cdot c_1(\cL)\cdot \wt{\rho}^{*}\ch_1(\cE)=36a.
\end{equation}
In fact, by the equalities in~\eqref{ciunoix} and in~\eqref{miglioramica}, we have
\begin{equation}\label{eq416}
\td_2(X)=\frac{1}{12}\left(\nu^{*}c_2(K_2(B))+i_{*}(c_1(\lambda)+c_1(\xi))\right).
\end{equation}
By the exact sequence in~\eqref{lambxi} and the equality in~\eqref{ciunonorm}, we may rewrite \eqref{eq416} as
\begin{equation}
\td_2(X)=\frac{1}{12}\left(\nu^{*}c_2(K_2(B))+i_{*}(\nu_D^{*}\delta(B))\right).
\end{equation}
As is easily checked (use the equality in~\eqref{ciduequadro}), we have
\begin{equation}\label{semintuno}
\frac{1}{12}\int_X \nu^{*}c_2(K_2(B))\cdot c_1(\cL)\cdot \wt{\rho}^{*}\ch_1(\cE)=36a.
\end{equation}
On the other hand, we have
\begin{equation}\label{semintdue}
\frac{1}{12}\int_X i_{*}(\nu_D^{*}\delta(B))\cdot c_1(\cL)\cdot \wt{\rho}^{*}\ch_1(\cE)=0.
\end{equation}
(In order to get \eqref{semintdue}, one needs to know that $\int_{V(f)}\mu_B(\omega_B)\cdot\delta(B)=0$, which follows from 
Corollary~\ref{diciotto}.) The equality in~\eqref{integrotodddue} follows at once from the equalities in~\eqref{semintuno} and in~\eqref{semintdue}.
Lastly, one checks easily that 
\begin{equation}\label{integrotodduno}
\frac{1}{2}\int_X \td_1(X)\cdot c_1(\cL)^2\cdot \wt{\rho}^{*}\ch_1(\cE)=-9a 
\end{equation}
and
\begin{equation}\label{integrotoddzero}
\frac{1}{6}\int_X c_1(\cL)^3\cdot \wt{\rho}^{*}\ch_1(\cE)=24a^2.
\end{equation}
The equality in~\eqref{gianni} follows at once from the equalities in~\eqref{ninfasapienza}, \eqref{integrotoddtre}, \eqref{integrotodddue}, \eqref{integrotodduno} and~\eqref{integrotoddzero}.

The equality in~\eqref{cidueseconda} is proved by invoking the equality in~\eqref{eccochardue}.

Lastly we prove the equality in~\eqref{terenzi}.
We have  $\cL\cong \nu^{*}(\cL_0)$, where $\cL_0$ is the line bundle on $K_2(B)$ 
such that $c_1(\cL_0)=\mu_B(\omega_B)$. By Proposition~\ref{prp:loclibero}, we have
\begin{equation}\label{caruguali}
\chi(K_2(A),\cE)=\chi(X,\cL)=\chi(K_2(B),\cL_0).
\end{equation}
The Hirzebruch--Riemann--Roch (HRR) formula for HK manifolds of Kummer type  reads
\begin{equation}
\chi(K_2(B),\cL_0)=3\cdot{\frac{q(\cL_0)}{2}+2\choose 2}=3\cdot{ a+2\choose 2}=\frac{3}{2}a^2+\frac{9}{2}a+3.
\end{equation}
On the other hand, the HRR formula for $\cE$ reads
\begin{equation}\label{lamer}
\chi(K_2(A),\cE)=12+\int\ch_2(\cE)\cdot \td_2(K_2(A))+\int\ch_4(\cE).
\end{equation}
 (We denote by $\int$ the integral over $K_2(A)$.)
By the equality in~\eqref{eccochardue}, we have
\begin{multline}
96\int\ch_2(\cE(\cL))\cdot \td_2(K_2(A))=\\
=12\int\left((2\mu_A(\omega_A)-\delta(A))^2-c_2(K_2(A))\right)\cdot \td_2(K_2(A))=\\
=\int\left((2\mu_A(\omega_A)-\delta(A))^2-c_2(K_2(A))\right)\cdot c_2(K_2(A))=24(36a-45).
\end{multline}
(One must invoke the equalities in~\eqref{intcidue} and~\eqref{ciduequadro}.)
The above equality together with the equality in~\eqref{lamer} gives the equality in~\eqref{terenzi}.
\end{proof}
\subsection{Proof of Proposition~\ref{prp:charendotre}}
Let $\cE=\cE(\cL)$, and let $\eta\in H^8(K_2(A);\ZZ)$ be the fundamental class of $K_2(A)$. By the HRR theorem, we have
\begin{multline*}
\chi(K_2(A),\End(\cE)=\int\ch(E^{\vee})\cdot\ch(\cE)\cdot \td(K_2(A))=\\
=48+\frac{1}{12}\int(8\ch_2(\cE)-\ch_1(\cE)^2)\cdot c_2(K_2(A)+\int 8\ch_4(\cE)-2\ch_1(\cE)\cdot \ch_3(\cE)+\ch_2(\cE)^2.
\end{multline*}
 By the equalities in~\eqref{eccochardue} and~\eqref{ciduequadro}, we have
\begin{equation}
\frac{1}{12}\int(8\ch_2(\cE)-\ch_1(\cE)^2)\cdot c_2(K_2(A)=-\frac{1}{12}\int c_2(K_2(A)\cdot c_2(K_2(A)=-63.
\end{equation}
Using  the equalities in~\eqref{terenzi}, \eqref{gianni} and~\eqref{eccochardue}, one gets that 
\begin{equation}
\int 8\ch_4(\cE)-2\ch_1(\cE)\cdot \ch_3(\cE)+\ch_2(\cE)^2=18.
\end{equation}
Adding up one gets the equality in~\eqref{charendotre}

\section{The vector bundle \texorpdfstring{${\cE(\cL)}$}{E(L)} via  Bridgeland--King--Reid}\label{sec:datodabkr}
\subsection{Outline of the section}
The contents of the present section are entirely due to the anonymous referee (if there are  incorrect statements,  I am  responsible for them). 
The notation introduced in Section~\ref{sec:conteduca} is in force throughout the  section. 
The first result is a conceptual proof that if  the class $\omega_B\in\NS(B)$ satisfies a suitable hypothesis, then 
$\cE(\cL)$, for $y=x$ or $y=x+1$, is modular,  locally free,   and moreover all cohomology groups of the vector bundle of its traceless endomorphisms vanish.
This is proved by showing that, via Krugs' version of the Bridgeland--King--Reid (BKR) equivalence, 
 $\cE(\cL)$ corresponds to a simple semi-homogeneous $\cS_{3}$-equivariant vector bundle on $N_A(3)\subset A^{3}$, the kernel of the summation map $A^{3}\to A$. Actually the construction gives many other modular vector bundles on $K_n(A)$, corresponding to isogenies $f\colon B\to A$ with kernels of arbitrary cardinality (satisfying a suitable hypothesis).  
\subsection{The main results}
Let $X_{n+1}(A)$ be the isospectral Hilbert scheme of $n+1$ points on $A$ (see \cite[Definition 3.2.4]{haiman}).
We have a commutative diagram  
\begin{equation}\label{commiso}
\xymatrix{ X_{n+1}(A)\ar[d]_{\wt{\eta}}\ar[rr]^{\wt{\tau}}    &  &  A^{n+1} \ar[d]^{\wt{\pi}}\\ 
  A^{[n+1]}  \ar[rr]^{\wt{\gamma}} & &  A^{(n+1)}\rlap{,}}
\end{equation}
where $\wt{\pi}$ is the quotient map by the action of the symmetric group $\cS_{n+1}$ and $\wt{\gamma}$ is the cycle (or Hilbert-to-Chow) map. Note that  $\cS_{n+1}$ acts on all the spaces in the above diagram (the action on the spaces on the bottom row is trivial) and the maps in the diagram are  
$\cS_{n+1}$-equivariant. Note that  $\cS_{n+1}$  maps $N_A(n+1)$ (the kernel of the summation map $A^{n+1}\to A$) to itself.
Let $X^0_{n+1}(A)\coloneq \wt{\tau}^{-1} N_A(n+1)$. We have the commutative diagram
\begin{equation}\label{ancona}
\xymatrix{ X^0_{n+1}(A)\ar[d]_{\eta}\ar[rr]^{\tau}    &  &  N_A(n+1) \ar[d]^{\pi}\\ 
  K_n(A)  \ar[rr]^{\gamma} & &  N_A(n+1)/\cS_{n+1}\rlap{.}}
\end{equation}
Let $D^{b}_{\cS_{n+1}}(N_A(n+1))$ be the bounded derived category of the (abelian) category of $\cS_{n+1}$-equivariant coherent sheaves on $N_A(n+1)$. Krugs' version (see \cite[Proposition~2.8]{krug:rmks-bkr}) of the BKR equivalence is given by 
the functor
\begin{equation}\label{equivkrug}
\Phi_A(n+1)\coloneq \eta_{*}^{\cS_{n+1}}\circ \tau^{*}\colon D^{b}_{\cS_{n+1}}(N_A(n+1))\lra D^{b}(K_n(A)),
\end{equation}
where $\eta_{*}^{\cS_{n+1}}$ is the derived functor  of  the  functor $\Coh(X_{n+1}^0(A))\to \Coh(K_n(A))$ mapping $\cF$ to the $\cS_{n+1}$-invariant subsheaf of $\eta_{*}(\cF)$. 

Now let $f\colon B\to A$ be an isogeny of abelian surfaces (of arbitrary degree).
 The homomorphism of abelian varieties 
\begin{equation}\label{tanteeffe}
\begin{matrix}
N_B(n+1) & \overset{F}{\lra} & N_A(n+1) \\
(b_1)+\dots+(b_{n+1}) & \longmapsto & (f(b_1))+\dots+(f(b_{n+1}))
\end{matrix}
\end{equation}
is $\cS_{n+1}$-equivariant. 
Let $L$ be a line bundle on $B$. Let $\wt{L}(n+1)$ be the line bundle on $N_B(n+1)$ defined by
\begin{equation}
\wt{L}(n+1)\coloneq \bigotimes_{i=1}^{n+1}\pr_i^{*}L_{|N_B(n+1)}.
\end{equation}
The symmetric group $\cS_{n+1}$ acts on $\wt{L}(n+1)$ by permutations; we call this action $\lambda_{+}$. Twisting the permutation action by the sign representation, we get another action, that we denote by $\lambda_{-}$. Thus we get two $\cS_{n+1}$-equivariant line bundles 
$(\wt{L}(n+1),\lambda_{+})$ and $(\wt{L}(n+1),\lambda_{-})$. Since the homomorphism $F$ in~\eqref{tanteeffe} is $\cS_{n+1}$-equivariant, we have the  $\cS_{n+1}$-equivariant vector bundles $F_{*}(\wt{L}(n+1),\lambda_{\pm})$ on $N_A(n+1)$. Let
\begin{equation}
\cE_{\pm}(L)\coloneq\Phi_A(n+1)\left(F_{*}\left(\wt{L}(n+1),\lambda_{\pm}\right)\right).
\end{equation}
Then $\cE_{\pm}(L)$ is a locally free sheaf. In fact, the map $\eta$ in~\eqref{ancona} is finite,  and moreover it is flat because 
$X^0_{n+1}(A)$ is CM by \cite[Theorem~3.1]{haiman}. The rank of $\cE_{\pm}(L)$ is equal to the degree of $F$, \textit{i.e.}, $|\ker(f)|^n$. 
In order to state the first main result, we let $\varphi_{\wt{L}(n+1)}\colon N_B(n+1)\to N_B(n+1)^{\vee}$ be the homomorphism defined by 
$\varphi_{\wt{L}(n+1)}(x)\coloneq [T_x^{*}\wt{L}(n+1)\otimes \wt{L}(n+1)^{-1}]$, where $T_x\colon N_B(n+1)\to N_B(n+1)$ is translation by $x$.
\begin{thm}\label{thm:vienedabkr}
Let $F$ be the homomorphism in~\eqref{tanteeffe}, and suppose that 
\begin{equation}\label{interzero}
\ker(F)\cap\ker\left(\varphi_{\wt{L}(n+1)}\right)=0.
\end{equation}
Then 
\begin{equation}\label{nocoom}
H^p\left(K_n(A),\End^0\left(\cE_{\pm}(L)\right)\right)=0
\end{equation}
for all $p$, where $\End^0(\cE_{\pm}(L))$ is the sheaf of  traceless endomorphisms of $\cE_{\pm}(L)$.
\end{thm}
The proof of Theorem~\ref{thm:vienedabkr} is in Section~\ref{subsec:eccoleprove}.
\begin{crl}\label{crl:defliscio}
Keep hypotheses as in Theorem~\ref{thm:vienedabkr}. Then  the natural maps 
\begin{equation*}
 \Def(K_2(A),\cE_{\pm}(L))\lra \Def(K_2(A),\det\cE_{\pm}(L))
\end{equation*}
and
\begin{equation}\label{defpro}
\Def(K_2(A),\PP(\cE_{\pm}(L)))\lra \Def(K_2(A))
\end{equation}
are smooth. Moreover,  $\cE_{\pm}(L)$ is a modular vector bundle. 
\end{crl}
\begin{proof}
The first statement is an immediate consequence of the vanishing in~\eqref{nocoom} for $p=2$ and results in deformation theory; see~\cite[Proposition~3.6]{og-rigidi-su-k3n}. By the surjectivity of the map in~\eqref{defpro}, the discriminant $\Delta(\cE_{\pm}(L))$ remains of type $(2,2)$ for all small deformations of $K_2(A)$; see the proof of Corollary~3.7 in \textit{op.~cit.} This implies that $\cE_{\pm}(L)$ is modular.
\end{proof}
Next we give the relation between the vector bundles  $\cE_{\pm}(L)$ and $\cE(\cL)$. Thus we assume that the isogeny $f\colon B\to A$ has degree $2$, with kernel generated by $0\not=\epsilon\in B[2]$.
\begin{prp}\label{prp:propriobkr}
Let $L$ be a line bundle on $B$, and let $\omega_B=c_1(L)\in\NS(B)$. Let $\cL_{+}$, $\cL_{-}$ be the line bundles on $X$ $($notation as in Section~\ref{sec:conteduca}\,$)$ that one gets by setting $x=y=0$ and $x=-1$, $y=0$, respectively,
 in~\eqref{ixipsilon}.  Then $\cE(\cL_{\pm})\cong\cE_{\pm}(L)$.
\end{prp}
The proof of Proposition~\ref{prp:propriobkr} is sketched in Section~\ref{subsec:eccoleprove}. Let 
$\varphi_{L}\colon B\to B^{\vee}$ be the homomorphism defined by 
$\varphi_{L}(b)\coloneq [T_b^{*}L\otimes L^{-1}]$.
\begin{crl}\label{crl:propriobkr}
Keep notation and hypotheses as  in Proposition~\ref{prp:propriobkr}. If $\epsilon\notin\ker \varphi_{L}$, then 
\[
H^p\left(K_n(A),\End^0(\cE(\cL_{\pm}))\right)=0
\]
for all $p$,  $\cE(\cL_{\pm})$ is a modular vector bundle, and the natural map 
\begin{equation*}
 \Def(K_2(A),\cE(\cL_{\pm}))\lra \Def(K_2(A),\det\cE(\cL_{\pm}))
\end{equation*}
is smooth.
\end{crl}
\begin{proof}
By Proposition~\ref{prp:propriobkr} and Corollary~\ref{crl:defliscio}, it suffices to prove that the equality in~\eqref{interzero} holds.
Let $F\colon N_B(3)\to N_A(3)$ be as in~\eqref{tanteeffe}. The nonzero elements of $\ker F$ are $(\epsilon,\epsilon,0)$ and its permutations. By symmetry, it suffices to prove that 
\begin{equation}\label{bastaquesto}
T^{*}_{(\epsilon,\epsilon,0)}\wt{L}(3)\otimes \wt{L}(3)^{-1}\not\cong\cO_{N_B(3)}.
\end{equation}
Let $\wh{L}(3)$ be the line bundle on $B^3$ given by $\wh{L}(3)\coloneq \bigotimes_{i=1}^{3}\pr_i^{*}L$. Thus $\wt{L}(3)=\wh{L}(3)_{|N_B(3)}$, and it suffices to prove that
\begin{equation}\label{dadimostrare}
T^{*}_{(\epsilon,\epsilon,0)}\wh{L}(3)\otimes \wh{L}(3)^{-1}\notin\ker\left(\Pic^0(B^3)\lra \Pic^0(N_B(3))\right).
\end{equation}
(Abusing notation, we denote by $T_{(\epsilon,\epsilon,0)}$   both the translation of $N_B(3)$ and that of $B^3$.)
The kernel in the right-hand side of~\eqref{dadimostrare} is equal to $\sigma_3^{*}(\Pic^0(B))$, where $\sigma_3\colon B^3\to B$ is the summation map. One checks~\eqref{dadimostrare}  by restricting $T^{*}_{(\epsilon,\epsilon,0)}\wh{L}(3)\otimes \wh{L}(3)^{-1}$ to $B\times\{(0,0)\}$ and $\{(0,0)\}\times B$ (both identified with $B$). The first restriction is isomorphic to $T_b^{*}L\otimes L^{-1}$, which is nontrivial by hypothesis, while the second restriction is trivial. This shows that the left-hand side of ~\eqref{dadimostrare} does not belong to 
$\sigma_3^{*}(\Pic^0(B))$. This proves that~\eqref{bastaquesto} holds.
\end{proof}
\subsection{Proofs of the main results}\label{subsec:eccoleprove}
\begin{proof}[Proof of Theorem~\ref{thm:vienedabkr}]
By the equality in~\eqref{interzero}, the vector bundle $F_{*}(\wt{L}(n+1))$ is simple semi-homogeneous; see \cite[Propositions 5.4 and 5.6]{muksemi}. By Proposition 5.9 in \textit{op.~cit.}, it follows that for all $p$, we have an isomorphism
\begin{equation}\label{jacobs}
H^p(\Tr)\colon H^p\left(A^{n+1},\End\left(F_{*}\left(\wt{L}(n+1)\right)\right)\right)\overset{\lowsim}{\lra} H^p\left(A^{n+1},\cO_{A^{n+1}}\right).
\end{equation}
Hence we get a series of isomorphisms
\begin{multline*}
\Ext^p_{K_n(A)}(\cE_{+}(L),\cE_{+}(L))\cong 
\Ext^p_{D^b_{\cS_{n+1}}(N_A(n+1))}\left(F_{*}\left(\wt{L}(n+1),\lambda_{+}\right),F_{*}\left(\wt{L}(n+1),\lambda_{+}\right)\right)\cong \\
\cong \Ext^p_{N_A(n+1)}\left(F_{*}\left(\wt{L}(n+1)\right),F_{*}\left(\wt{L}(n+1)\right)\right)^{\cS_{n+1}}\cong  \\ 
\cong H^p\left(N_A(n+1),\cO_{N_A(n+1)}\right)^{\cS_{n+1}}\cong
\begin{cases}
\CC & \text{if $p\in\{0,2,4,\ldots,2n\}$,}\\
0 & \text{otherwise.}
\end{cases}
\end{multline*}
In fact, the first isomorphism holds because  of the BKR equivalence in~\eqref{equivkrug}, the second isomorphism holds because $\lambda_{+}$ is the permutation representation, the third isomorphism holds because of the isomorphism in~\eqref{jacobs}, and the last isomorphism holds because by the BKR equivalence in~\eqref{equivkrug}, the left-hand side (of the isomorphism) is isomorphic to
$H^p(K_n(A), \cO_{K_n(A)})$. This proves the equality in~\eqref{nocoom} for $\cE_{+}(L)$. To prove that the equality also holds  for 
$\cE_{-}(L)$, it suffices to show that
\begin{multline}\label{pavane}
\Ext^p_{D^b_{\cS_{n+1}}(N_A(n+1))}\left(F_{*}\left(\wt{L}(n+1),\lambda_{+}\right),F_{*}\left(\wt{L}(n+1),\lambda_{+}\right)\right)\cong \\
\cong\Ext^p_{D^b_{\cS_{n+1}}(N_A(n+1))}\left(F_{*}\left(\wt{L}(n+1),\lambda_{-}\right),F_{*}\left(\wt{L}(n+1),\lambda_{-}\right)\right).
\end{multline}
Recall that $\lambda_{-}$ is obtained by tensoring $\lambda_{+}$ by the sign representation $\chi$. The isomorphism in~\eqref{pavane} holds because tensorization by $\chi$ commutes with the $\cS_{n+1}$-equivariant morphism $F_{*}$, and moreover tensorization by $\chi$ is an autoequivalence of 
$D^b_{\cS_{n+1}}(N_A(n+1))$.
\end{proof}
Next we sketch how one proves Proposition~\ref{prp:propriobkr}. Let $\Phi_B(n+1)$ be the equivalence of categories that one gets upon replacing $A$ by $B$ in~\eqref{equivkrug}. We have the functor
\begin{equation}
\Xi\coloneq \Phi_A(3)\circ F_{*}\circ \Phi_B(3)^{-1}\colon D^b(K_2(B))\lra D^b(K_2(A)).
\end{equation}
The key result is the following. 
\begin{prp}\label{prp:titmus}
The kernel of\, $\Xi$ is isomorphic to $(\wt{\rho},\nu)_{*}\cO_X$ $($notation as in~\eqref{allevavisoni}$)$.
\end{prp}
We show how Proposition~\ref{prp:propriobkr} follows from Proposition~\ref{prp:titmus}, and after that we indicate the key elements
 that go into the proof of Proposition~\ref{prp:titmus}.  
\begin{proof}[Proof of Proposition~\ref{prp:propriobkr}  granting Proposition~\ref{prp:titmus}]
We have
\begin{equation*}
\Phi_B(n+1)^{-1}\left(\wt{L}(n+1),\lambda_{+}\right)=\mu_B(\omega_B),\quad 
\Phi_B(n+1)^{-1}\left(\wt{L}(n+1),\lambda_{-}\right)=\mu_B(\omega_B)-\delta(B).
\end{equation*}
Hence $\cE_{+}(L)=\Xi(\mu_B(\omega_B))$ and $\cE_{-}(L)=\Xi(\mu_B(\omega_B)-\delta(B))$. By Proposition~\ref{prp:titmus}, the former 
is isomorphic to $\cE(\cL_{+})$ and the latter is isomorphic to $\cE(\cL_{-})$.
\end{proof}
\begin{rmk}\label{rmk:comeclust}
In what follows, we identify $K_2(A)$ (and $K_2(B)$) with the parameter space for $\cS_3$-clusters in $N_A(3)$ (respectively, clusters in $N_B(3)$). If $[Z]\in K_2(A)$, we let $\ov{Z}\subset N_A(3)$ be the corresponding (length $6$) $\cS_3$-cluster, and similarly for $[W]\in K_2(B)$. 
\end{rmk}
\begin{lmm}\label{lmm:homclust}
Let $\ov{W}$ be an $\cS_3$-cluster in $N_B(3)$,  and let 
$\lambda_{\ov{W}}$ be the $\cS_3$-linearization of its structure sheaf. 
Let $\ov{Z}$ be an $\cS_3$-cluster in $N_A(3)$, 
let $\ov{Y}\coloneq F^{-1}(\ov{Z})$, and let 
$\lambda_{\ov{Y}}$ be the $\cS_3$-linearization of its structure sheaf. 
\begin{enumerate}
\item\label{l:h-1}
If\, $\ov{W}$ is a subscheme of\, $\ov{Y}$, then 
$\Hom_{D^b_{\cS_3}(N_B(3)}((\cO_{\ov{Y}},\lambda_{\ov{Y}}),(\cO_{\ov{W}},\lambda_{\ov{W}}))\cong\CC$. 
\item\label{l:h-2}
If\, $\ov{W}$ is not a subscheme of\, $\ov{Y}$, then $\Ext^p_{D^b_{\cS_3}(N_B(3)}((\cO_{\ov{Y}},\lambda_{\ov{Y}}),(\cO_{\ov{W}},\lambda_{\ov{W}}))=0$  for all $p$. 
\end{enumerate}
\end{lmm}
\begin{proof}
If $\ov{W}$ is a subscheme of $\ov{Y}$, then 
$\Hom_{D^b_{\cS_3}(N_B(3)}((\cO_{\ov{Y}},\lambda_{\ov{Y}}),(\cO_{\ov{W}},\lambda_{\ov{W}}))$ is nonzero. 
Suppose that $h\in \Hom_{D^b_{\cS_3}(N_B(3)}((\cO_{\ov{Y}},\lambda_{\ov{Y}}),(\cO_{\ov{W}},\lambda_{\ov{W}}))$ is nonzero. Then $h(1)\not=0$ because 
$1$ generates the stalk of $\cO_{\ov{Y}}$ at every point (as a module over $\cO_{N_B(3)}$). Since $h(1)\in H^0(\ov{W},\cO_{\ov{W}})^{\gS_3}$, which is $1$-dimensional and generated by the constant function $1$ (because $\ov{W}$ is an $\cS_3$-cluster), we get that $h(1)$ 
generates the stalk of $\cO_{\ov{W}}$ at each point in the support of $\ov{W}$. It follows that $h$ is surjective and that we may rescale $h$ so that $h(1)=1$. But then $h$ is a surjective homomorphism of sheaves of $\cO_{N_B(3)}$-algebrae, and hence $\ov{W}$ is a subscheme of $\ov{Y}$. Since $h$ is determined by $h(1)$, we also get that $\Hom_{D^b_{\cS_3}(N_B(3)}((\cO_{\ov{Y}},\lambda_{\ov{Y}}),(\cO_{\ov{W}},\lambda_{\ov{W}}))\cong\CC$. 
This proves item~\eqref{l:h-1}, and it proves the case $p=0$ of item~\eqref{l:h-2}. One proves the rest of item~\eqref{l:h-2} by considering the BKR equivalence $\Phi_B(3)$ which associates to  
$(\cO_{\ov{W}},\lambda_{\ov{W}})$ and 
$(\cO_{\ov{Y}},\lambda_{\ov{Y}})$, respectively, $\cO_{[W]}$, where $[W]\in K_2(B)$, and $\cT$, where $\cT$ is a length $4$ sheaf on $K_2(B)$. 
\end{proof}
Before stating the next lemma, we note that the map $(\nu,\wt{\rho})\colon X\to K_2(B)\times K_2(A)$ defines an isomorphism 
\begin{equation}\label{lyles}
X\overset{\lowsim}{\lra} \left\{([W],[Z])\in K_2(B)\times K_2(A) \mid W\subset f^{-1}(Z)\right\}.
\end{equation}
\begin{lmm}
Keeping in mind~\eqref{lyles} and the  identification in Remark~\ref{rmk:comeclust}, we have the equality
\begin{equation}\label{glistessi}
X= \left\{([W],[Z])\in K_2(B)\times K_2(A) \mid \ov{W}\subset F^{-1}(\ov{Z})\right\}.
\end{equation}
\end{lmm}
\begin{proof}
Let $Y$ be the right-hand side of~\eqref{glistessi}. By Lemma~\ref{lmm:homclust} and the upper semicontinuity of cohomology, $Y$ is closed. Let us prove that $X\subset Y$. 
 Since $X$ is irreducible (it is a blow-up of $K_2(B)$), it suffices to prove that if $([W],[Z])$ is a  general point of $X$, then 
 $\ov{W}\subset F^{-1}(\ov{Z})$.
This is clear because  $W$, $Z$ are reduced and hence $\ov{W}$, $\ov{Z}$ are reduced $\cS_3$-clusters. 

Let us prove that $Y\subset X$. If $V$ is a zero-dimensional scheme and $x$ is a point of $V$, we let $V_{(x)}$ be the connected component of $V$ containing $x$.  Let $([W],[Z])\in Y$, \textit{i.e.}, $\ov{W}\subset F^{-1}(\ov{Z})$. Let $b=(b_1,b_2,b_3)\in\supp\ov{W}$, and let 
$F(b)=a=(a_1,a_2,a_3)$. By hypothesis, $W_{(b)}\subset F^{-1}(Z_{(a)})$, and since $F$ is \'etale, it follows that 
$F(W_{(b)})\subset Z_{(a)}$. Let $\wt{p}_i\colon N_B(3)\to B$ be the restriction of the $\supth{i}$ projection $B^3\to B$, and define similarly
$p_i\colon N_A(3)\to A$. Haiman proved that $p_i(\ov{Z})=Z$ as schemes (and of course $\wt{p}_i(\ov{W})=W$); see the discussion preceding Proposition~3.2 in~\cite{krug:rmks-bkr}. Hence we get that
\begin{equation}
f\left(W_{(b_1)}\right)=f\left(\wt{p}_1\left(\ov{W}_{(b)}\right)\right)=p_1\left(F\left(\ov{W}_{(b)}\right)\right.\subset p_1\left(\ov{Z}_{(a)}\right)=Z_{(a_1)}.
\end{equation}
Since $f$ is \'etale, it follows that $W$ is a subscheme of $f^{-1}(Z)$, \textit{i.e.}, $([W],[Z])\in X$. This proves that $Y\subset X$.
\end{proof}
\begin{proof}[Sketch of proof of Proposition~\ref{prp:titmus}]
The Fourier--Mukai kernel of the derived push-forward 
$F_{*}\colon D^{b}_{\cS_{n+1}}(N_B(3))\to D^{b}_{\cS_{n+1}}(N_A(3))$ is given by $(\cO_{\Gamma(F)},\gamma)$, where $\Gamma(F)$ is the graph of $F$ and $\gamma$ is the obvious $\cS_3$ linearization. The Fourier--Mukai kernel of $\Phi_A(3)$ is
 given by $(\cO_{X^0_3(A)},\alpha)$,  
where $\alpha$ is the natural $\cS_3$ linearization. The Fourier--Mukai kernel of the composition $\Phi_A(3)\circ F_{*}$ is given by
\begin{equation*}
(F\times\Id)^{*}\left(\cO_{X^0_3(A)},\alpha\right)\in D^{b}_{\cS_{3}}(N_B(3)\times K_2(A)).
\end{equation*}
The Fourier--Mukai kernel of $\Phi_B(3)^{-1}$ is given by $(\cO_{X^0_3(B)},\beta)^{\vee}[4]$, where $(\cO_{X^0_3(B)},\beta)^{\vee}$ is the derived dual of $(\cO_{X^0_3(B)},\beta)\in D^{b}_{\cS_{3}}(N_B(3)\times K_2(B))$, where the latter is the analogue of $(\cO_{X^0_3(A)},\alpha)$. It follows that 
the Fourier--Mukai kernel of $\Xi$ has $\supth{p}$ sheaf cohomology given by
\begin{equation}\label{ultimaformula}
\cE xt^{p+4}_{p_{1,3}}\left(p_{2,3}^{*}\left(\cO_{X^0_3(B)},\beta\right),p_{1,2}^{*}(F\times\Id)^{*}\left(\cO_{X^0_3(A)},\alpha\right)\right)\cong 
\cE xt^{p+4}_{p_{1,3}}\left(p_{2,3}^{*}\cO_{X^0_3(B)},p_{1,2}^{*}(F\times\Id)^{*}\cO_{X^0_3(A)}\right)^{\cS_3},
\end{equation}
where $p_{i,j}$ is the projection of $K_2(A)\times N_B\times K_2(B)$ to the product of the $\supth{i}$ and $\supth{j}$ factors. If $\ov{W}$ and $\ov{Y}$ are 
$\cS_3$-equivariant sheaves  on $N_B$, then the dual of $\Ext^q_{N_B(3)}(\cO_{\ov{W}},\cO_{\ov{Y}})^{\cS_3}$ is identified with 
$\Ext^{4-q}_{N_B(3)}(\cO_{\ov{Y}},\cO_{\ov{W}})^{\cS_3}$. Using this, Lemma~\ref{lmm:homclust}, and cohomology and base change,
one proves that the sheaf
 in~\eqref{ultimaformula} vanishes for $p\not=0$ and is isomorphic to $\cO_X$ for $p=0$ (see for example~\cite[Proposition~2.26]{mukvb}). 
\end{proof}
\section{The case when \texorpdfstring{$\boldsymbol{K_2(A)}$}{K\textunderscore 2(A)} is a Lagrangian fibration}\label{sec:analisifine}
\subsection{Set-up}\label{subsec:assicurazione}
\begin{hyp-dfn}\label{hyp:biellittica}\leavevmode
\begin{enumerate}
\item
We assume that $B$ is an abelian surface containing   a (\textit{bona fide}) elliptic curve $C_B\subset B$. We denote  by 
$\gamma_B$ the Poincar\`e dual   of $C_B$.
\item
Let $\ov{\omega}_B\in \NS(B)$ be   a class such that the  subgroup $(\ZZ\ov{\omega}_B+\ZZ\gamma_B)<\NS(B)$ is  saturated. Let 
\begin{equation}\label{omomandomga}
\ov{\omega}_B\cdot\ov{\omega}_B=2\ov{a},\quad \ov{\omega}_B\cdot\gamma_B=d.
\end{equation}
\end{enumerate}
\end{hyp-dfn}
Assume that Hypothesis~\ref{hyp:biellittica} holds, and 
choose a nonzero $\epsilon\in C_B[2]$.  Let 
\begin{equation}
B\overset{f}{\lra} A:=B/\la\epsilon\ra
\end{equation}
 be the quotient map. 
Thus we are in the set-up of 
Section~\ref{subsec:quandomod}. 
Let 
\begin{equation}\label{eccofibi}
B\overset{\varphi_B}{\lra} E:=B/C_B
\end{equation}
be the quotient map. 
 Then  $\varphi_B$ descends to a homomorphism
\begin{equation}\label{mappaphi}
A\overset{\varphi_A}{\lra} E
\end{equation}
with kernel the quotient $C_A:=C_B/\la\epsilon\ra$.  The homomorphisms $\varphi_B$ and $\varphi_A$  induce Lagrangian fibrations 
 \begin{equation*}
\begin{matrix}
 K_2(B) & \overset{\pi_B}{\lra} &  |\cO_{E}(3(0_{E}))|\\
 [Z] & \longmapsto & \varphi_B(|Z|), 
\end{matrix}
\qquad \qquad
\begin{matrix}
 K_2(A) & \overset{\pi_A}{\lra} &  |\cO_{E}(3(0_{E}))| \\
 [Z] & \longmapsto & \varphi_A(|Z|), 
\end{matrix}
\end{equation*}
where $|Z|$ is the cycle associated to the scheme $Z$. The commutative diagram
in~\eqref{allevavisoni} extends to the following commutative diagram:
\begin{equation}
\xymatrix{   & X \ar[dl]_{\nu}  \ar[dr]^{\wt{\rho}} &   \\ 
  K_2(B)  \ar@{-->}[rr]^{\rho} \ar[dr]^{\pi_B} & & K_2(A) \ar[dl]^{\pi_A} \\
  &  |\cO_{E}(3(0_{E}))|\rlap{.}  & }
\end{equation}
We recall the following fact.
\begin{prp}\label{prp:fiblaglis}
Let $D_0\in  |\cO_{E}(3(0_{E}))|$. Then $\pi_A^{-1}(D_0)$ is smooth if and only if\, $D_0$ is reduced. 
\end{prp}
In other words, $\pi_A^{-1}(D_0)$ is a singular Lagrangian if and only if $D_0$ is a point of  the sextic curve $E^{\vee}\subset  |\cO_{E}(3(0_{E}))|$  parametrizing nonreduced divisors, \textit{i.e.}, the dual of the smooth cubic  $E\subset  |\cO_{E}(3(0_{E}))|^{\vee}$.

Let 
 $\cL$ be the line bundle on $X$ such that  
 \begin{equation}\label{nostraelle}
c_1(\cL)=\nu^{*}(\mu_B(m\ov{\omega}_B)).
\end{equation}
In other words, in~\eqref{ixipsilon},   $\omega_B=m\ov{\omega}_B$ and $x=y=0$.  We state a consequence of the results in Section~\ref{sec:datodabkr}.
\begin{prp}\label{prp:harbridge}
Assume that Hypothesis~\ref{hyp:biellittica} and Equation~\eqref{nostraelle} hold and that  $md$ is odd $(d$  as in~\eqref{omomandomga} and $m$  as in~\eqref{nostraelle}$)$. 
Then for all $p$, we have
\begin{equation*}
H^p\left(K_2(A),\End^0(\cE(\cL))\right)=0.
\end{equation*}
The natural map 
\begin{equation*}
 \Def(K_2(A),\cE(\cL))\lra \Def(K_2(A),\det\cE(\cL))
\end{equation*}
is smooth.
\end{prp}
\begin{proof}
Since $md$ is odd, the hypothesis of Corollary~\ref{crl:propriobkr} is satisfied, and hence Proposition~\ref{prp:harbridge} follows from 
Corollary~\ref{crl:propriobkr}.
\end{proof}
At the end of the present section, we give another (pedestrian) proof of Proposition~\ref{prp:harbridge}.
\subsection{Main results}\label{subsec:scenario}
In the present section, we are concerned with the restriction of $\cE(\cL)$ to Lagrangian fibers $\pi_A^{-1}(D)$ where $D\in  |\cO_{E}(3(0_{E}))|$.
The results that we prove will be needed in the proof of Theorem~\ref{thm:unicita}.

Let $D_0\in  |\cO_{E}(3(0_{E}))|$. If the schematic fiber
 $\pi_A^{-1}(D_0)$  is smooth, then the image of the restriction map $H^2(K_2(A);\ZZ)\to H^2(\pi_A^{-1}(D_0);\ZZ)$ has rank $1$, it is contained in $\NS(\pi_A^{-1}(D_0))$ and its saturation is generated by  an ample 
class $\theta\in \NS(\pi_A^{-1}(D_0))$ with elementary divisors $(1,3)$; see~\cite{wieneck1}.  If $\cF$ is a sheaf on 
$\pi_A^{-1}(D_0)$, then slope (semi)stability of 
$\cF$  always refers to $\theta$ slope (semi)stability. Below is the first main result of the present section.
\begin{prp}\label{prp:stabonlagr}
Assume that Hypothesis~\ref{hyp:biellittica}  and Equation~\eqref{nostraelle} hold and
   that   $md$  is odd $(d$ is as in~\eqref{omomandomga}  and $m$ is as in~\eqref{nostraelle}$)$. 
 Then the restriction of  $\cE(\cL)$  to a smooth fiber of the Lagrangian fibration $\pi_A$ is slope stable.
\end{prp}
Proposition~\ref{prp:stabonlagr} is proved in Section~\ref{subsec:reslagliscia}.
The second main result of the present section is the following.
\begin{prp}\label{prp:sempfibre}
Assume that Hypothesis~\ref{hyp:biellittica} and Equation~\eqref{nostraelle} hold and that  $md$ is odd $(d$  as in~\eqref{omomandomga} and $m$  as in~\eqref{nostraelle}$)$. 
Also suppose  that  $D_0\in E^{\vee}$ is not an inflection divisor. 
Then the restriction of  $\cE(\cL)$  to  $\pi_A^{-1}(D_0)$ is  a simple sheaf.
\end{prp}
Proposition~\ref{prp:sempfibre} is proved in Section~\ref{subsec:demimoore}.
\subsection{Restriction  of \texorpdfstring{${\cE(\cL)}$}{E(L)}  to a smooth Lagrangian fiber}\label{subsec:reslagliscia}
  Before proving Proposition~\ref{prp:stabonlagr}  we  go through a series of preliminary results. Let $\gamma_A\in \NS(A)$ be  the Poincar\'e dual of  $C_A$.  
Since the restriction of $f$ to $C_B$ defines the degree $2$ quotient map $C_B\to C_A$, we have the equalities 
\begin{equation}
f_{*}(\gamma_B)=2\gamma_A,\quad f^{*}(\gamma_A)=\gamma_B.
\end{equation}
\begin{prp}\label{prp:satollo}
Suppose that Hypothesis~\ref{hyp:biellittica} holds and that $d$ is odd,  where $d$ is as in~\eqref{omomandomga}. Let 
$\ov{\omega}_A:=f_{*}(\ov{\omega}_B)$. Then $\ZZ\ov{\omega}_A+\ZZ\gamma_A$ is a saturated subgroup of 
$\NS(A)$, and 
\begin{equation}\label{omegamdi}
\ov{\omega}_A\cdot \gamma_A=d, \quad \ov{\omega}_A\cdot \omega_A=4\ov{a}.
\end{equation}
Moreover, $\ov{\omega}_A$ has elementary divisors $(1,2\ov{a})$.
\end{prp}
\begin{proof}
Suppose that $\ZZ\ov{\omega}_A+\ZZ\gamma_A$ is not a saturated subgroup of $\NS(A)$. By hypothesis,   
$\ZZ\ov{\omega}_B+\ZZ\gamma_B$ is a saturated subgroup of $\NS(B)$; since 
$f^{*}(\ov{\omega}_A)=2\ov{\omega}_B$ and $f^{*}(\gamma_A)=\gamma_B$,  it follows that $\ZZ\ov{\omega}_A+\ZZ\gamma_A$ is of index $2$ in its saturation. This gives a contradiction because the discriminant of the intersection form on $\ZZ\ov{\omega}_A+\ZZ\gamma_A$ is equal to $-d^2$, which is odd.
The first equality in~\eqref{omegamdi} follows from
\begin{equation*}
2\int_B \ov{\omega}_B\cdot\gamma_B=\int_B f^{*}(\ov{\omega}_A)\cdot f^{*}(\gamma_A)=
2\int_A \ov{\omega}_A\cdot \gamma_A,
\end{equation*}
and the second one follows from a similar  computation. Lastly, $\ov{\omega}_A$ has elementary divisors $(1,2\ov{a})$ because it is primitive and $\ov{\omega}_A\cdot \ov{\omega}_A=4\ov{a}$.
\end{proof}
In order to analyze the smooth fibers of $\pi_A$ (or $\pi_B$), we  introduce the following  notation: If $Y$ is a (\textit{bona fide}) elliptic curve, we let 
\begin{equation}\label{zedenc}
Z_n(Y):=\left\{(y_1,\ldots,y_{n+1})\in Y^{n+1} \mid y_1+\dots+y_{n+1}=0\right\}.
\end{equation}
Let   $D_0\in  |\cO_{E}(3(0_{E}))|$ be reduced, given by $D_0=(x_1)+(x_2)+(x_3)$. Let $C_{B,i}:=\varphi_B^{-1}(x_i)$, $C_{A,i}:=\varphi_A^{-1}(x_i)$, where $\varphi_B$ and $\varphi_A$ are as in Hypothesis~\ref{hyp:biellittica} and~\eqref{mappaphi}, respectively. For $i\in\{1,2\}$, choose $b_1,b_2\in B$ such that $T_{b_i}(C_B)=C_{B,i}$, where $T_{b_i}\colon B\to B$ is the translation by $b_i$. Let $a_i:=f(b_i)$, so that $T_{a_i}(C_A)=C_{A,i}$. We have  isomorphisms
\begin{equation}\label{fibrapitilde}
\begin{matrix}
Z_2(C_B) & \overset{\lowsim}{\lra} & \pi_B^{-1}((x_1)+(x_2)+(x_3)) \\
(y_1)+(y_2)+(y_3) & \longmapsto & (y_1+b_1)+(y_2+b_2)+(y_3-b_1-b_2)
\end{matrix}
\end{equation}
and
\begin{equation}\label{fibrapisecco}
\begin{matrix}
Z_2(C_A) & \overset{\lowsim}{\lra} & \pi_A^{-1}((x_1)+(x_2)+(x_3)) \\
(y_1)+(y_2)+(y_3) & \longmapsto & (y_1+a_1)+(y_2+a_2)+(y_3-a_1-a_2). 
\end{matrix}
\end{equation}
(Note: Since elements of $\pi_A^{-1}((x_1)+(x_2)+(x_3))$ are reduced subschemes, we identify them with the corresponding $0$-cycles.)
\begin{prp}\label{prp:restsimphom}
Assume that Hypothesis~\ref{hyp:biellittica} and Equation~\eqref{nostraelle} hold and that  $md$ is odd $(d$  as in~\eqref{omomandomga} and $m$  as in~\eqref{nostraelle}$)$. 
 Then the restriction of  $\cE(\cL)$  to a smooth fiber of  $\pi$ is a simple semi-homogeneous vector bundle.
\end{prp}
\begin{proof}
Let  $\pi_A^{-1}(D_0)$ be a smooth fiber of $\pi_A$. Then $\pi^{-1}_A(D_0)$ does not meet 
$\Delta(A)$, and hence the last equation in~\eqref{solotre} gives that $\wt{\rho}^{-1}(\pi_A^{-1}(D_0))$ does not meet the exceptional divisor  of $\nu\colon X\to K_2(B)$. (Watch out: The exceptional divisor $D\subset X$ and the divisor $D_0\in  |\cO_{E}(3(0_{E}))|$ are denoted by similar letters although they are unrelated.) It follows that $\wt{\rho}^{-1}(\pi_A^{-1}(D_0))$ is equal to $\pi_B^{-1}(D_0)$. Moreover,  the \'etale map
\begin{equation*}
\pi_B^{-1}(D_0)=\wt{\rho}^{-1}\left(\pi_A^{-1}(D_0)\right) \lra \pi_A^{-1}(D_0)
\end{equation*}
is identified, given the isomorphisms in~\eqref{fibrapitilde} and~\eqref{fibrapisecco}, with the \'etale map
\begin{equation}\label{zetadueci}
\begin{matrix}
 Z_2(C_B) & \overset{\psi^2}{\lra} &  Z_2(C_A) \\
 (z_1,\ldots,z_{3}) & \longmapsto &   (f(z_1),\ldots,f(z_{3})).  
\end{matrix}
\end{equation}
We claim that the thesis of the proposition follows from Proposition~\ref{prp:semplice}; in fact,  $\deg (\psi^2)=4$ and the degree $d_0$ of Proposition~\ref{prp:semplice} is equal to $md$, which is odd by hypothesis, and hence $\deg (\psi^2)$ is coprime to 
$3\cdot d_0$. 
\end{proof}
\begin{proof}[Proof of Proposition~\ref{prp:stabonlagr}]
Let $\cE:=\cE(\cL)$. Let  $S$ be a smooth fiber of $\pi_A$; \textit{i.e.}, $S=\pi^{-1}(D_0)$, where $D_0\in  |\cO_{E}(3(0_{E}))|$ is reduced.  Let $\theta$ be  the $(1,3)$ polarization of $S$ 
  induced by  $\pi_A$; see the discussion at the beginning of Section~\ref{subsec:scenario}. We claim that  $\cE_{|S}$ is 
   $\theta$ slope stable
 because  the hypotheses of Corollary~\ref{crl:sempstab} are satisfied  by $(S,\theta)$ and $\cE_{|S}$. In fact,
by Proposition~\ref{prp:restsimphom},  the restriction $\cE_{|S}$ is a simple semi-homogeneous vector bundle, and hence it is Gieseker stable with respect to an arbitrary polarization  
by \cite[Proposition~6.16]{muksemi}. In particular,  $\cE_{|S}$ is $\theta$ slope semistable.
Moreover, since $\cE$ is a modular vector bundle, we have $\Delta(\cE_{|S})=0$ by~\cite[Lemma 2.5]{ogfascimod}. Lastly, since $r(\cE)=4$, it remains to show that $c_1(\cE_{|S})=2b_0\theta$ with $b_0$ odd. 
By~\eqref{eccociuno}, we have $c_1(\cE_{|S})=(2\mu_A(m\omega_A)-\delta_A)_{|S}$. Since  $S$ does not intersect 
$\Delta(A)$ (see  Proposition~\ref{prp:fiblaglis}), $c_1(\cE_{|S})=2\mu_A(m\omega_A)_{|S}$. We identify $S$ with $Z_2(C_A)$ via the isomorphism in~\eqref{zetadueci}, and we let $p_i\colon Z_2(C_A)\to C_A$ be the $\supth{i}$ projection for $i\in\{1,2,3\}$.  By the first equation in~\eqref{omegamdi}, we get that
\begin{equation}
c_1(\cE_{|S})=2\mu_A(m\omega_A)_{|S}=2md\left(p_1^{*}\left(\eta_{C_A}\right)+p_2^{*}\left(\eta_{C_A}\right)+p_3^{*}\left(\eta_{C_A}\right)\right),
\end{equation}
where $\eta_{C_A}\in H^2(C_A;\ZZ)$ is the fundamental class. Since $\theta$ is a polarization of type $(1,3)$, 
we get  that 
\begin{equation}\label{ciunosulagr}
c_1\left(\cE_{|S}\right)=2md\theta. 
\end{equation}

Since  $md$ is odd, this finishes the proof that  the hypotheses of Corollary~\ref{crl:sempstab} are satisfied  by $(S,\theta)$ and $\cE_{|S}$. 
\end{proof}
\subsection{The  general singular Lagrangian fiber}\label{subsec:gensing}
Let $E^{\vee}\subset  |\cO_{E}(3(0_{E}))|$ be the sextic curve parametrizing nonreduced divisors, \textit{i.e.}, the dual of the smooth cubic  $E\subset  |\cO_{E}(3(0_{E}))|^{\vee}$. We  have   
$\pi_A(\Delta(A))=E^{\vee}$, but there is another irreducible component of $\pi_A^{-1}(E^{\vee})$.
\begin{dfn}
Let $V(B)\subset K_2(B)$ be the closure of the set parametrizing reduced  subschemes $Z$ such that 
 $\pi_B(Z)$  has length smaller than $3$,  and define similarly
$V(A)\subset K_2(A)$ (by definition, $V(B),V(A)$ are reduced schemes). Let ${\bf V}(B)\subset K_2(B)$  be the subscheme with locally principal ideal such that  the associated cycle is $2V(B)$, and define similarly
 ${\bf V}(A)\subset K_2(A)$.
\end{dfn}
Note that we have the equality of closed sets
\begin{equation}\label{iminvduale}
\pi_A^{-1}\left(E^{\vee}\right)_{\rm red}=V(A)\cup\Delta(A). 
\end{equation}
\begin{prp}\label{prp:duevu}
Keep notation and assumptions as above. Then $\pi_A^{-1}(E^{\vee})={\bf V}(A)\cup\Delta(A)$. 
\end{prp}
\begin{proof}
The proof is similar to the proof of \cite[Proposition~6.8]{ogfascimod}. It suffices to prove that we have the equality of Cartier divisors
\begin{equation}\label{mulinobianco}
\pi_A^{*}\left(E^{\vee}\right)=2V(A)+\Delta(A).
\end{equation}
By~\eqref{iminvduale}, there exist positive integers $a,b$ such that 
$\pi_A^{*}(E^{\vee})=aV(A)+b\Delta(A)$.
Let $X_3(A)$ be the isospectral Hilbert scheme obtained by blowing up the big diagonal in $A^3$. By  \cite[Proposition 3.4.2]{haiman}, there exists a regular finite map 
$X_3(A)\to A^{[3]}$, invariant under the natural action of the permutation group $\cS_3$, which is the obvious map on the open dense subset of $X_3(A)$ parametrizing $3$-tuples of distinct points. Let $X_3^0(A)\subset X_3(A)$ be the preimage of $K_2(A)\subset A^{[3]}$, and let 
$\alpha\colon X_3^0(A)\to K_2(A)$ be the restriction of the map $X_3(A)\to A^{[3]}$. Restricting to $X_3^0(A)$ the natural map $X_3(A)\to A^3$, we get a  map $X_3^0(A)\to Z_2(A)$, where $Z_2(A)\subset A^3$ is the kernel of the summation map $A^3\to A$. Let $\lambda_A\colon X_3^0(A)\to Z_2(E)$ be the composition of $X_3^0(A)\to Z_2(A)$ and the map 
$Z_2(A)\to Z_2(E)$ defined by  $\varphi_A\colon A\to E$. Lastly, let $\beta\colon Z_2(E) \to  |\cO_E(3(0_E))|$ be the map sending $(x_1,x_2,x_3)$ to $(x_1)+(x_2)+(x_3)$.
We have a commutative square 
\begin{equation}
\xymatrix{
X_3^0(A)  \ar@{->}^{\alpha}[r]  \ar@{->}_{\lambda_A}[d] & K_2(A) \ar@{->}^{\pi_A}[d] \\
Z_2(E)    \ar@{->}^{\beta}[r]  & |\cO_E(3(0_E))|.
}
\end{equation}
As is easily checked, we have 
$\lambda_A^{*}(\beta^{*}E^{\vee})=2V(A)+2\Delta(A)$. On the other hand, since $\pi_A^{*}(E^{\vee})=aV(A)+b\Delta(A)$, we have 
$\alpha^{*}(\pi_A^{*}E^{\vee})=aV(A)+2b\Delta(A)$. Since $\lambda_A^{*}\circ\beta^{*}=\alpha^{*}\circ\pi_A^{*}$, we get that $a=2$ and $b=1$.
\end{proof}
\begin{dfn}
If $D_0\in E^{\vee}$, we let $V(B)_{D_0}\subset V(B)$ be the fiber over $D_0$ of the map 
$V(B)\to  E^{\vee}$ given by the restriction of $\pi_B$, and we define similarly $\bV(B)_{D_0}\subset \bV(B)$, and likewise for $V(A)_{D_0}\subset V(A)$ and  $\bV(A)_{D_0}\subset \bV(A)$. 
We let $\Delta(B)_{D_0}\subset \Delta(B)$ be the fiber over $D_0$ of the map 
$\Delta(B)\to  E^{\vee}$ given by the restriction of $\pi_B$, and  we define similarly  $\Delta(A)_{D_0}\subset \Delta(A)$. 
\end{dfn}
Let  $D_0\in E^{\vee}$. By Proposition~\ref{prp:duevu},  
we have
\begin{equation}\label{iminvdizero}
\pi_A^{-1}(D_0)=\bV(A)_{D_0}\cup\Delta(A)_{D_0}. 
\end{equation}
In the present section, we are concerned with fibers of $\pi_A$ over a  $D_0\in E^{\vee}$ which is not an inflection divisor, \textit{i.e.}, such that
\begin{equation}\label{noflesso}
D_0=2(x_0)+(y_0),\quad 2x_0+y_0=0,  \;   x_0\not= y_0\in E. 
\end{equation}
Let $\xi(A)$ be the line bundle on $K_2(A)$ such that 
\begin{equation}\label{eccoxi}
c_1(\xi(A))=\delta(A). 
\end{equation}
\begin{prp}\label{prp:ginseng}
If\,  $D_0\in E^{\vee}$ is not an inflection divisor,  
then we have an exact sequence
\begin{equation}
0 \lra \cO_{V(A)_{D_0}} \otimes\left(\xi(A)_{|V(A)_{D_0}}\right) \lra \cO_{\bV(A)_{D_0}} \lra \cO_{V(A)_{D_0}}  \lra 0,
\end{equation}
with the right-hand map  given by restriction.
\end{prp}
\begin{proof}
We must prove that the ideal sheaf of $V(A)_{D_0}$ in $\bV(A)_{D_0}$ is isomorphic to the restriction of the invertible sheaf $\xi(A)$.
By the equality in~\eqref{mulinobianco}, we have $-2V(A)+\pi_A^{*}(\cO(6))\equiv\Delta(A)$ (the degree of $E^{\vee}$ is $6$). Since $K_2(A)$ is simply connected, we get that
\begin{equation}\label{delcroix}
-V(A)+\pi_A^{*}(\cO(3))\equiv\xi(A).
\end{equation}
Restricting to $\bV(A)_{D_0}$, we get what we need.
\end{proof}
\subsection{Geometry of  $V(B)_{D_0}$ and $V(A)_{D_0}$}\label{subsec:geomvert}
Assume that 
$D_0=(2x_0+y_0)\in E^{\vee}$ is not an inflection divisor.
Let $C_{B,x_0}:=\varphi_B^{-1}(x_0)$ and $C_{B,y_0}:=\varphi_B^{-1}(y_0)$, where $\varphi_B\colon B\to E$ is as in~\eqref{eccofibi}. 
Choose $\wt{x}_0\in C_{B,x_0}$, and let $\wt{y}_0:=-2\wt{x}_0$. Notice that 
$\wt{y}_0\in C_{B,y_0}$.   We have an isomorphism
\begin{equation}\label{blade}
\begin{matrix}
C_B^{(2)} & \overset{\lowsim}{\lra} & V(B)_{D_0} \\
(z_1)+(z_2) & \longmapsto & (z_1+\wt{x}_0)+(z_2+\wt{x}_0)+(-z_1-z_2+\wt{y}_0). 
\end{matrix}
\end{equation}
A word about notation: A subscheme parametrized by a point of  $V(B)_{D_0}$ is the disjoint union of a length $2$ subscheme of $C_{B,\wt{x}_0}$ and a reduced point, and for this reason, we identify it with  the corresponding cycle. In other words, if $z_1=z_2$, then $(z_1)+(z_2)$ is to be understood as the unique length $2$ subscheme of $C_{B,\wt{x}_0}$ supported at $z_1=z_2$. 
Of course, we have an analogous identification between $C_A^{(2)}$ and  $V(A)_{D_0}$. 

Recall that the degree $4$ rational map $\rho\colon K_2(B)\dra K_2(A)$ (see~\eqref{dabiada}) has indeterminacy locus $V(f)$ (see~\eqref{vuenne}). The intersection of $V(f)$ and  $V(B)_{D_0}$ is  the reduced curve identified by the isomorphism in~\eqref{blade} with the curve
\begin{equation*}
\{(z)+(z+\epsilon)\mid z\in C_B\}.
\end{equation*}
(Recall that $\epsilon\in B[2]$ is the generator of $\ker f$.)
Hence the inclusion $V(B)_{D_0}\hra K_2(B)$ lifts to an  inclusion $V(B)_{D_0}\hra X$ (recall that  
$\nu\colon X\to K_2(B)$ is the blow-up of $V(f)$). With this understood, we have the schematic equality
\begin{equation}\label{verdone}
\wt{\rho}^{-1}(V(A)_{D_0})=V(B)_{D_0}.
\end{equation}
Moreover, once we make the identification in~\eqref{blade} and the analogous identification   between $C_A^{(2)}$ and  $V(A)_{D_0}$,    the restriction of $\wt{\rho}$ to $V(B)_{D_0}$ is  identified with the degree $4$ regular map 
\begin{equation}
\begin{matrix}
V(B)_{D_0}=C_B^{(2)} & \xrightarrow{\left(f_{|C_B}\right)^{(2)}} & C_A^{(2)}=V(A)_{D_0} \\
(z_1)+(z_2) & \longmapsto & (f(z_1))+(f(z_2)). 
\end{matrix}
\end{equation}
Restricting the line bundle $\cL$ to $V(B)_{D_0} $ and pulling it back to $C_B^{(2)}$ via the isomorphism 
in~\eqref{blade},  we get a line bundle ${\mathsf L}$ on $C_B^{(2)}$. By the equality in~\eqref{verdone}, we have the isomorphism (we make the identification in~\eqref{blade})
\begin{equation}\label{carlo}
\cE(\cL)_{|V(A)_{D_0}}\cong \left(f_{|C_B}\right)^{(2)}_{*}({\mathsf L}).
\end{equation}
In order to analyze the vector bundle in the right-hand side in~\eqref{carlo}, we recall that we have $\PP^1$ fibrations
\begin{equation}\label{melbrooks}
\xymatrix{\PP^1 \ar[r]   & C_B^{(2)}=V(B)_{D_0}  \ar[d]_{g_B}    \\ 
   & C_B, }\quad
 \xymatrix{\PP^1 \ar[r]   & C_A^{(2)}=V(A)_{D_0}  \ar[d]_{g_A}    \\ 
   & C_A, }
\end{equation}
where $g_B((z_1)+(z_2))=z_1+z_2$ and  $g_A$ is defined similarly. For use later on, we notice that if $x\in C_B$, we have
\begin{equation}\label{triste}
\deg({\mathsf L}_{|g_B^{-1}(x)})=md.
\end{equation}
The fibrations in~\eqref{melbrooks} are the projectivizations of rank $2$ vector bundles:
\begin{equation}\label{theyesmen}
V(B)_{D_0}=C_B^{(2)}=\PP(\cF_B),\quad V(A)_{D_0}=C_A^{(2)}=\PP(\cF_A)
\end{equation}
(of course, $\cF_B$ is determined only up to tensorization by the pull-back of a line bundle on $C_B$, and similarly for 
$\cF_A$). Let $\cO_{C_B^{(2)}}(1)=\cO_{V(B)_{D_0}}(1)$ be the corresponding line bundle on $C_B^{(2)}=V(B)_{D_0}$ and similarly for $V(A)_{D_0}=C_A^{(2)}$.  (We follow the classical definition of projectivization of a vector bundle: The space of global sections of $\cO_{C_B^{(2)}}(1)$ is $H^0(C_B,\cF^{\vee}_B)$.) The result below is well known (see~\cite[Section~3, p.~451]{atiyah-ell}).
\begin{prp}\label{prp:manuela}
Keeping notation as above, the rank $2$ vector bundles $\cF_B$ and $\cF_A$  are stable $($in particular, they have  odd degrees$)$.
\end{prp}
\begin{dfn}\label{dfn:gamsig}
Let $\Gamma\subset V(A)_{D_0}$ be (the class of) a fiber of $g_A$, and let
\begin{equation}\label{nicoletta}
\Sigma:=V(A)_{D_0}\cap\Delta(A)_{D_0}=V(A)_{D_0}\cap\Delta(A)=\{2(p)\mid p\in C_A\}
\end{equation}
 (note that the intersection $V(A)\cap \Delta(A)$ is generically transverse, and hence 
$\Sigma$ is reduced). (For the sake of simplicity, we omit  $A$ and $D_0$ in our notation for $\Gamma$ and $\Sigma$.) 
\end{dfn}
Note that 
$\{\Gamma,\Sigma\}$ is a basis of $\NS(V(A)_{D_0})_{\QQ}$. The following result is used later on; the proof is left to the reader.
\begin{prp}\label{prp:franchin}
The ample cone of\, $V(A)_{D_0}$ is the interior of the convex cone generated by $\Gamma$ and $\Sigma$. 
\end{prp}
\subsection{Restriction of $\cE(\cL)$ to a general singular Lagrangian fiber}\label{subsec:reslag}
\begin{prp}\label{prp:monella}
Assume that Hypothesis~\ref{hyp:biellittica} and Equation~\eqref{nostraelle} hold and that  $m d$ is odd $(d$  as in~\eqref{omomandomga} and $m$ as in~\eqref{nostraelle}$)$. 
If\,  $D_0\in E^{\vee}$ is not an inflection divisor, then
\begin{equation}
\cE(\cL)_{|V(A)_{D_0}} \cong g_A^{*}(\cV)\otimes\cO_{V(A)_{D_0}}((md-1)/2),
\end{equation}
where $g_A\colon V(A)_{D_0}\to C_A$ is the fibration in~\eqref{melbrooks} and $\cV$ is a stable rank $4$ vector bundle on $C_A$.
\end{prp}
\begin{proof}
Throughout the proof we identify $V(A)_{D_0}$ with $C_A^{(2)}$.  In order to simplify notation, we let $\psi:=\left(f_{|C_B}\right)^{(2)}$. Let 
\begin{equation}\label{dadead}
{\mathsf L}':={\mathsf L}\otimes \psi^{*}\cO_{C_A^{(2)}}(-(md-1)/2).
\end{equation}
By the push-pull formula, it suffices to prove that there exists  a stable rank $4$ vector bundle $\cV$ on $C_A$ such that 
\begin{equation}\label{funeral}
\psi_{*}\left({\mathsf L}'\right)\cong g_A^{*}(\cV).
\end{equation}
We factor $\psi$ as the composition of two maps. 
Let $\epsilon\in C_B[2]$ be the generator of the kernel of $f\colon B\to A$. Then $\psi\colon C_B^{(2)}\to C_A^{(2)}$ is the quotient map for the  action of $\ZZ/(2)^2$ on $C_B^{(2)}$ defined by
\begin{equation*}
(z_1)+(z_2)\longmapsto (z_1+k_1\epsilon)+(z_2+k_2\epsilon),
\end{equation*}
where $(k_1,k_2)\in \ZZ/(2)^2$.
Let $i$ be the involution  of $C_B^{(2)}$ defined by 
$i((z_1)+(z_2))=(z_1+\epsilon)+(z_2+\epsilon)$, and  let $Y:=C_B^{(2)}/\la i\ra$. We  let 
$p\colon  C_B^{(2)}\to Y$ be the quotient map. Since $i$ commutes with  $g_B$, the $\PP^1$ fibration 
 $g_B$  induces a   $\PP^1$ fibration $g_Y\colon Y\to C_B$. The restriction of $p$ to a fiber $g_B^{-1}(x)$ is a  map 
\begin{equation}\label{audace}
\PP^1\cong g_B^{-1}(x)\xrightarrow{p_{|g_B^{-1}(x)}} g_Y^{-1}(x)\cong \PP^1
\end{equation}
 of  degree $2$.
 
 Next, by mapping $[(z_1)+(z_2)]$ to $[(z_1+\epsilon)+(z_2)]$, we get a well-defined involution $j\colon Y\to Y$ such that $Y/\la j\ra$ is identified with $C_A^{(2)}$. 
 Let $q\colon Y\to Y/\la j\ra\cong C_A^{(2)}$ be the quotient map.
  Notice that  the  fibration $g_Y\colon Y\to C_B$
 is identified with the pull-back of the fibration $g_A$ via the double cover $C_B\to C_A$. In other words, we have an identification $Y=C_B\times_{C_A}C_A^{(2)}$. Summing up, we have a factorization
$\psi=q\circ p$. Hence  the isomorphism in~\eqref{funeral} holds if and only if 
\begin{equation}\label{martello}
 q_{*}(p_{*}({\mathsf L}'))\cong g_A^{*}(\cV)
\end{equation}
for a stable rank $4$ vector bundle $\cV$ on $C_A$. Let us prove that there exists a line bundle $\lambda$ on $C_B$ such that 
\begin{equation}\label{carlreiner}
p_{*}({\mathsf L}')\cong g_Y^{*}\left(\cF_B^{\vee}\otimes \lambda\right),
\end{equation}
where $\cF_B$ is a vector bundle on $C_B$ such that the second equality in~\eqref{theyesmen} holds.
Since the map in~\eqref{audace} (for $x\in C_B$)  has  degree $2$, and since  $g_Y\colon Y\to C_B$
 is identified with the pull-back of  $g_A$ via the double cover $C_B\to C_A$, 
we have $\psi^{*}\left(\cO_{C_A^{(2)}}(1)\right)\cong \cO_{C_B^{(2)}}(2)\otimes g_B^{*}(\xi)$ for a suitable line bundle $\xi$ on $C_B$. By the equality in~\eqref{triste}, it follows that  the restriction of ${\mathsf L}'$ to any fiber $g_B^{-1}(x)$ has degree $1$. 
As is easily checked, if $\alpha\colon\PP^1\to\PP^1$ is a degree $2$ map, then $\alpha_{*}(\cO_{\PP^1}(1))$ is the trivial rank $2$ vector bundle  $\cO^2_{\PP^1}$. Note the identification 
\begin{equation}
H^0\left(\PP^1,\cO_{\PP^1}(1)\right)\overset{\lowsim}{\lra} H^0\left(\PP^1,\alpha_{*}\cO^2_{\PP^1}(1)\right)
= H^0\left(\PP^1,\cO^2_{\PP^1}\right).
\end{equation}
From this, we get that there exists a vector bundle $\cW$ on $C_B$ such that 
$p_{*}({\mathsf L}')\cong g^{*}_Y(\cW)$, and moreover we get an isomorphism $g_{B,*}(\cO_{C_B^{(2)}}(1))\otimes\lambda\cong \cW$ for a suitable line $\lambda$ on $C_B$. Since $g_{B,*}(\cO_{C_B^{(2)}}(1))\cong \cF_B^{\vee}$, this proves~\eqref{carlreiner}. 

Let $\tau\colon C_B\to C_A$ be the double cover map. From the isomorphism in~\eqref{carlreiner}, we get that 
\begin{equation}\label{annebancroft}
 q_{*}\left(p_{*}({\mathsf L}')\right)\cong g_A^{*}\left(\tau_{*}\left(\cF_B^{\vee}\otimes\lambda\right)\right).
\end{equation}
This proves that~\eqref{funeral} holds with $\cV:=\tau_{*}(\cF_B^{\vee}\otimes\lambda)$. It remains to prove that $\cV$ is stable. The Grothendieck--Riemann--Roch theorem and a straightforward computation give that 
$\deg(\cV)=\deg(\cF_B^{\vee}\otimes\lambda)$. Since  
$\cF_B^{\vee}\otimes\lambda$ has  odd degree (see Proposition~\ref{prp:manuela}),  the rank (\textit{i.e.}, $4$) and the degree of  
$\cV$ are coprime. Hence in order to prove that $\cV$ is stable, it suffices to show that $\cV$ is semistable.
 The vector bundle $\cF_B^{\vee}\otimes\lambda$ is stable by Proposition~\ref{prp:manuela}; hence it is 
semi-homogeneous. Since the map $\Pic^0(C_A)\to\Pic^0(C_B)$ defined by pull-back is surjective, it follows that $\cV$ is 
semi-homogeneous as well. By \cite[Proposition~6.13]{muksemi}, it follows that $\cV$ is semistable, and we are done.  
\end{proof}
\begin{crl}\label{crl:monella}
Suppose that Hypothesis~\ref{hyp:biellittica} and Equation~\eqref{nostraelle} both hold and that  $md$ is odd. If\,
  $D_0\in E^{\vee}$ is not an inflection divisor,  then the restriction of $\cE(\cL)$ to $V(A)_{D_0}$ is slope stable with respect to any polarization.
\end{crl}
\begin{proof}
Let $\cV$ be the rank $4$ stable vector bundle on $C_A$ of Proposition~\ref{prp:monella}. It suffices to prove that 
$g_A^{*}(\cV)$ is  slope stable with respect to any polarization. First we notice that $g_A^{*}(\cV)$ is slope stable with respect to a polarization  represented by a section  of $g_A$ because  the restriction of $g_A^{*}(\cV)$ to a section
 is identified with~$\cV$ and hence is slope stable. Because of this (existence of polarizations for which 
$g_A^{*}(\cV)$ is slope stable), it suffices to prove that there does not exist a polarization for which  $g_A^{*}(\cV)$ is strictly slope semistable, \textit{i.e.}, slope semistable but not semistable. Since $\Delta(g_A^{*}(\cV))=0$ and $\deg\cV$ is coprime to $r(\cV)=4$ (by the stability of $\cV$), this follows from Lemma~\ref{lmm:disczero} below.
\end{proof}
\begin{lmm}\label{lmm:disczero}
Let $(S,h)$ be a polarized smooth irreducible projective surface. Let $V$ be a vector bundle on $S$ 
such that $\Delta(V)=0$ and the rank of\, $V$ is coprime to the maximum integer dividing $c_1(V)$. Then $V$ is not  properly slope semistable. 
\end{lmm}
\begin{proof}
Suppose that $V$ is  properly slope semistable. Then there exists an exact sequence of sheaves
\begin{equation}
0\lra U\lra V\lra W\lra 0
\end{equation}
such that  $U$ is locally free slope semistable of rank $0<r_U<r_V$ (where $r_U$ is the rank of $U$, \textit{etc.}), $W$ is torsion-free slope semistable and
\begin{equation}\label{lambdaperp}
(r_V c_1(U)-r_U c_1(V))\cdot h=0.
\end{equation}
A Chern class computation gives the equality
\begin{equation}\label{buy}
r_{V}(r_{W} \Delta(U)+  r_{U} \Delta(W))=r_{U}\cdot r_{W}\Delta( V)+
(r_{V} c_1(U)-r_{U}c_1(V))^2.
\end{equation}
 By hypothesis,  
$\Delta(V)=0$, and moreover $\Delta(U)\ge 0$, $\Delta(W)\ge 0$ by Bogomolov's inequality. On the other hand,
$(r_V c_1(U)-r_U c_1(V))^2\le 0$ by the equality in~\eqref{lambdaperp} and 
 the Hodge index theorem. By the equality in~\eqref{buy}, we get that $(r_V c_1(U)-r_U c_1(V))^2= 0$
and hence 
\begin{equation*}
r_V c_1(U)=r_U c_1(V). 
\end{equation*}
This  contradicts the hypothesis that $r_V$ is coprime to the maximum integer dividing $c_1(V)$.
\end{proof}
\begin{prp}\label{prp:portuense}
Assume that Hypothesis~\ref{hyp:biellittica} and Equation~\eqref{nostraelle} hold and that  $md$ is odd $(d$ is as in~\eqref{omomandomga} and $m$  as in~\eqref{nostraelle}$)$. 
  If\,  $D_0\in E^{\vee}$ is not an inflection divisor,   then we have  an exact sequence
\begin{equation}\label{melbunivmaga}
0 \lra \cE(\cL)\otimes\xi(A)_{| V(A)_{D_0}} \overset{\alpha}{\lra} \cE(\cL)_{|\bV(A)_{D_0}}
\overset{\beta}{\lra}  \cE(\cL)_{| V(A)_{D_0}}\lra 0,
\end{equation}
where $\beta$ is given by restriction. Moreover,  the restriction $ \cE(\cL)_{| V(A)_{D_0}}$ is slope stable for any polarization of\, $V(A)_{D_0}$.
\end{prp}
\begin{proof}
This is a straightforward consequence of Proposition~\ref{prp:ginseng} and Corollary~\ref{crl:monella}.
\end{proof}
Next we analyze the restriction of $\cE(\cL)$ to $\Delta(A)_{D_0}$. We recall that $\Delta(A)_{D_0}$ is isomorphic to $C_A\times\PP^1$.
\begin{prp}\label{prp:shkabarnya}
Assume that Hypothesis~\ref{hyp:biellittica} and Equation~\eqref{nostraelle} hold and that  $md$ is odd $(d$ is as in~\eqref{omomandomga} and $m$  as in~\eqref{nostraelle}$)$. 
If\,  $D_0\in E^{\vee}$ is not an inflection divisor,   then we have the Harder--Narasimhan filtration $($with respect to any polarization$)$
\begin{equation}\label{harnarfil}
0\lra \cO_{C_A}(P)\boxtimes\cO_{\PP^1}(1) \lra \cE(\cL)_{|\Delta(A)_{D_0}}\lra 
\bigoplus\limits_{k=1}^3\cO_{C_A}(P_k)\boxtimes\cO_{\PP^1}\lra 0,
\end{equation}
 where $P_k$ and $P$ are divisors  on $C_A$  of degree $3md$. 
\end{prp}
\begin{proof}
 In the proof of Lemma~\ref{lmm:geomro}, we showed that 
\begin{equation}\label{allianz}
\wt{\rho}^{*}\Delta(A)=\nu^{*}(\Delta(B))+2D.
\end{equation}
 Let
\begin{equation*}
\Xi(A):=\wt{\rho}(\nu^{-1}(\Delta(B))\cap D)= \{[Z]\in K_2(A) \mid \gh(Z)=3(p),\ p\in A[3]\}
\end{equation*}
(recall that $\gh\colon K_2(A)\to A^{(3)}$ is the Hilbert--Chow map), and let $\Delta(A)^0:=\Delta(A)\setminus \Xi(A)$. 
It follows from the equality in~\eqref{allianz} that 
we have
\begin{equation}\label{montecristo}
\cE(\cL)_{|\Delta(A)^0}\cong \mu_{1,*}\left(\cL_{|\nu^{*}(\Delta(B))}\right)_{|\Delta(A)^0}\oplus \mu_{2,*}\left(\cL_{|2D}\right)_{|\Delta(A)^0},
\end{equation}
where $\mu_1$ and $\mu_2$ are the restrictions of $\wt{\rho}$ to $\nu^{*}(\Delta(B))$ and $2D$, respectively. Since 
$\Delta(A)_{D_0}\subset \Delta(A)^0$, we get that
\begin{equation}\label{intesa}
\cE(\cL)_{|\Delta(A)_{D_0}}\cong \mu_{1,*}\left(\cL_{|\nu^{*}(\Delta(B))}\right)_{|\Delta(A)_{D_0}}\oplus 
\mu_{2,*}\left(\cL_{|2D}\right)_{|\Delta(A)_{D_0}}.
\end{equation}
The direct summands in the right-hand side of~\eqref{intesa} are described as follows. First, we have 
$\wt{\rho}^{-1}(\Delta(A)_{D_0})=\nu^{-1}(\Delta(B)_{D_0})$. Second, since  $\Delta(B)_{D_0}$ is disjoint from the center of the blow-up $\nu\colon X\to K_2(B)$, the map $\nu^{-1}(\Delta(B)_{D_0})\to \Delta(B)_{D_0}$ is an isomorphism; for this reason, we identify 
$\nu^{-1}(\Delta(B)_{D_0})$ and $\Delta(B)_{D_0}$. 
Third, the restriction of  $\mu_1$ to $\Delta(B)_{D_0}$ is 
 the \'etale double cover 
$C_B\times \PP^1\to C_A\times \PP^1$ defined by the quotient map $\tau\colon C_B\to C_A$. 
 It follows that, letting
$\sigma_B\colon C_B\hra \Delta(B)_{D_0}=C_B\times\PP^1$ be a section, we have 
\begin{equation}
 \mu_{1,*}\left(\cL_{|\nu^{*}(\Delta(B))}\right)_{|\Delta(A)_{D_0}}\cong \tau_{*}\left(\sigma_B^{*}(\cL)\right)\boxtimes\cO_{\PP^1}.
\end{equation}
A computation\footnote{This is the computation $\la \mu_A(\ov{\omega}_A,\Sigma\ra=6d$, which also holds with $B$ replacing $A$, of course.} gives that $\deg \sigma_B^{*}(\cL)=6md$. Hence there exists a divisor $P_1$ on $C_A$ of degree $3md$ such that 
$\tau^{*}\cO_{C_A}(P_1)\cong  \sigma_B^{*}(\cL)$.
It follows that 
\begin{equation}
 \tau_{*}\left(\sigma_B^{*}(\cL)\right)\cong \cO_{C_A}(P_1)\oplus \cO_{C_A}(P_2),
\end{equation}
where $\deg P_2=3md$. (In fact, $P_2\equiv P_1+\eta$, where $\eta$ is the divisor class of order $2$ determined by the double cover $\tau\colon C_B\to C_A$.) 

Next we notice that the map $D\to\Delta(A)$ given by the restriction of $\wt{\rho}$ is an isomorphism away from $\Xi(A)$ (see Section~\ref{subsec:vueffe}); hence we have an open inclusion $\iota\colon \Delta(A)^0\hra D$ that composed with $\wt{\rho}$ gives the identity.  
We have an exact sequence
\begin{equation*}
0 \lra \iota^{*}(\cL)\otimes\iota^{*}(\cO_D(-D))\lra   \mu_{2,*}\left(\cL_{|2D}\right)_{|\Delta(A)^0}\lra \iota^{*}(\cL)\lra 0.
\end{equation*}
Restricting the above exact sequence to $\Delta(A)_{D_0}$, we get an exact sequence  
\begin{equation}
0\lra \cO_{C_A}(P')\boxtimes\cO_{\PP^1}(1)\lra  \mu_{2,*}\left(\cL_{|2D}\right)_{|\Delta(A)_{D_0}}\lra \cO_{C_A}(P')\boxtimes\cO_{\PP^1}\lra 0, 
\end{equation}
where $P'$ is a divisor of degree $3md$. Hence we get that $ \cE(\cL)_{|\Delta(A)_{D_0}}$ fits into an exact sequence as in~\eqref{harnarfil}, where  $P_3=P=P'$. It is clear that the exact sequence is the Harder--Narasimhan filtration.
\end{proof}
\subsection{Simplicity of the restriction of $\cE(\cL)$ to a general singular fiber}\label{subsec:demimoore}
We prove Propositiopn~\ref{prp:sempfibre}.
Let $\cE=\cE(\cL)$, and let $\varphi\colon \cE\to\cE$ be an endomorphism. Then  
  $\varphi$ maps the kernel of $\alpha$  in~\eqref{melbunivmaga} to itself (because, as subsheaf of 
  $\cE_{|\bV(A)_{D_0}}$, it is the  kernel of multiplication by the ideal of $V(A)_{D_0}$ in $\bV(A)_{D_0}$).
   Since $\cE_{|V(A)_{D_0}}$ is slope stable, the restriction of  $\varphi$ to $\ker(\alpha)$ is  multiplication by  a  certain $\lambda\in\CC$. 
 Thus    
  $\varphi=\lambda \Id_{\cE}+\ov{\varphi}$, where  $\ov{\varphi}$ annihilates $\ker(\alpha)$. Let us prove that 
  $\ov{\varphi}$ vanishes. The restriction of  
  $\ov{\varphi}$  to $\bV(A)_{D_0}$ equals the composition
\begin{equation}\label{tenco}
\cE_{|\bV(A)_{D_0}} \lra \cE_{|V(A)_{D_0}} \overset{\psi}{\lra} \cE\otimes\xi(A)_{|V(A)_{D_0}}.
\end{equation}
We start by proving that    the  determinant of   $\psi$ vanishes.
In fact, by the factorization in~\eqref{tenco}, we get that
 the restriction of  $\ov{\varphi}$  to $\Delta(A)_{D_0}$ is a homomorphism 
\begin{equation*}
\cE_{|\Delta(A)_{D_0}}\lra \cE_{|\Delta(A)_{D_0}}\otimes \cO_{|\Delta(A)_{D_0}}(-\Sigma),  
\end{equation*}
where $\Sigma=V(A)_{D_0}\cap \Delta(A)_{D_0}$; see Definition~\ref{dfn:gamsig}.
By the exact sequence in~\eqref{harnarfil}, we get that  the restriction of  $\ov{\varphi}$  to $\Delta(A)_{D_0}$ has image contained in the sub-line bundle $ \cO_{C_A}(P)\boxtimes\cO_{\PP^1}(1)\otimes \cO_{|\Delta(A)_{D_0}}(-\Sigma)$. In particular, 
the restriction to $\Sigma$ of such a homomorphism has 
rank at most $1$. It follows that   $\det\psi$ vanishes on $\Sigma$ with order at least $3$. 
 Since $\det\psi$ is a section of $\xi(A)^{\otimes 4}_{|V(A)_{D_0}}\cong \cO_{V(A)_{D_0}}(2\Sigma)$, we get that 
   the  determinant of  $\psi$   is  identically zero. 
   
  By Proposition~\ref{prp:monella}, the restriction of $\cE$ to $V(A)_{D_0}$ is 
  the tensor product of $g_A^{*}(\cV)$ (notation as in Proposition~\ref{prp:monella}) and a line bundle. Thus $\psi$ defines a map 
\begin{equation}
\psi'\colon g_A^{*}(\cV)\lra g_A^{*}(\cV)\otimes \xi(A)_{|V(A)_{D_0}}
\end{equation}
with (generic) rank at most $3$. 

We claim that $\psi'=0$ (\textit{i.e.}, $\psi=0$). 
We have $c_1(\im\psi')=a\Sigma +b\Gamma$ with $a,b$ rational numbers (see Section~\ref{subsec:geomvert}). Let $H$ be a polarization of $V(A)_{D_0}$; thus $c_1(H)=s\Sigma +t\Gamma$, where  $s,t$ are positive rational numbers (see Definition~\ref{dfn:gamsig}).

Assume that the (generic) rank of $\psi'$ is $1$. By Corollary~\ref{crl:monella},  the vector bundle $g_A^{*}(\cV)$ is $H$ slope stable.  Hence  the surjection 
$g_A^{*}(\cV)\twoheadrightarrow \im\psi'$  gives  the inequality
\begin{equation}
s\deg\cV=\mu_H\left(g_A^{*}(\cV)\right)<\mu_H( \im\psi')=4at+4bs, 
\end{equation}
and the injection  $\im\psi'\hra g_A^{*}(\cV)\otimes \xi(A)_{|V(A)_{D_0}}$ gives   the inequality
\begin{equation}
4at+4bs=\mu_H( \im\psi')< \mu_H\left( g_A^{*}(\cV)\otimes \xi(A)_{|V(A)_{D_0}}\right)=s\deg\cV+2t.
\end{equation}
By Proposition~\ref{prp:franchin},  the coefficient $t$ may be arbitrarily small (and $s$ is positive); it follows that $\deg\cV\le 4b\le \deg\cV$. Hence 
$\deg\cV= 4b$. On the other hand, $2b$ is an integer because 
$c_1( \xi(A)_{|V(A)_{D_0}})\cdot c_1(\im\psi')=\frac{1}{2}\Sigma\cdot(a\Sigma +b\Gamma)=2b$, and hence we get that 
$\cV$ has even degree. This contradicts the stability of $\cV$.  

Next we deal with the (hypothetical)   cases in which the (generic) rank of  $\psi'$ is $r\in\{2,3\}$. Then we have  inclusions of sheaves on $V(A)_{D_0}$
\begin{equation*}
\im(\psi')\subset \cF\subset g_A^{*}(\cV)\otimes\xi(A)_{|V(A)_{D_0}},
\end{equation*}
where  $\cF$ is the saturation of $\im(\psi')$, \textit{i.e.}, a  sheaf of rank $r$  and with  torsion-free cokernel. The key observation is that, since the rank of  $\im(\psi)$  on the divisor $\Sigma$ is at most $1$ (this was proved above, when we showed that   $\psi$ has vanishing determinant), we have 
$c_1(\cF)=c_1(\im(\psi'))+m\Sigma+{\rm Eff}$, where $m\ge(r-1)$ and ${\rm Eff}$ is an effective divisor.    
 By the stability of $g_A^{*}(\cV)$, we get the inequality
\begin{equation*}
s\deg\cV=\mu_H\left(g_A^{*}(\cV)\right)<\mu_H( \im\psi') 
\end{equation*}
and the  inequality
\begin{equation*}
\mu_H( \im\psi')+\frac{4(r-1)t}{r}\le \mu_H( \cF)< \mu_H\left( g_A^{*}(\cV)\otimes \xi(A)_{|V(A)_{D_0}}\right)=s\deg\cV+2t.
\end{equation*}
Since   $r\in\{2,3\}$, the second equality gives that $\mu_H( \im\psi')<s\deg\cV$, and this contradicts the first inequality.
This proves that $\psi=0$, \textit{i.e.}, that the restriction of  $\ov{\varphi}$  to $\bV(A)_{D_0}$ vanishes. 

It follows that the restriction of  $\ov{\varphi}$  to $\Delta(A)_{D_0}$ is a homomorphism 
\begin{equation*}
\cE_{|\Delta(A)_{D_0}}\lra \cE_{|\Delta(A)_{D_0}}\otimes \cO_{|\Delta(A)_{D_0}}(-2\Sigma).  
\end{equation*}
Since  $\cO_{|\Delta(A)_{D_0}}(-2\Sigma)\cong \cO_{C_A}\boxtimes \cO_{\PP^1}(-2)$, it follows from
 Proposition~\ref{prp:shkabarnya}  that such a homomorphism vanishes. This proves that the restriction of 
 $\ov{\varphi}$  to $\Delta(A)_{D_0}$ is zero, and hence  $\ov{\varphi}=0$. 
\qed
\subsection{Pedestrian proof of Proposition~\ref{prp:harbridge}}\label{subsec:monicavitti}
By Proposition~\ref{prp:stabonlagr}, the restriction of $\cE(\cL)$ to a general Lagrangian fiber is simple; consequently,  $H^0(K_2(A),\End^0(\cE(\cL))=0$. By Serre duality, it follows that $H^4(K_2(A),\End^0(\cE(\cL))=0$. 
Let us prove that 
\begin{equation}
H^1\left(K_2(A),\End^0(\cE(\cL)\right)=H^3\left(K_2(A),\End^0(\cE(\cL)\right)=0.
\end{equation}
By Serre duality, it suffices to prove that $H^1(K_2(A),\End^0(\cE(\cL))=0$. 
The proof is analogous to the proof of \cite[Proposition~6.7]{ogfascimod}. Suppose that $D_0\in |\cO_C(3(0_E))|$ is not a flex divisor, and let $\cW_{D_0}$ be the restriction of  $\End^0\cE(\cL)$  to  $\pi_A^{-1}(D_0)$.
By Propositions~\ref{prp:stabonlagr} and~\ref{prp:sempfibre}, there are no nonzero global sections of  
$\cW_{D_0}$.
Since  $\pi_A^{-1}(D_0)$ is a local complete intersection surface with trivial dualizing sheaf, it follows that  
$H^2(\pi_A^{-1}(D_0),\cW_{D_0})$ vanishes as well. On the other hand,  $\chi(\pi_A^{-1}(D_0),\cW_{D_0})=0$ because $\cW_{D_0}$ is a semi-homogeneous vector bundle if $D_0$ is reduced (see Proposition~\ref{prp:restsimphom}), and  
hence all the cohomology of $\cW_{D_0}$ vanishes. Since the set of flex divisors has codimension $2$ in 
$ |\cO_C(3(0_E))|$, it follows that
 $R^q\pi_{A,*}\End^0\cE(\cL)=0$  for $q\in\{0,1\}$ (see \cite[Proposition 2.26]{mukvb}). By 
the Leray spectral sequence  for $\pi_A$, we get that $H^1(K_2(A),\End^0(\cE(\cL)))=0$.

By Proposition~\ref{prp:charendotre},  we have 
$\chi(K_2(A),\End^0(\cE(\cL))=0$. It follows that  
\begin{equation}\label{accadue}
H^2\left(K_2(A),\End^0(\cE(\cL)\right)=0. 
\end{equation}
This proves the vanishing statement of Proposition~\ref{prp:harbridge}. The statement about the smoothness of the map between deformation spaces is a consequence of the vanishing in~\eqref{accadue}.
\qed
\section{Restriction of $\cE(\cL)$ to a general singular Lagrangian fiber}\label{sec:nazariosauro}
\subsection{Main result}
We adopt the notation introduced in Section~\ref{sec:analisifine}. 
Suppose that $D_0\in E^{\vee}$ is not an inflection divisor. We will show that if $md>1$ ($d$  as in~\eqref{omomandomga}, $m$ as in~\eqref{nostraelle}), then the restriction of $\det \cE(\cL)$ to 
$\pi^{-1}(D_0)$ is ample; see Proposition~\ref{prp:fiberample}. 
\begin{prp}\label{prp:potentialstab}
Assume that Hypothesis~\ref{hyp:biellittica}  and Equation~\eqref{nostraelle} hold and
 that  $m d$ is an odd number greater than $8$ $(d$  as in~\eqref{omomandomga} and $m$ as in~\eqref{nostraelle}$)$.
Let $D_0\in E^{\vee}$ and suppose that $D_0$ is not an inflection divisor. Let $\cF$ be a subsheaf of
 $\cE(\cL)_{|\pi^{-1}(D_0)}$ such that its rank $r(\cF)$ $($with respect to   the restriction of $\det \cE(\cL)$, see~\eqref{eccorango}$)$ satisfies
\begin{equation}
r(\cF)\in\{1,2,3\}.
\end{equation}
Then
\begin{equation}
\mu(\cF)<\mu\left(\cE(\cL)_{|\pi^{-1}(D_0)}\right).
\end{equation}
\end{prp}
(Note that,  since $\pi^{-1}(D_0)$ is reducible and nonreduced, the rank of   a pure $2$-dimensional sheaf on 
$\pi^{-1}(D_0)$  (with respect to  the restriction of 
 $\det \cE(\cL)$) is a rational number, possibly not an integer.)

Proposition~\ref{prp:potentialstab} is important because it implies the following result (\textit{i.e.}, Proposition~\ref{prp:stabcoduno}): 
Let $(X_t,L_t,\pi_t)$ be a general deformation of $(K_2(A),\det \cE(\cL),\pi)$, and let
$\cE_t$ be the extension of $\cE(\cL)$ to $X_t$ which exists by Proposition~\ref{prp:harbridge}. Then the restriction of 
$\cE_t$ to $\pi_t^{-1}(x)$ is slope stable (with respect to   the restriction of $\det \cE_t$) for $x$ outside of a finite subset. 
In turn, Proposition~\ref{prp:stabcoduno} is a key element in the proof of the unicity statement of Theorem~\ref{thm:unicita}. 
\subsection{Restriction of   $\det\cE(\cL)$ to a general singular Lagrangian fiber}
Let $\cE:=\cE(\cL)$. 
We have 
\begin{equation}\label{biobio}
c_1(\cE)=2m\mu_A(\ov{\omega}_A)-\delta(A)
\end{equation}
 by~\eqref{eccociuno}. Let $D_0\in E^{\vee}$, and suppose that $D_0$ is not an inflection divisor. 
 In order to simplify notation, we let 
\begin{equation}\label{treix}
Y:=\pi^{-1}(D_0),\quad Y_1:={\bf V}(A)_{D_0},\quad Y_2:=\Delta(A)_{D_0}.
\end{equation}
The fiber $Y=\pi^{-1}(D_0)$ is the schematic union of $Y_1$ and $Y_2$. The schematic intersection
\begin{equation}\label{interypsilon}
Y_{12}:=Y_1\cap Y_2
\end{equation}
is a Cartier divisor both on $Y_1$ and on $Y_2$, and it is supported on the curve $\Sigma$ of Definition~\ref{dfn:gamsig}. Since $\Sigma$ is the branch divisor of the double cover map $C_A^2\to C_A^{(2)}=V(A)_{D_0}$ and the ramification divisor is the diagonal, which has trivial normal bundle,  the normal bundle of $\Sigma$ in $V(A)_{D_0}$ is trivial. 
It follows that we have an exact sequence
\begin{equation}\label{sigsig}
0\lra\cO_{\Sigma}\lra \cO_{Y_{12}}\lra\cO_{\Sigma}\lra 0.
\end{equation}
 Let $\Sigma,\Gamma\in\NS(V(A)_{D_0})$ be as in Definition~\ref{dfn:gamsig} (we recall that $V(A)_{D_0}$ is the reduced scheme associated to ${\bf V}(A)_{D_0}=X_1$). Then $\{\Sigma,\Gamma\}$ is a basis of 
 $\NS(V(A)_{D_0})_{\QQ}$, and
\begin{equation}
(\Sigma\cdot\Sigma)_{V(A)_{D_0}}=(\Gamma\cdot\Gamma)_{V(A)_{D_0}}=0,\quad 
(\Sigma\cdot\Gamma)_{V(A)_{D_0}}=4.
\end{equation}
Straightforward computations give that
\begin{equation}\label{ishiguro}
\mu_A(\ov{\omega}_A)_{|V(A)_{D_0}}=\frac{d}{4}\Sigma+\frac{3d}{2}\Gamma,\quad 
\delta(A)_{|V(A)_{D_0}}=\frac{1}{2}\Sigma,
\end{equation}
 and hence (by~\eqref{biobio}) we have
\begin{equation}\label{detesuvu}
 c_1(\cE)_{|V(A)_{D_0}}=\frac{md-1}{2}\Sigma+3md\Gamma.
\end{equation}
 Recall that 
 $\Delta(A)_{D_0}\cong C_A\times\PP^1$. Let $\Lambda:=\{{\rm pt}\}\times\PP^1$.   Then  $\{\Sigma,\Lambda\}$ is a basis of 
 $\NS(\Delta(A)_{D_0})$. By Proposition~\ref{prp:shkabarnya}, we have
\begin{equation}\label{detesudel}
c_1(\cE)_{|\Delta(A)_{D_0}}=\Sigma+12md\Lambda.
\end{equation}
\begin{prp}\label{prp:fiberample}
Keep notation and assumptions as in Proposition~\ref{prp:potentialstab}. $($Here there is no need to assume that $md$ is odd, and moreover it suffices that $md>1$.$)$ Then the restriction of   $\det\cE(\cL)$ to $Y=\pi^{-1}(D_0)$ is ample. 
\end{prp}
\begin{proof}
By the equality in~\eqref{detesuvu},  the restriction of $\det\cE(\cL)$ to $V(A)_{D_0}$ is ample (recall Proposition~\ref{prp:franchin}). Since   the restriction of $\det\cE(\cL)\otimes\xi(A)$ to $V(A)_{D_0}$ is also ample (it equals $(md/2) \Sigma+3md\Gamma$), it follows that   the restriction of $\det\cE(\cL)$ to ${\bf V}(A)_{D_0}$ is ample (see Proposition~\ref{prp:ginseng}).
By~\eqref{detesudel}, the restriction of $\det\cE(\cL)$ to 
$\Delta(A)_{D_0}$ is ample as well. The proposition follows. 
\end{proof}
\subsection{Simpson (semi)stability of sheaves on projective schemes}\label{subsec:recapstab}
We recall the definition of Simpson (semi)stability of a coherent sheaf  $\cG$ of  pure dimension $d$ on a complex  projective scheme $Z$ polarized by 
$\cO_Z(1)$ (see~\cite{simpson1}). Write the Hilbert polynomial of $\cG$ as
\begin{equation}
\chi(Z,\cG(n))=\sum_{i=0}^d\alpha_i(\cG)\frac{n^i}{i!}.
\end{equation}
The rank of $\cG$ (relative to $\cO_Z(1)$) is
\begin{equation}\label{eccorango}
r(\cG):=\frac{\alpha_d(\cG)}{\alpha_d(\cO_Z)}.
\end{equation}
If $Z$ is integral, then $r(\cG)$ does not depend on the polarization and equals the classical rank, \textit{i.e.}, the dimension 
(as vector space over the field of rational functions on $Z$) of the fiber of $\cG$ at the generic point of~$Z$. The \emph{reduced Hilbert polynomial} of $\cG$ is given by
\begin{equation}
P_{\cG}:=\frac{1}{r(\cG)}\sum_{i=0}^d\alpha_i(\cG)\frac{n^i}{i!}\in\QQ[m].
\end{equation}
The sheaf $\cG$ is semistable if for every proper subsheaf $0\not=\cF\subsetneq \cG$ we have 
$P_{\cF}(n)\le P_{\cG}(n)$ for large  $n$, and it is stable if strict inequality holds (for large $n$). By definition,  
$\alpha_d(\cF)/r(\cF)=\alpha_d(\cG)/r(\cG)=\alpha_d(\cO_Z)$; hence the leading coefficients 
of $P_{\cF}$ and $P_{\cG}$ are equal. The sheaf $\cG$ is slope semistable if 
\begin{equation}
\frac{\alpha_{d-1}(\cF)}{r(\cF)}\le \frac{\alpha_{d-1}(\cG)}{r(\cG)}
\end{equation}
 for every nonzero subsheaf $0\not=\cF\subset \cG$, and it is slope stable if strict inequality holds whenever $\cF$ has strictly smaller rank. A slope (semi)stable sheaf is (semi)stable.
If $Z$ is a smooth integral variety, slope (semi)stability coincides with the classical notion.  
\subsection{Pure sheaves on  $\pi^{-1}(D_0)$ for general $D_0\in E^{\vee}$}
We describe  pure sheaves  of dimension $2$ on $Y=\pi^{-1}(D_0)$  for general $D_0\in E^{\vee}$ 
following~\cite{nagaraj-seshadri-1,inaba-sheave-on-reducible}. In what follows, $Y_1$, $Y_2$ and $Y_{12}$ are  as in~\eqref{treix} and~\eqref{interypsilon}.  For $i\in\{1,2\}$, let $\cF_i$  be a pure sheaf of dimension $2$ on $Y_i$ (\textit{i.e.}, torsion-free), and let 
\begin{equation}\label{morfismogi}
G\colon \cF_1\lra(\cF_2)_{|Y_{12}}
\end{equation}
 be a morphism of sheaves. Let 
$\rho_i^{\cF_i}\colon \cF_i\to (\cF_i)_{|Y_{12}}$ be the restriction morphism. The sheaf $\cF$ on $Y$ fitting into the exact sequence of sheaves
\begin{equation}\label{succeseffe}
0\lra \cF\lra \cF_1\oplus\cF_2\xrightarrow{G-\rho_2^{\cF_2}} (\cF_2)_{|Y_{12}}\lra 0
\end{equation}
is pure of dimension $2$. Conversely, every pure sheaf of dimension $2$ on $Y$ is isomorphic to one such $\cF$;  
see~\cite[Remark~2.5]{nagaraj-seshadri-1} and~\cite[Proposition~1.5]{inaba-sheave-on-reducible}. The description of $\cE_{|Y}$ according to the above procedure is the following.  Let $\cE_i:=\cE_{|Y_i}$  
for $i\in\{1,2\}$ (in general, $\cF_i$ need not be $\cF_{|Y_i}$), and let 
$\cE_{12}:=\cE_{|Y_{12}}$. Then 
we have the exact sequence 
\begin{equation}\label{succesesuy}
0 \lra \cE_{|Y} \lra \cE_1\oplus \cE_2\xrightarrow{\rho^{\cE_1}_1-\rho^{\cE_2}_2} \cE_{12} \lra 0.
\end{equation}
Let $\cF$ be a pure sheaf of dimension $2$ on $Y$ fitting into the exact sequence in~\eqref{succeseffe}.
An injection of sheaves $\cF\hra(\cE_{|Y})$ is described as follows. Let $\phi_i\colon\cF_i\hra \cE_i$ be injections of sheaves for $i\in\{1,2\}$ such that 
\begin{equation}\label{phicompat}
(\phi_{2|Y_{12}})\circ G=\rho_1^{\cE_1}\circ \phi_1.
\end{equation}
Then there is a well-defined injection $\phi\colon\cF\hra \cE_{|Y}$ fitting into the morphism of complexes
\begin{equation}
\xymatrix{0  \ar[r] & \cF \ar[r] \ar_{\phi}@{^{(}->}[d] & \cF_1\oplus\cF_2 \ar^-{G-\rho_2^{\cF_2}}[r] \ar_{(\phi_1,\phi_2)}@{^{(}->}[d] & (\cF_2)_{|Y_{12}} \ar[r] \ar^{\phi_{2|Y_{12}}}[d]   & 0    \\ 
0  \ar[r] & \cE_{|Y} \ar[r]  & \cE_1\oplus\cE_2 \ar^-{\rho^{\cE_1}_1-\rho^{\cE_2}_2}[r]  & \cE_{12} \ar[r]    & 0\rlap{.}    }
\end{equation}
Conversely, every injection $\phi\colon\cF\hra \cE_{|Y}$ is described as above for suitable injections $\phi_i$ such that the equality in~\eqref{phicompat} holds; see~\cite[Remark~2.10]{nagaraj-seshadri-1}. 
 
 For $(\cF_1,\cF_2,G)$ as above, let
\begin{equation}
\cF'_1:=\cF_{1|V(A)_{D_0}},\quad \cF''_1:=\ker(\cF_1\to\cF'_1) 
\end{equation}
and
\begin{equation}
r'_1:=r(\cF'_1),\quad r''_1:=r(\cF''_1),\quad r_2:=r(\cF_2).
\end{equation}
\begin{rmk}\label{rmk:lalaland}
The sheaf $\cF''_1$ is annihilated by the ideal of $V(A)_{D_0}$ in ${\bf V}(A)_{D_0}$; hence it is the push-forward of a sheaf on $V(A)_{D_0}$. Suppose that $\cF$ is a subsheaf of $\cE_Y$. Then the multiplication morphism 
\begin{equation}
\cI_{V(A)_{D_0}/{\bf V}(A)_{D_0}}\otimes \cF_1'\lra \cF_1''
\end{equation}
is generically an injection, and hence $r_1'\le r_1''$.
\end{rmk}
\begin{prp}\label{prp:rangointero}
Keeping notation as above, suppose that $md>8$ and that $\cF$ is a subsheaf of $\cE_{|Y}$. Then 
 the rank  of $\cF$ $($with respect to the restriction of $c_1(\cE))$ is an integer if and only if
$r_1'+r_1''=2r_2$. If this is the case, then $r(\cF)=r_2$.
\end{prp}
\begin{proof}
Let  $n$ be an integer. We have  the exact sequence of sheaves on $Y_1$ 
\begin{equation}\label{airbnb}
0 \lra \cF''_1(n) \lra \cF_1(n)\lra \cF'_1(n) \lra 0,
\end{equation}
and hence $\alpha_2(\cF_1)=\alpha_2(\cF'_1)+\alpha_2(\cF''_1)$. Tensoring the exact sequence in~\eqref{succeseffe} by $\cO_Y(n)$, we get that
\begin{equation*}
\alpha_2(\cF)=\alpha_2(\cF'_1)+\alpha_2(\cF''_1)+\alpha_2(\cF_2)=(r'_1+r''_1)\deg V(A)_{D_0}+r_2\deg\Delta(A)_{D_0},
\end{equation*}
where degrees are with respect to $c_1(\cE)$. Hence we have
\begin{equation}\label{rangoeffe}
r(\cF)=\frac{(r'_1+r''_1)\deg V(A)_{D_0}+r_2\deg\Delta(A)_{D_0}}{2\deg V(A)_{D_0}+\deg\Delta(A)_{D_0}}.
\end{equation}
From this, the \lq\lq if\rq\rq\ direction follows at once (no need to suppose that $md>8$ or that $\cF$ is a subsheaf of $\cE_{|Y}$). In order to prove the \lq\lq only if\rq\rq\ direction, we suppose that $r_1'+r_1''\not=2r_2$. By~\eqref{detesuvu} and~\eqref{detesudel},  we have
\begin{equation}\label{gradicomp}
\deg V(A)_{D_0}=12md(md-1),\quad \deg \Delta(A)_{D_0}=24md,
\end{equation}
and hence~\eqref{rangoeffe} gives that
\begin{equation}\label{tombola}
r(\cF)=\frac{r'_1+r''_1}{2}-\frac{r'_1+r''_1-2r_2}{2md}.
\end{equation}
Now $r'_1,r''_1,r_2$ are at most equal to $4$ because $\cF$ is a subsheaf of $\cE_{|Y}$.  Since $r_1'+r_1''\not=2r_2$ and $md>8$, it follows that 
$0<|(r'_1+r''_1-2r_2)/2md|<1/2$. This proves that the right-hand side in~\eqref{tombola} is not an integer.
\end{proof}
\subsection{Slope of subsheaves of  $\cE(\cL)_{|\pi^{-1}(D_0)}$ for general $D_0\in E^{\vee}$}
\begin{prp}\label{prp:pendeneffe}
With notation as above, suppose that $md>8$ and that $\cF$ is a subsheaf of $\cE_{|Y}$ with integer rank. Then
\begin{equation}
\frac{\alpha_1(\cF)}{r(\cF)}=\frac{\alpha_1(\cF_1')}{r_2}+\frac{\alpha_1(\cF_1'')}{r_2}+
\frac{\alpha_1(\cF_2)}{r_2}-2\deg\Sigma. 
\end{equation}
\end{prp}
\begin{proof}
Tensoring  the exact sequence of sheaves in~\eqref{succesesuy} by $\cO_Y(n)$ and recalling the exact sequences in~\eqref{airbnb},  
we get that
\begin{equation}\label{befana}
\alpha_1(\cF)=\alpha_1(\cF_1')+\alpha_1(\cF_1'')+\alpha_1(\cF_2)-\alpha_1((\cF_2)_{|Y_{12}}).
\end{equation}
The exact sequence in~\eqref{sigsig} and Riemann--Roch for $\Sigma$ give that 
\begin{equation}
\alpha_1((\cF_2)_{|Y_{12}})=2r_2\deg\Sigma,
\end{equation}
where 
\begin{equation}
\deg\Sigma=\int_{\Sigma}c_1(\cE)=12md.
\end{equation}
By Proposition~\ref{prp:rangointero}, we have $r(\cF)=r_2$ because $\cF$ has integer rank. Dividing  the equality in~\eqref{befana} by $r(\cF)$, one gets the proposition.
\end{proof}
\begin{crl}\label{crl:sedestab}
With notation as above, suppose that $md>8$ and that $\cF$ is a   subsheaf of $\cE_{|Y}$ with integer rank $r(\cF)\in\{1,2,3\}$ such that 
$\mu(\cF)\ge \mu(\cE_{|Y})$. Then $\cF_2$ is slope desemistabilizing for the sheaf $\cE_2$ on $Y_2$ $($with respect to the restriction of $c_1(\cE))$.
\end{crl}
\begin{proof}
Since $\cE_1''\cong\cE_1'\otimes\xi(A)$ (see Proposition~\ref{prp:portuense}), we have
\begin{equation}\label{primodueprimo}
\alpha_1(\cE_1'')=\alpha_1(\cE_1')+2\deg\Sigma.
\end{equation}
By Proposition~\ref{prp:pendeneffe}, for  $\cF=\cE_{|Y}$, it follows that
\begin{equation}\label{pendene}
\frac{\alpha_1(\cE_{|Y})}{4}=\frac{\alpha_1(\cE_1')}{2}+\frac{\alpha_1(\cE_2)}{4}-\frac{3}{2}\deg\Sigma.
\end{equation}
The morphism $G\colon \cF_1\to (\cF_2)_{|Y_{12}}$ (see~\eqref{morfismogi}) maps $\cF_1''$ to $(\cF''_2)_{|Y_{12}}$. By the equality
 in~\eqref{phicompat}, we get that the restriction of $\phi_1$ to $\cF_1''$ drops rank along $\Sigma$ at least by 
 $r_1''-r_2$.
(Note that $(r_1''-r_2)\ge 0$ by Remark~\ref{rmk:lalaland} and Proposition~\ref{prp:rangointero} because the rank of $\cF$ is integral). Thus there exist a subsheaf 
 $\ov{\cF_1''}\subset\cE_1''$ and an exact sequence
\begin{equation}
  0\lra \cF_1''\lra\ov{\cF_1''}\lra\cQ\lra 0, 
\end{equation}
 where $\cQ$ is supported on $\Sigma$ and 
\begin{equation}
\alpha_1(\cQ)\ge(r_1''-r_2)\deg\Sigma.
\end{equation}
 By Proposition~\ref{prp:portuense}, the  vector bundle $\cE_1''$  is  slope stable, and hence
$4\alpha_1(\ov{\cF_1''})<r_1''\alpha_1(\cE_1'')$. Thus
\begin{multline}
\alpha_1(\cF_1'')=\alpha_1(\ov{\cF_1''})-c\deg\Sigma\le \alpha_1(\ov{\cF_1''})-(r_1''-r_2)\deg\Sigma<\\
<\frac{r_1''}{4}\alpha_1(\cE_1'')-(r_1''-r_2)\deg\Sigma=\frac{r_1''}{4}\alpha_1(\cE_1')+\frac{2r_2-r_1''}{2}\deg\Sigma.
\end{multline}
(The last equality follows from~\eqref{primodueprimo}.) By Proposition~\ref{prp:pendeneffe} and the slope stability of $\cE_1'$, we get that
\begin{multline}
\frac{\alpha_1(\cF)}{r(\cF)}<\frac{\alpha_1(\cF_1')}{r_2}+\frac{r_1''}{4r_2}\alpha_1(\cE_1')+
\frac{\alpha_1(\cF_2)}{r_2}-\frac{3}{2}\deg\Sigma-\frac{r_1''-r_2}{2r_2}\deg\Sigma=\\
=\frac{\alpha_1(\cE_1')}{2}+\frac{\alpha_1(\cF_2)}{r_2}-\frac{3}{2}\deg\Sigma-\frac{r_1''-r_2}{2r_2}\deg\Sigma\le\\
\le \frac{\alpha_1(\cE_1')}{2}+\frac{\alpha_1(\cF_2)}{r_2}-\frac{3}{2}\deg\Sigma.
\end{multline}
By~\eqref{pendene}, it follows that $4\alpha_1(\cF_2)>r_2\alpha_1(\cE_2)$.
\end{proof}
\subsection{Elementary modifications  of  $\cE_1$}
The main result of the present subsection is motivated by  Corollary~\ref{crl:sedestab}.   Let $\cH_2\subset\cE_2$ be a nonzero subsheaf  such that $\mu(\cH_2)>\mu(\cE_2)$ (here the slope is as sheaves on $Y_2$, with respect to the restriction of $c_1(\cE)$) and the quotient $\cE_2/\cH_2$ is torsion-free. By Proposition~\ref{prp:shkabarnya}, the quotient $\cE_2/\cH_2$ is actually locally free. 
Let 
\begin{equation}
\cR(\cH_2):=(\cE_2/\cH_2)_{|Y_{12}}.
\end{equation}
Let $\cR(\cH_2)'$ be the restriction of $\cR(\cH_2)$ to $\Sigma=(Y_{12})_{\rm{red}}$. By the exact sequence in~\eqref{sigsig}, we have an exact sequence
\begin{equation}
0\lra \cR'(\cH_2)\lra \cR(\cH_2) \lra \cR'(\cH_2)\lra 0.
\end{equation}
Let $\cG_1(\cH_2)$ be the sheaf on $Y_1$ fitting into the exact sequence
\begin{equation}
0\lra \cG_1(\cH_2)\lra  \cE_1\overset{\gamma}{\lra}\cR(\cH_2)\lra 0,
\end{equation}
where $\gamma$ is the composition of the restriction morphism $\cE_1\to (\cE_1)_{|Y_{12}}$, the natural isomorphism 
$(\cE_1)_{|Y_{12}}\overset{\lowsim}{\to} (\cE_2)_{|Y_{12}}$ and the quotient map $(\cE_2)_{|Y_{12}}\to \cR(\cH_2)$.

Let $\cG_1'(\cH_2)$ be the restriction of $\cG_1(\cH_2)$ to $V(A)_{D_0}=(Y_1)_{\rm{red}}$, and let $\cG_1''(\cH_2)$ be the kernel of the restriction morphism 
$\cG_1(\cH_2)\to \cG_1'(\cH_2)$. Thus we have a commutative diagram with exact rows and columns
\begin{equation}\label{cruciverba}
\xymatrix{   & 0  \ar[d] & 0  \ar[d] & 0 \ar[d]   &    \\ 
0  \ar[r] & \cG_1''(\cH_2) \ar[r] \ar[d] & \cG_1(\cH_2) \ar[r] \ar[d] & \cG_1'(\cH_2)\ar[r] \ar[d]   & 0    \\ 
0  \ar[r] & \cE_1'' \ar[r]  \ar[d] & \cE_1 \ar[r]  \ar[d] & \cE_1' \ar[r]   \ar[d]  & 0    \\ 
0  \ar[r] & \cR'(\cH_2)  \ar[r]   \ar[d]  & \cR(\cH_2)  \ar[r]  \ar[d]   & \cR'(\cH_2) \ar[r]  \ar[d]  & 0  \\
   & 0   & 0   & 0\rlap{.}    &     }
\end{equation}
For later use, we record the slopes of $\cG_1'(\cH_2)$ and $\cG_1''(\cH_2)$. By Proposition~\ref{prp:shkabarnya}, we have $\alpha_1(\cR'(\cH_2))=(12-3r_2)md$. By the left and right vertical exact sequences in~\eqref{cruciverba}, it follows that 
\begin{equation}\label{numerprimo}
\alpha_1(\cG_1'(\cH_2))  =  \alpha_1(\cE_1')-(4-r_2)\deg\Sigma=\alpha_1(\cE_1')-12(4-r_2) md
\end{equation}
and 
\begin{equation}\label{numersecondo}
\alpha_1(\cG_1''(\cH_2))=\alpha_1(\cE_1'')-(4-r_2)\deg\Sigma=\alpha_1(\cE_1')-12(4-r_2) md.
\end{equation}
\begin{prp}\label{prp:trasfsemist}
Keep notation and hypotheses as above. Then the sheaves $\cG_1'(\cH_2)$ and $\cG_1''(\cH_2)$  are slope semistable with respect to any ample line bundle.
\end{prp}
\begin{proof}
We give the   proof for $\cG_1'(\cH_2)$. (The proof for $\cG_1''(\cH_2)$ is similar.)
Let $h$ be the  class of an ample line bundle on $V(A)_{D_0}$, and suppose that $\cG_1'(\cH_2)$ is not $h$ slope semistable, \textit{i.e.}, there  exists an exact sequence of sheaves
\begin{equation}
0\lra \cU\lra \cG_1'(\cH_2)\lra \cW\lra 0
\end{equation}
such that  $\cU$ is nonzero of rank $r_{\cU}<4$  and
\begin{equation}
(4 c_1(\cU)-r_{\cU} c_1(\cG_1'(\cH_2)))\cdot h>0.
\end{equation}
We claim that
\begin{equation}
(4 c_1(\cU)-r_{\cU} c_1(\cG_1'(\cH_2)))\cdot \cl(\Sigma)\le 0.
\end{equation}
In fact, the restriction of $c_1(\cG_1'(\cH_2))$ to $\Sigma$ is (strictly) slope semistable because the exact sequence defining $\cG_1'(\cH_2)$ (the right-hand side short exact column in~\eqref{cruciverba}) gives that we have an exact sequence
\begin{equation}
0\lra \cR'(\cH_2) \lra  \cG_1'(\cH_2)_{|\Sigma}\lra (\cH_2)_{|\Sigma} 
\lra 0
\end{equation}
and  $\mu(\cR'(\cH_2))=3md=\mu((\cH_2)_{|\Sigma} )$. Since $\cl(\Sigma)$ is on the boundary of the ample cone of $V(A)_{D_0}$, we get that there exists an ample class $h_1$ in the convex cone spanned by $h$ and 
$\cl(\Sigma)$ such that 
\begin{equation}
(4 c_1(\cU)-r_{\cU} c_1(\cG_1'(\cH_2)))\cdot h_1=0.
\end{equation}
Arguing as in the proof of the wall-and-chamber decomposition for slope stability of sheaves on (smooth projective) surfaces, one gets that there exists an ample class $h_0$ (in the convex cone spanned by $h$ and  $\cl(\Sigma)$) such that $\cG_1'(\cH_2)$ is strictly $h_0$ slope semistable, \textit{i.e.}, $h_0$ slope semistable but not slope stable. We claim that the hypotheses of Lemma~\ref{lmm:disczero} are satisfied by the vector bundle $\cG_1'(\cH_2)$, and  hence by  Lemma~\ref{lmm:disczero},    $\cG_1'(\cH_2)$ is not strictly $h_0$ slope semistable, giving a contradiction. 
First a straightforward computation gives that $\Delta(\cG_1'(\cH_2))=0$. It remains to show that $c_1(\cG_1'(\cH_2))$ is not divisible by $2$. The exact sequence defining $\cG_1'(\cH_2)$ gives the equality
\begin{equation*}
c_1(\cG_1'(\cH_2))=c_1(\cE_1')-(4-r_r)\cl(\Sigma)=(\deg\cV)\Gamma+2(md-1)c_1(\cO(1))-(4-r_r)\cl(\Sigma),
\end{equation*}
where $\cV$ is the stable rank $4$ vector bundle on $C_A$ of Proposition~\ref{prp:monella} and $\cO(1)$ is the line bundle on  $V(A)_{D_0}$ considered in Section~\ref{subsec:geomvert}. Choose $p_0\in C_A$, and let $\Pi\subset V(A)_{D_0}\cong C_A^{(2)}$ be the section of the $\PP^1$ bundle $g_A\colon C_A^{(2)}\to C_A$ given by 
\begin{equation}
\Pi:=\{(p)+(p_0)\mid p\in C_A\}.
\end{equation}
Then $\Pi\cdot \Sigma=2$, and hence  
\begin{equation}
\int_{\Pi} c_1(\cG_1'(\cH_2))\equiv \deg\cV \pmod{2}.
\end{equation}
Since $\cV$ is a stable vector bundle of rank $4$ on the elliptic curve $C_A$, it has odd degree. This proves  that $c_1(\cG_1'(\cH_2))$ is not divisible by $2$. 
\end{proof}
\subsection{Proof of Proposition~\ref{prp:potentialstab}}
Suppose that $\mu(\cF)\ge\mu(\cE_{|Y})$. By Corollary~\ref{crl:sedestab}, it follows that $\mu(\cF_2)>\mu(\cE_2)$, where the slope is as sheaves on $Y_2$, with respect to the restriction of $c_1(\cE)$. Hence there exists a chain of subsheaves $\cF_2\subset\cH_2\subset\cE_2$ such that $\cH_2/\cF_2$ is torsion and $\cE_2/\cH_2$ is torsion-free. By the key relation in~\eqref{phicompat},  the morphism  $\phi_1$ defines a morphism 
$\ov{\phi}_1\colon \cF_1\to \cG_1(\cH_2)$
and hence morphisms
\begin{equation}
\cF_1'\xrightarrow{\ov{\phi}_1'} \cG_1'(\cH_2),\quad \cF_1''\xrightarrow{\ov{\phi}_1''} \cG_1''(\cH_2).
\end{equation}
By Proposition~\ref{prp:trasfsemist} and the equalities in~\eqref{numerprimo} and~\eqref{numersecondo}, 
it follows that
\begin{eqnarray}
\alpha_1(\cF_1') & \le & \frac{r_1'}{4}\alpha_1(\cG_1'(\cH_2)=\frac{r_1'}{4}\alpha_1(\cE_1')-3r_1'(4-r_2)md,
\label{bologna}\\
\alpha_1(\cF_1'') & \le & \frac{r_1'}{4}\alpha_1(\cG_1''(\cH_2)=\frac{r_1''}{4}\alpha_1(\cE_1'')-3r_1''(4-r_2)md.
\label{erbe}
\end{eqnarray}
We compare the expressions for $\alpha_1(\cF)/r(\cF)$ and $\alpha_1(\cE)/4$ which appear in 
Proposition~\ref{prp:pendeneffe} and~\eqref{pendene}. First we have
\begin{equation}
\frac{\alpha_1(\cF_2)}{r_2}\le \frac{\alpha_1(\cH_2)}{r_2}=\frac{\alpha_1(\cE_2)}{4}+\frac{12-3r_2}{r_2}md.
\end{equation}
By the inequalities in~\eqref{bologna} and~\eqref{erbe} and the equality in~\eqref{primodueprimo}, we get that
\begin{equation}
\frac{\alpha_1(\cF)}{r(\cF)}\le \frac{\alpha_1(\cE_1')}{2}+\frac{\alpha_1(\cE_2)}{4}-\frac{3}{2}\deg\Sigma
-\left(33-6r_2-\frac{12+6r_1''}{r_2}\right)md.
\end{equation}
One checks easily that the right-hand side is strictly smaller than the right-hand side in~\eqref{pendene}. This gives a contradiction.
\qed

\section{Proof of Theorem~\ref{thm:unicita}}\label{sec:dimprinc}
\subsection{Ampleness of $\det\cE(\cL)$}
Below is the main result of the present subsection.
\begin{prp}\label{prp:amponkum}
Let $\ov{a}$ be a positive integer. If $d\gg 0$ $($the lower bound depends on $\ov{a})$, the following holds. 
Let $A$ be an abelian surface such that $\NS(A)=\ZZ\ov{\omega}_A\oplus\ZZ\gamma_A$, where $\ov{\omega}_A$ is ample, 
$\gamma_A$ is the Poincar\`e dual  
of   a $($bona fide$)$ elliptic curve $C_A$ and
\begin{equation}\label{testakoala}
\ov{\omega}_A\cdot\ov{\omega}_A=4\ov{a},\quad \ov{\omega}_A\cdot\gamma_A=d.
\end{equation}
Then for any positive integer $m$, the class $2m\mu_A(\ov{\omega}_A)-\delta(A)$ is ample on $K_2(A)$. 
\end{prp}
In order to prove Proposition~\ref{prp:amponkum}, we make explicit the general results 
of~\cite{yoshi-pos-kumm} and~\cite{knu-mon-lel} in the case of $\Kum_2(A)$. 

If $M$ is a HK manifold, let $\cC(M)\subset H^{1,1}_{\RR}(M)$ be the cone of $x\in H^{1,1}_{\RR}(M)$ such that $q_M(x)>0$, and let $\cC^{+}(M)$ the connected component of $\cC(M)$ containing K\"ahler classes.
We recall 
that the  K\"ahler cone $\cK(M)\subset \cC^{+}(M)$ is a connected component (open chamber) of the complement of the wall divisors, given by $w^{\bot}\cap\cC^{+}(M)$ for certain classes $w\in \NS(M)$ of negative square. 
\begin{prp}\label{prp:murikumm}
Suppose that  $w\in \NS(K_2(A))$ is primitive and $w^{\bot}\cap\cC^{+}(K_2(A))$ is a wall divisor. Then
$q(w)=-6$ and $\divisore(w)\in\{2,3,6\}$. 
\end{prp}
\begin{proof}
First we recall that the  Mukai pairing $\la,\ra$ on $H^{2*}(A;\ZZ)$ is defined by
\begin{equation*}
\la (r,\ell,s),(r',\ell',s')\ra:=-rs'-r's+\int_S\ell\cup\ell'.
\end{equation*}
If $v^2 \ge 6$, then the cohomology of the Albanese fiber of a moduli  space of stable sheaves on $A$ with assigned  primitive Mukai vector 
 $v\in H^{2*}(A;\ZZ)$) is identified with the orthogonal $v^{\bot}\subset H^{2*}(A;\ZZ)$.
 In our case, \textit{i.e.}, $K_2(A)$, the   Mukai vector is $v:=(1,0,-3)$. 
 
According to \cite[Theorem~2.9]{knu-mon-lel} a class 
 $w\in v^{\bot}$ defines a  wall divisor (\textit{i.e.}, $w^{\bot}\cap\cC^{+}(K_2(A))$ is a wall divisor) if and only if $\la w,w\ra $ is negative and  the saturated sublattice of  
 $H^{2*}(A;\ZZ)$ generated by $v$ and 
 $w$  contains an $s\in H^{2*}(A;\ZZ)$ such that 
 \begin{equation}\label{mantegazza}
0\le \la s,s\ra < \la s,v\ra \le \frac{1}{2}(\la v,v\ra+\la s,s\ra)=3+ \frac{1}{2}\la s,s\ra.
\end{equation}
Note that we may assume that $s$ is primitive.  
 Clearly 
\begin{equation}
w=\frac{1}{n}\left(\la s,v\ra v-6s\right) 
\end{equation}
for a certain integer $n$ which is a multiple of $\gcd(\la v,s\ra, 6)$. Thus we get that
\begin{equation}
\la w,w\ra=-\frac{6}{n^2}\left(\la s,v\ra^2 -6\la s,s\ra\right),\quad 
\text{$\frac{6}{n}$ divides $\divisore(w)$}.
\end{equation}
The inequalities in~\eqref{mantegazza}
give that $\la s,s\ra<6$, and hence $\la s,s\ra\in\{0,2, 4\}$ because  the Mukai pairing is even, and also give that $\la s,v\ra$ belongs to a finite list. Explicitly, either $\la s,s\ra=0$ and   
$\la s,v\ra \in\{1,2,3\}$,  or $\la s,s\ra=2$ and  $\la s,v\ra\in\{3,4\}$, or $\la s,s\ra=4$ and  $\la s,v\ra=5$.
A case-by-case analysis gives that 
$n=\gcd(\la v,s\ra, 6)$ in all cases, and  the proposition follows (recall that $\divisore(x)$ is a divisor of $6$ for every $x\in  H^2(K_2(A);\ZZ)$). 
\end{proof}
We will use the following elementary result.
\begin{lmm}[\textit{cf.} \protect{\cite[Lemma 4.3]{ogfascimod}}]\label{lmm:nocamere}
Let $(\Lambda,q)$ be a nondegenerate rank $2$ lattice which represents $0$, and hence $\disc(\Lambda)=-d_0^2$, where $d_0$ is a strictly positive integer. Let $\alpha\in \Lambda$ be primitive isotropic, and complete it to a basis $\{\alpha,\beta\}$ such that $q(\beta)\ge 0$. If $\gamma\in\Lambda$ has strictly negative square $($i.e., $q(\gamma)<0)$, then
\begin{equation}\label{menoenne}
 q(\gamma)\le -\frac{2d_0}{1+q(\beta)}.
\end{equation}
\end{lmm}
\begin{proof}[Proof of Proposition~\ref{prp:murikumm}]
Recall that by~\eqref{eccociuno}, we have
\begin{equation*}
c_1(\cE(\cL))=2\mu_A(m\ov{\omega}_A)-\delta(A).
\end{equation*}
Since  $\ov{\omega}_A$ is ample, the class $\mu_A(\ov{\omega}_A)$ is in the closure of the ample cone. Suppose that 
$2m\mu_A(\ov{\omega}_A)-\delta(A)$ is not ample. By Proposition~\ref{prp:murikumm}, there exist 
$\beta\in \NS(A)$, $x\in\ZZ$ such that 
\begin{equation}\label{cinquantaquattro}
q(\beta^2-x\delta(A))=-6\le \beta^2-6x^2<0 
\end{equation}
and either
\begin{equation}\label{sulmuro}
q(\mu_A(\beta)-x\delta(A), \mu_A(\ov{\omega}_A))=q(\mu_A(\beta)-x\delta(A), 2m\mu_A(\ov{\omega}_A)-\delta(A))=0
\end{equation}
or
\begin{equation}\label{sepmuro}
q(\mu_A(\beta)-x\delta(A), \mu_A(\ov{\omega}_A))>0,\quad 
q(\mu_A(\beta)-x\delta(A), 2m\mu_A(\ov{\omega}_A)-\delta(A))\le 0.
\end{equation}
In other words, either both $2m\mu_A(\ov{\omega}_A)$ and $2m\mu_A(\ov{\omega}_A)-\delta(A)$ lie on a wall, or 
there exists  a wall separating $2m\mu_A(\ov{\omega}_A)$ from $2m\mu_A(\ov{\omega}_A)-\delta(A)$. 

Assume that~\eqref{sulmuro} holds. Then $(\beta,\ov{\omega}_A)=0$, and hence $\beta^2\le 0$ by the Hodge index theorem. The last inequality is strict because if it were an equality, then we would have $\beta=0$ and hence also $x=0$. 
By~\eqref{cinquantaquattro}, we get that
$-6\le \beta^2<0$. By Lemma~\ref{lmm:nocamere}, this is impossible if $d>12\ov{a}+3$. 

Next assume that~\eqref{sepmuro} holds. We rewrite~\eqref{cinquantaquattro} as
\begin{equation}\label{fringuelli}
\beta^2<6x^2\le \beta^2+6.
\end{equation}
By~\eqref{sepmuro}, we  have 
\begin{equation}\label{carcaricola}
0<m(\beta,\ov{\omega}_A)\le 3x.
\end{equation}
 In particular, $x$ is positive. 
Since  by the Hodge index theorem, we have $(\ov{\omega}_A^2)\cdot(\beta^2)\le(\beta,\ov{\omega}_A)^2$, 
it follows that
\begin{equation}\label{}
m^2\cdot 4\ov{a}\beta^2\le m^2\cdot (\beta,\ov{\omega}_A)^2\le 9x^2\le \frac{3}{2}\beta^2+9.
\end{equation}
It follows that $\beta^2\le 18/5$. Feeding this into the second inequality in~\eqref{fringuelli}, we get that $x=1$. 
This together with~\eqref{carcaricola} gives that $(\beta,\ov{\omega}_A)\in\{1,2,3\}$. (Notice that if $m>3$, we are done.) 
The class 
$(4\ov{a}\beta-(\beta,\ov{\omega}_A)\ov{\omega}_A)$ is orthogonal to $\ov{\omega}_A$; hence 
\begin{equation}\label{}
-12\ov{a}\le (4a\beta-(\beta,\ov{\omega}_A)\ov{\omega}_A)^2\le 0,
\end{equation}
where the first inequality follows from $0\le\beta^2$ (see~\eqref{fringuelli}) and  $(\beta,\ov{\omega}_A)\in\{1,2,3\}$. By Lemma~\ref{lmm:nocamere}, if $d>24\ov{a}^2+6\ov{a}$, this forces the square to vanish and hence $\beta$ to be a multiple of $\ov{\omega}_A$. This is absurd by~\eqref{carcaricola}.
\end{proof}
\begin{rmk}
Arguing  geometrically, one can give a lower bound on $d$ in Proposition~\ref{prp:amponkum}  which is much stronger than the bound one gets by following the proof, at least if $a\not=2$. In fact, one can show that if $\ov{a}=1$, then $d\ge 2$ suffices for ampleness of  $2m\mu_A(\ov{\omega}_A)-\delta(A)$, and that  if $\ov{a}\ge 3$, then $d\ge 4$ suffices.
We sketch the proof. 
Let $L$ be a line bundle on $A$ such that $c_1(L)=2m\ov{\omega}_A$. 
Then $L$ is ample on $A$, and   it is globally generated. We claim    
 that the map  $A\to |L|^{\vee}\cong\PP^{8\ov{a}m^2-1}$ is an embedding and that the image   is cut out by quadrics. If $\ov{a}=1$, this follows from the results  in~\cite{barth-1-2}, and if $\ov{a}\ge 3$, the statement holds because  $\cL$ satisfies property $(N_1)$ by the main result
  in~\cite{kuro-lozo-ab-surfcs} or~\cite{ito-rmk-ab-surfcs}.   Thus we have $A\subset |L|^{\vee}\cong\PP^{8\ov{a}m^2-1}$. 
Since $A$  is cut out by quadrics and it contains no lines,  every length $3$ subscheme of $A$ spans a plane in 
$\PP^{8\ov{a}m^2-1}$, and hence we have a regular map
\begin{equation}
\begin{matrix}
A^{[3]} & \overset{\varphi}{\lra} & \Gr(3,\CC^{8\ov{a}m^2}) \\
[Z] & \longmapsto & \la Z\ra.
\end{matrix}
\end{equation}
Next, one proves that  $\varphi$ has finite fibers, and hence the pull-back 
$\varphi^{*}(\cP)$ of the Pl\"ucker (ample) line bundle on $\Gr(3,\CC^{8\ov{a}m^2})$ is ample on  $A^{[3]}$. At this point, we are done because the restriction of $\varphi^{*}(\cP)$ to $K_2(A)$ is $2\mu_A(m\ov{\omega}_A)-\delta(A)$.
\end{rmk}
\subsection{Existence}\label{subsec:esiste}
If $e\equiv -6\pmod{16}$, we write $e=16\ov{a}-6$, and if $e\equiv -6\pmod{144}$,  we write $e=144\ov{a}-6$ (note that in both cases, $\ov{a}$ is a positive integer). Let $B$ be an abelian surface such that 
Hypothesis~\ref{hyp:biellittica} holds. We adopt the notation introduced in Proposition~\ref{prp:amponkum}. In particular, 
$A=B/\la \epsilon\ra$, and we have the ample class $\ov{\omega}_A$ on $A$ with $\ov{\omega}_A^2=4\ov{a}$. We assume throughout that the positive integer $d=\ov{\omega}_A\cdot\gamma_A$ (see~\eqref{omomandomga} and~\eqref{omegamdi}) is odd and large (the  lower bound on $d$ depends only on $\ov{a}$).
Now we choose the line bundle $\cL$ on $X$:
\begin{enumerate}
\item\label{case1}
If $e\equiv -6\pmod{16}$, we let $c_1(\cL)=\nu^{*}(\mu_B(\ov{\omega}_B))$,
\item\label{case2}
if $e\equiv -6\pmod{144}$, we let $c_1(\cL)=\nu^{*}(\mu_B(3\ov{\omega}_B))$.
\end{enumerate}
Set  $\cE_0:=\cE(\cL)$ and $h_0:=c_1(\cE_0)$. By~\eqref{eccociuno}, we have
\begin{enumerate}
\item
$h_0=2\mu_A(\ov{\omega}_A)-\delta(A)$ if $e\equiv -6\pmod{16}$ and
\item
$h_0=6\mu_A(\ov{\omega}_A)-\delta(A)$   if $e\equiv -6\pmod{144}$.
\end{enumerate}
Hence $q(h_0)=e$ in both cases, and $\divisore(h_0)=2$ in the first case, while $\divisore(h_0)=6$ in the second case. Moreover. $h_0$ is ample by Proposition~\ref{prp:amponkum} (recall that  $d\gg 0$; this is where we need $d$ to be large). 
This proves that $h_0$ is a polarization of the type appearing in Theorem~\ref{thm:unicita}.

Let 
$\cM\to T$ be a (sufficiently small) representative of $\Def(K_2(A),h_0)$.
 By Proposition~\ref{prp:harbridge}, the vector bundle $\cE_0$ has an extension to a vector bundle $\cE_t$ on $M_t$ for every $t\in T$ (the extension is unique by \textit{loc.~cit.}).  
\begin{prp}\label{prp:stabdefo}
Keeping notation as above, the vector bundle  $\cE_t$ is $h_t$ slope stable for general $t\in T$, where $h_t$ is the polarization $($deformation of $h_0)$ on $M_t$.  
\end{prp}
\begin{proof}
Let $N_d\subset T$ be the Noether--Lefschetz divisor parametrizing $M_t$ such that $\mu_A(\gamma_A)$ remains of type $(1,1)$. If  $t\in N_d$,  
then we have a Lagrangian fibration $\pi_t\colon M_t\to\PP^2$ extending the Lagrangian fibration 
$\pi_A\colon K_2(A)\to |\cO_E(3(0))|$. By Proposition~\ref{prp:stabonlagr} and the openness of stability, the restriction of 
$\cE_t$ to a general fiber of $\pi_t$ is slope stable. Now assume in addition that  $t$ is  very general in $N_d$, so that the N\'eron--Severi group of $M_t$ has rank $2$ and hence is obtained via parallel transport from the subgroup of $\NS(K_2(A))$ generated by $(2\mu_A(\ov{\omega}_A)-\delta(A)),\mu_A(\gamma_A)$ in case~\eqref{case1}, respectively  
$6\mu_A(\ov{\omega}_A)-\delta(A),\mu_A(\gamma_A)$ in case~\eqref{case2}. Applying
 Lemma~\ref{lmm:nocamere}, we get that there is a single $72$-chamber in $\NS(M_t)$, and hence $h_t$ is $72$-suitable. It follows that, with the notation introduced in~\cite[Definition~3.3]{ogfascimod}, we have
\begin{equation}
a(\cE_t):=\frac{r(\cE_t)^2 d(\cE_t)}{4c_{M_t}}=72.
\end{equation}
(Here $c_{M_t}$ is the normalized Fujiki constant of $M_t$, \textit{i.e.}, $3$, and $d(\cE_t)=54$  by~\eqref{primocinquantaquattro}.)
Hence $\cE_t$ is $h_t$ slope stable by~\cite[Proposition~3.6]{ogfascimod}. 
\end{proof}
Since  $\Kum_{e}^2$ and  $\Kum_{e}^6$ are irreducible, the existence of a slope stable vector bundle $\cF$ on $M$  such that the equalities in~\eqref{ele} hold, for  a general point  $[(M,h)]\in\Kum_{e}^2$ (respectively,  $[(M,h)]\in\Kum_{e}^6$), follows at once from 
 Proposition~\ref{prp:stabdefo}. 

\subsection{Stability of the restriction of $\cE_t$ to Lagrangian fibers}\label{subsec:cinqueuno}
Keep notation as in Section~\ref{subsec:esiste}. In particular,  $N_d\subset T$ is  the Noether--Lefschetz divisor appearing in the proof of Proposition~\ref{prp:stabdefo}. Let $t\in N_d$ be a general point, let $(M_t,h_t)$ be the corresponding deformation of $(K_2(A),h_0)$, and let $\pi_t\colon M_t\to\PP^2$ be the Lagrangian fibration which is a deformation of $\pi_0=\pi_A\colon K_2(A)\to |\cO_E(3(0))|$. Lastly, let $\cE_t$ be the extension to $M_t$ of the vector bundle $\cE_0=\cE(\cL)$ on $K_2(A)$.   The main result of the present subsection is the following.
\begin{prp}\label{prp:stabcoduno}
Keep notation as above, and suppose that $md$ is odd and greater than $8$. 
Let $t\in N_d$ be a general point. There exists a finite subset $B_d(t)\subset\PP^2$ such that 
for $x\in(\PP^2\setminus B_d(t))$,  the restriction of $\cE_t$ to the Lagrangian fiber $\pi_t^{-1}(x)$ is slope stable with respect to the restriction of\, $h_t$. 
\end{prp}
First we deal with smooth Lagrangian fibers and then  
with singular fibers. 
\begin{prp}\label{prp:stabsuliscie}
With notation and hypotheses as above, let $t\in N_d$ be a general point. The restriction of $\cE_t$ to a smooth  Lagrangian fiber 
$\pi_t^{-1}(x)$ is slope stable with respect to the restriction of\, $h_t$, except possibly for a finite set of $x\in\PP^2$.
\end{prp}
\begin{proof}
By Proposition~\ref{prp:stabonlagr} and the openness of slope stability, the restriction of $\cE_t$ to a general smooth  Lagrangian fiber is 
slope stable with respect to the restriction of $h_t$. By  Propositions~\ref{prp:stabonlagr} and~\ref{prp:sempfibre}, the restriction of $\cE_0$ to 
a  Lagrangian fiber 
$\pi_0^{-1}(x)$ is simple, except possibly for a finite set of $x\in\PP^2$. By the  openness of simpleness, it follows that the restriction of $\cE_t$ to a   Lagrangian fiber 
$\pi_t^{-1}(x)$ is simple except possibly for a finite set of $x\in\PP^2$. 
By the argument given in the proof of~\cite[Proposition~5.14]{og-rigidi-su-k3n}, it follows that 
the restriction of $\cE_t$ to a  smooth  Lagrangian fiber is simple semi-homogeneous, except possibly for a finite set of $x\in\PP^2$.  By~\cite[Proposition~6.13]{muksemi}, it follows that the restriction of $\cE_t$ to any  smooth  Lagrangian fiber is 
slope semistable with respect to the restriction of $h_t$, and hence it is slope stable by Corollary~\ref{crl:sempstab} (see the proof of 
Proposition~\ref{prp:stabonlagr}).
\end{proof}
\begin{prp}\label{prp:discrimino}
With notation and hypotheses as above, let $t\in N_d$ be a general point. Then, except possibly for a finite set of $x\in\PP^2$,  the Lagrangian fiber $\pi_t^{-1}(x)$ is integral.
\end{prp}
\begin{proof}
For $t\in N_d$, let  $\Disc(\pi_t)\subset\PP^2$ be  the discriminant curve of $\pi_t\colon M_t\to\PP^2$. If 
$x\in(\PP^2\setminus\Disc(\pi_t))$, then  $\pi_t^{-1}(x)$ is integral by the definition of a Lagrangian fibration. It remains to deal 
with  $\pi_t^{-1}(x)$ for $x\in\Disc(\pi_t)$. 

If $t\in N_d$ is general, then $\Disc(\pi_t)$ is irreducible. In fact,  by~\cite[Theorem~2]{wieneck2}, there is single deformation class of Lagrangian fibrations of HK fourfolds of Kummer type, and among such Lagrangian fibrations, there are the Beauville--Mukai systems. The discriminant curve of a general 
Beauville--Mukai system is the Severi variety (a curve in the present case) parametrizing singular divisors in the complete linear system of a (general) polarization with elementary divisors $(1,3)$ on an abelian surface. Such a Severi variety 
 is a plane irreducible curve of degree $18$ by~\cite[Proposition~5.3]{lange-sernesi-ab-surf}. 

We also claim   that 
$\pi_t^{-1}(\Disc(\pi_t))$ is irreducible. It suffices to prove the  claim for very general $t\in N_D$. Hence we may assume that the Picard number $\rho(M_t)$ equals $2$. It follows that every  irreducible component of $\pi_t^{-1}(\Disc(\pi_t))$ has cohomology class a (rational) multiple of 
$\pi_t^{*}c_1(\cO_{\PP^2}(1))$ and that 
every curve contained in a Lagrangian fiber has zero intersection number with $\pi_t^{*}c_1(\cO_{\PP^2}(1))$. Since every Lagrangian fiber is connected (because  a general Lagrangian fiber is integral), it follows that  $\pi_t^{-1}(\Disc(\pi_t))$ is irreducible.

For general $x\in\Disc(\pi_t)$, let 
\begin{equation}
[\pi_t^{-1}(x)]=\sum_{i=1}^{\ell}c_i S_i
\end{equation}
be the cycle associated to the fiber $\pi_t^{-1}(x)$. Since $\pi_t^{-1}(\Disc(\pi_t))$ is irreducible, we have $c_i=c_j$ and $\deg S_i=\deg S_j$ for any $i,j\in\{1,\ldots,\ell\}$. Now let $t\to 0$, so that the Lagrangian HK variety $M_t$ specializes to $M_0=K_2(A)$ and $\pi_t$ specializes to $\pi_0=\pi_A$. Then the cycle $[\pi_t^{-1}(x)]$ specializes to $2V(A)_{D_0}+\Delta(A)_{D_0}$, where $D_0\in E^{\vee}$ is any noninflectionary divisor; see Proposition~\ref{prp:duevu}. It follows that if $\pi_t^{-1}(x)$ is not integral, then either $\deg V(A)_{D_0}=\deg \Delta(A)_{D_0}=\deg S_i$ or
$2\deg V(A)_{D_0}=\deg \Delta(A)_{D_0}=\deg S_i$. Neither of the two equalities holds 
(the degrees of 
$V(A)_{D_0}$, $\Delta(A)_{D_0}$ are given in~\eqref{gradicomp}), and hence  $\pi_t^{-1}(\Disc(\pi_t))$ is integral. 
\end{proof}
\begin{crl}\label{crl:discrimino}
With notation as above, suppose that $md$ is odd and that $md>8$.
Let $t\in N_d$ be a general point. Then for a general $x\in \Disc(\pi_t)$, the restriction of $\cE_t$ to $\pi_t^{-1}(x)$ is slope stable with respect to the restriction of\, $h_t$.
\end{crl}
\begin{proof}
  Suppose the contrary. By Proposition~\ref{prp:discrimino},
  $\pi_t^{-1}(x)$ is integral, and hence there exists a slope destabilizing subsheaf 
$\cF_t\subset\cE_t$ with $r(\cF_t)\in\{1,2,3\}$. Take the limit for $t\to 0$, so that $M_t$ specializes to $M_0=K_2(A)$ and $\pi_t$ specializes to 
$\pi_0=\pi_A$.
Then we get  that for all $x\in \Disc(\pi_A)$, the sheaf $\cE_0=\cE(\cL)$ has  a slope destabilizing subsheaf $\cF_0\subset\cE_0$ with 
$r(\cF_0)\in\{1,2,3\}$. This contradicts Proposition~\ref{prp:potentialstab}.

\end{proof}
\subsection{A monodromy computation}
Keep notation as in Section~\ref{subsec:cinqueuno}. 
\begin{prp}\label{prp:mondue}
Let $t\in N_d$ be a general point. Let $x_0\in\PP^2$ be a regular value of $\pi_t$, and let $S:=\pi_t^{-1}(x_0)$.  Then $0$ is the only element of 
$S^{\vee}[2]$ invariant under the monodromy action of $\pi_1(\PP^2\setminus \Disc(\pi_t),x_0)$. 
\end{prp}
\begin{proof}
Let $x_0=(y_1)+(y_2)+(y_3)$ (we follow the notation in~\eqref{fibrapisecco}). Then 
\begin{equation}
S^{\vee}[2]\cong Z_2(C_A)^{\vee}[2]\cong C_A^{\vee}[2]\oplus C_A^{\vee}[2].
\end{equation}
The monodromy action is generated by the following two involutions: 
\begin{equation}
(a,b)\longmapsto (b,a),\quad (a,b)\longmapsto (a,-a-b).
\end{equation}
The proposition follows easily from this.
\end{proof}
\begin{crl}\label{crl:mondue}
Let $t\in N_d$ be a general point. Let $x_0\in\PP^2$ be a regular value of $\pi_t$, and let $S:=\pi_t^{-1}(x_0)$.  Then  the only coset of the subgroup 
$S^{\vee}[2]<S^{\vee}[4]$ invariant under the monodromy action of $\pi_1(\PP^2\setminus \Disc(\pi_t),x_0)$ is $S^{\vee}[2]$ itself. 
\end{crl}
\begin{proof}
The homomorphism $S^{\vee}[4]\to S^{\vee}[2]$ defined by multiplication by $2$ identifies the quotient 
$S^{\vee}[4]/S^{\vee}[2]$ with  $S^{\vee}[2]$; hence the corollary follows from Proposition~\ref{prp:mondue}.

\end{proof}
\subsection{Unicity}
We finish the proof of Theorem~\ref{thm:unicita} by showing that if $[(M,h)]\in\Kum_{e}^2$ or  $[(M,h)]\in\Kum_{e}^6$ are general, then there exists a unique (up to isomorphism)  slope stable vector bundle $\cF$ on $M$  such that the equalities in~\eqref{ele} hold. 
The key steps of the proof  are those of the proofs of the analogous  results in~\cite{ogfascimod,og-rigidi-su-k3n}.

Let $Z^i_e\to \cP^i_e$ (for $i\in\{2,6\}$) be a complete family of polarized HK fourfolds of Kummer type with polarization of square $e$ and divisibility $i$. 
Since $\Kum_{e}^2$ and  $\Kum_{e}^6$ are irreducible, we may assume that $ \cP^i_e$ is irreducible. By results of Gieseker and Maruyama, there exists a relative moduli space 
\begin{equation}\label{tipofinito}
\rho^i_e\colon \cM^i_e\lra \cP^i_e
\end{equation}
 with fiber over a point corresponding to $(M,h)$ isomorphic to the moduli space of 
 slope stable vector bundles $\cF$ on $M$  such that the equalities in~\eqref{ele} hold. 
Moreover, the morphism $\rho^i_e$ is  of finite type by Maruyama~\cite{maruyama-boundedness}. 

For $d\gg 0$, let $N_d(\cP^i_e)^0$ be a Noether--Lefschetz divisor containing the divisor $N_d$ defined in the proof of Proposition~\ref{prp:stabdefo} (we do not know whether the Noether--Lefschetz divisors defined by imposing that the N\'eron--Severi lattice contains the rank $2$ lattice in the proof of Proposition~\ref{prp:stabdefo} are irreducible).
Since the union of the
$N_d(\cP^i_e)^0$ is Zariski dense in 
$\cP^i_e$, it suffices to prove that if $M$ is HK corresponding to a general point of $N_d(\cP^i_e)^0$, then, up to isomorphism, there 
is a unique slope stable vector bundle $\cF$ on $M$  such that the equalities in~\eqref{ele} hold. In turn, to prove the latter statement, it suffices to prove unicity for the HK varieties  corresponding to a very general $t\in N_d$. 

Let $t\in N_d$ be a general point. Let $\cE_t$ be the vector bundle on $M_t$ obtained by deformation of $\cE(\cL)$ on $K_2(A)$ as in Section~\ref{subsec:esiste}, and let $\cF_t$ be a slope stable vector bundle on $M_t$ such that  the equalities in~\eqref{ele} hold. We must  prove that $\cF_t\cong\cE_t$. 
\begin{clm}\label{clm:genisom}
The restrictions of $\cF_t$ and $\cE_t$ to a general Lagrangian fiber are isomorphic. 
\end{clm}
\begin{proof}
As was noticed in the proof of Proposition~\ref{prp:stabdefo}, there is a single $72$-chamber in $\NS(M_t)$, and hence $h_t$ is $72$-suitable. Since  $a(\cE_t)=72$, it follows that the restriction of $\cF_t$ to a generic Lagrangian fiber of $\pi_t$ is slope semistable (with respect to the restriction of $h_t$)
 by~\cite[Proposition~3.6]{ogfascimod}. By Proposition~\ref{prp:delzero}, it follows that the restriction of $\cF_t$ to a generic Lagrangian fiber of $\pi_t$ is slope stable (see the equation in~\eqref{ciunosulagr}). Now let $S_x=\pi_t^{-1}(x)$ be a general Lagrangian fiber of $\pi_t$. The restrictions of $\cF_t$ and $\cE_t$ 
 to $S_x$ are slope stable vector bundles  with the same rank and  first Chern class. Moreover, we have   
\begin{equation}
\int_{S_x}\Delta\left(\cF_{t|S_x}\right)=\int_{S_x}\Delta\left(\cE_{t|S_x}\right)=
\int_{M_t}\Delta\left(\cF_{t|S_x}\right)\cdot \pi_t^{*}(c_1(\cO_{\PP^2}(1))^2)=54 q_{M_t}(c_1(\cO_{\PP^2}(1))^2)=0.
\end{equation}
It follows that the restrictions of $\cF_t$ and $\cE_t$ 
 to $S_x$ are (simple) semi-homogenous  (see~\cite[Proposition~A.2]{ogfascimod}).
 By~\cite[Theorem~7.11, Corollary~7.12]{muksemi}, (see the equation in~\eqref{ciunosulagr}), we get that
\begin{equation*}
V_x:=\left\{[\xi]\in S_x^{\vee} \mid \cF_{t|S_x}\cong (\cE_{t|S_x})\otimes\xi\right\}
\end{equation*}
is not empty, and  it is a coset of $S_x^{\vee}[2]$ in $S_x^{\vee}[4]$. By Corollary~\ref{crl:mondue}, we get that $V_x=S_x^{\vee}[2]$, and hence
$\cF_{t|S_x}\cong \cE_{t|S_x}$. 
\end{proof}
Let $C_d(t)\subset \PP^2$ be the union of 
the subset $B_d(t)\subset \PP^2$  of Proposition~\ref{prp:stabcoduno} and the set of $x\in \PP^2$ such that $S_x:=\pi_t^{-1}(x)$ is not integral. Then $C_d(t)$ is finite by Propositions~\ref{prp:stabcoduno} and~\ref{prp:discrimino}.
Let $x_0\in(\PP^2\setminus C_d(t))$, and let $\Gamma\subset(\PP^2\setminus C_d(t))$ be a general smooth curve containing $x_0$. 
Then $Y:=\pi_t^{-1}(\Gamma)$ is a  projective threefold, and the restriction of $\pi_t$ to $Y$ defines a surjection $Y\to\Gamma$ whose fibers are integral. 

Since $\Gamma$ is general, the restrictions $\cF_{t|S_x}$ and $\cE_{t|S_x}$ are isomorphic for general $x\in\Gamma$  by  Claim~\ref{clm:genisom}.  Moreover, $\cE_{t|S_x}$ is slope stable for all $x\in\Gamma$. By~\cite[Lemma~7.5]{ogfascimod},   it follows that for all $x\in\Gamma$ 
the restrictions $\cF_{t|S_x}$ and $\cE_{t|S_x}$ are isomorphic.\footnote{Among the hypotheses of~\cite[Lemma~7.5]{ogfascimod}, there is the requirement that $Y$ be smooth, but it is not necessary. In fact, one replaces the cup product in item~(i) of the statement of Lemma~7.5 by the degrees of the corresponding intersection products. The latter products make sense because one is intersecting Cartier divisors and Chern classes of vector bundles.} This shows that for all  
$x_0\in(\PP^2\setminus C_d(t))$ the restrictions $\cF_{t|S_{x_0}}$ and $\cE_{t|S_{x_0}}$ are isomorphic. Since $\cE_{t|S_x}$ is simple for all $x\in(\PP^2\setminus C_d(t))$ and $\det\cF\cong\det\cE$, it follows that $\cF$ and $\cE$ are isomorphic away from $C_d(t)$. By Hartogs' theorem, we get that $\cF$ and $\cE$ are isomorphic.
\qed

\appendix

\section{Examples of semi-homogeneous vector bundles}\label{sec:mezzomo}
\subsection{Basics} 
In the present section, \lq\lq abelian variety means abelian variety\rq\rq, and similarly for \lq\lq elliptic curve\rq\rq. Let $A$ be an abelian  variety.  For $a\in A$, we let $T_a\colon A\to A$ be the translation by $a$.   If $\cL$ is a line bundle on $A$, we let
\begin{equation}\label{eccophi}
\begin{matrix}
A & \overset{\varphi_{\cL}}{\lra} & A^{\vee} \\
a & \longmapsto & [T_a^{*}(\cL)\otimes\cL^{-1}].
\end{matrix}
\end{equation}
We recall that  a vector bundle  $\cF$ on $A$ is  \emph{semi-homogeneous} if, 
for every $a\in A$, there exists a line bundle $\cL_{a}$ such that $T_a^{*}\cF\cong \cF\otimes\cL_a$. We recall the following result of Oda (\cite[Theorem 1.2]{odavbell}) and Mukai (\cite[Theorem 5.8]{muksemi}) .
\begin{thm}[Oda, Mukai]\label{thm:critsimsem}
A vector bundle $\cE$ on an abelian variety $A$ is simple and semi-homogeneous if and only if it is isomorphic to $f_{*}\cL$, where $f\colon \wt{A}\to A$ is an isogeny  and $\cL$  is a line bundle  on $\wt{A}$  such that $\ker(f)\cap\ker(\varphi_{\cL})=\{0\}$.
\end{thm}
\subsection{Semi-homogeneous vector bundles on powers of elliptic curves}\label{subsec:pushsimp}
Let $C$ be an elliptic curve, let  $\wt{C}$ be an elliptic curve, and let $f\colon\wt{C}\to C$ be an isogeny. For $i\in\{1,\ldots,n+1\}$,   let  $L_i$ be a line bundle on $\wt{C}$ of degree $d_0$ (independent of $i$).  Let  $Z_n(C)\subset C^{n+1}$, $Z_n(\wt{C})\subset \wt{C}^{n+1}$ be as in~\eqref{zedenc}, and for  $i\in\{1,\ldots,n+1\}$,  let $\wt{p}_i\colon Z_n(\wt{C})\to  \wt{C}$ be the $\supth{i}$ projection. Lastly,  let 
\begin{equation*}
\begin{matrix}
 Z_n(\wt{C}) & \overset{\psi^n}{\lra} &  Z_n(C) \\
 (z_1,\ldots,z_{n+1}) & \longmapsto &   (f(z_1),\ldots,f(z_{n+1})). 
\end{matrix}
\end{equation*}
\begin{prp}\label{prp:semplice}
Keeping notation as above, suppose that $\deg f$  is coprime to $(n+1)\cdot d_0$. Then  
$\cG:=\psi^n_{*}\left(\bigotimes_{i=1}^{n+1}\wt{p}_i^{*}L_i\right)$ is a simple semi-homogeneous vector bundle on $Z_n(C)$ of rank $(\deg f)^n$. 
\end{prp}
\begin{proof}
The sheaf  $\cG$ is locally free of rank equal to $\deg(\psi^n)$ because $\psi^n$ is finite. Since 
$\ker(\psi^n)\cong (\ker f)^n$,
it follows that  $\cG$ has rank $(\deg f)^n$. In order to simplify  notation,  we let
\begin{equation*}
\xi:=\bigotimes\limits_{i=1}^{n+1}\wt{p}_i^{*}L_i.
\end{equation*}
Let us prove  that
\begin{equation}\label{zeppola}
\int\limits_{ Z_n(\wt{C})}c_1(\xi)^n=(n+1)\cdot d_0^n. 
\end{equation}
In fact, let $\{\alpha,\beta\}$ be a $\ZZ$-basis of $H^1(\wt{C};\ZZ)$, and for $i\in\{1,\ldots,n\}$ (we stop at $n$), let 
\begin{equation*}
x_i:=\wt{p}_i^{*}(\alpha)_{|Z_n(\wt{C})}, \quad y_i:=\wt{p}_i^{*}(\beta)_{|Z_n(\wt{C})}. 
\end{equation*}
 Then $\{x_1,\ldots,x_n,y_1,\ldots,y_n\}$ is a basis of $H^1(Z_n(\wt{C});\ZZ)$, and we have
\begin{equation}
c_1(\xi)=d_0\cdot\sum\limits_{i=1}^n x_i\wedge(y_1+\cdots+y_{i-1}+2y_i+y_{i+1}+\cdots+y_n).
\end{equation}
Equation~\eqref{zeppola} follows from the above formula and a straightforward computation. By~\eqref{zeppola}, the order of $\ker(\varphi_{\xi})$ is equal to $(n+1)^2\cdot d_0^{2n}$ and hence  is coprime to $(\deg f)^n$  by our hypothesis. 
Since  the order of $\ker(\psi^n)$ is $(\deg f)^n$, it follows  that 
\begin{equation}\label{nonparlano}
\ker(\psi^n)\cap \ker(\varphi_{\xi})=\{0\}.
\end{equation}
(Here $\varphi_{\xi}$ is as in~\eqref{eccophi}.)
 Hence  $\cG$ is  simple semi-homogeneous  by Theorem~\ref{thm:critsimsem}. 
\end{proof}

\section{Properly semistable vector bundles with vanishing discriminant on abelian surfaces}\label{sec:stabonabsurf}
In the present section, we  extend  Proposition 4.4 and Corollary 4.5   of~\cite{ogfascimod} to  arbitrary polarized abelian surfaces. 
\begin{prp}\label{prp:delzero}
Let $(A,\theta)$ be a $(1,e)$ polarized abelian surface. Let  $\cF$ be a slope  semistable  vector bundle on $A$ such that $\Delta(\cF)=0$ and  $c_1(\cF)=a\theta$. Then 
there exist coprime integers $r_0,b_0$, with $r_0$ positive, and a positive integer $m$ such that
\begin{equation}\label{carnevale}
r(\cF)=m\frac{r_0^2}{g},\quad a=m \frac{r_0}{g}b_0,
\end{equation}
where $g:=\gcd\{r_0,e\}$. If  $\cF$ is slope strictly semistable, \textit{i.e.}, slope semistable but not slope stable, then $m\ge 2$.  
\end{prp}
\begin{proof}
First suppose that $\cF$ is slope stable. Then $\cF$ is  semi-homogeneous because $\Delta(\cF)=0$; see \cite[Proposition~A.2]{ogfascimod}. Since it is also simple, the thesis holds by \cite[Proposition~A.3]{ogfascimod}. 

Now suppose that $\cF$ is strictly slope semistable. Then there exist an integer $m\ge 2$ and a (slope stability) Jordan--H\"older filtration  of $\cF$
\begin{equation}\label{jordhold}
0=\cG_0\subsetneq\cG_1\subsetneq\dots\subsetneq\cG_m=\cF
\end{equation}
such that for each $i\in\{1,\ldots,m\}$, the sheaf  $\cG_{i}/\cG_{i-1}$ is locally free, $c_1(\cG_{i}/\cG_{i-1})=a_i\theta$ and
$\Delta(\cG_{i}/\cG_{i-1})=0$;
 see   the proof of Proposition~4.4 in~\cite{ogfascimod}. Arguing as above we get that $\cG_{i}/\cG_{i-1}$  is 
 a simple semi-homogeneous vector bundle. Let 
 $i\in\{1,\ldots,m\}$. By \cite[Proposition~A.3]{ogfascimod}, 
 there exist coprime integers $r_i,b_i$, with $r_i$ positive,  such that 
\begin{equation*}
r(\cG_{i}/\cG_{i-1})=\frac{r_i^2}{g_i},\quad a_i=\frac{r_i}{g_i}b_i,
\end{equation*}
where $g_i:=\gcd\{r_i,e\}$. Hence the slope 
 of $\cG_{i}/\cG_{i-1}$ is equal to $(b_i/r_i)\theta^2$. Since  the filtration in~\eqref{jordhold} is the (slope stability) Jordan--H\"older filtration  of $\cF$, it follows that 
$(b_1/r_1)=\dots=(b_m/r_m)$. 
Since $r_i,b_i$ are coprime, we get  that $r_1=\dots=r_m$ and $b_1=\dots=b_m$. 
Hence there exist coprime integers $r_0,b_0$, with $r_0$ positive,  such that 
\begin{equation*}
r(\cG_{i}/\cG_{i-1})=\frac{r_0^2}{g_0},\quad c_1(\cG_{i}/\cG_{i-1})=\frac{r_0}{g_0}b_0, 
\end{equation*}
where $g_0:=\gcd\{r_0,e\}$. This proves the thesis because $m\ge 2$.
\end{proof}
\begin{crl}\label{crl:sempstab}
Let $(A,\theta)$ be a $(1,e)$ polarized abelian surface. Let  $\cF$ be a slope semistable vector bundle on $A$ such that 
$\Delta(\cF)=0$. Suppose  that there exist  coprime integers $s_0,c_0$, with $s_0$ positive, such that   $\gcd\{s_0,e\}=1$ and 
\begin{equation}
  r(\cF)=s_0^2,\quad c_1(\cF)=s_0 c_0\theta. 
\end{equation}
  Then $\cF$ is  slope stable.
\end{crl}
\begin{proof}
 By Proposition~\ref{prp:delzero},  there exist coprime integers $r_0,b_0$, with $r_0$ positive, and a positive integer $m$ such that 
\begin{equation}\label{essecizero}
s_0^2=m\frac{r_0^2}{g},\quad s_0 c_0=m \frac{r_0}{g}b_0,
\end{equation}
where $g:=\gcd\{r_0,e\}$. 
It follows that $r_0c_0=s_0b_0$. Since  $r_0,b_0$ are coprime and  $s_0,c_0$ also are, it follows that $r_0= s_0$ and 
$b_0= c_0$. This gives that $m=g$. Since $r_0= s_0$ and  by hypothesis $s_0$ is coprime to $e$, we have $g=1$, and thus $m=1$. By  Proposition~\ref{prp:delzero},  it follows that $\cF$ is   slope stable. 
\end{proof}
%


\newcommand{\etalchar}[1]{$^{#1}$}

 \end{document}